\pdfoutput=1
\documentclass[letterpaper,11pt]{article}
\usepackage[margin=1in]{geometry}
\usepackage{amsmath,amssymb,amsfonts,amsthm,mathtools}
\usepackage{microtype,xcolor,enumitem,bbm,booktabs,tabularx}
\usepackage{needspace}
\usepackage{graphicx}
\usepackage{tikz}
\usetikzlibrary{arrows.meta,patterns.meta,patterns,calc,intersections,decorations.pathreplacing,through,arrows,decorations.markings,positioning}
\mathtoolsset{showonlyrefs}
\usepackage[CJKbookmarks=true,
            bookmarksnumbered=true,
            bookmarksopen=true,
            colorlinks=true,
            citecolor=red,
            linkcolor=blue,
            anchorcolor=red,
            urlcolor=blue,
            pdfusetitle]{hyperref}

\newtheorem*{informaltheorem}{Theorem}
\newtheorem{theorem}{Theorem}[section]
\newtheorem{lemma}[theorem]{Lemma}
\newtheorem{proposition}[theorem]{Proposition}
\newtheorem{corollary}[theorem]{Corollary}
\newtheorem{remark}[theorem]{Remark}
\newtheorem{definition}[theorem]{Definition}

\providecommand{\keywords}[1]{\small\textbf{\textit{Keywords---}} #1}

\tikzset{global scale/.style={scale=#1,every node/.append style={scale=#1}}}

\newcommand{\maG}{\mathcal{G}}
\newcommand{\maH}{\mathcal{H}}

\newcommand{\maK}{\mathcal{K}}

\newcommand{\Bern}{\mathrm{Bern}}
\def\E{\mathbb{E}}
\def\P{\mathbb{P}}

\newcommand{\TV}{\mathsf{TV}}
\newcommand{\KL}{\mathsf{KL}}

\newcommand{\tr}{\mathrm{Tr}}

\newcommand{\indc}[1]{\mathbf{1}_{\{#1\}}}

\newcommand{\ER}{Erd\H{o}s--R\'enyi }

\def\var{\mathrm{Var}}

\definecolor{red1}{RGB}{255,0,0}
\definecolor{red2}{RGB}{220,47,2}
\definecolor{red3}{RGB}{189,58,85}

\newcommand{\Sph}{\mathbb{S}}

\newcommand{\op}{\mathrm{op}}
\newcommand{\sgn}{\operatorname{sgn}}

\newcommand{\HS}{\mathrm{HS}}
\newcommand{\post}{\mathrm{post}}
\newcommand{\1}{\mathbf{1}}
\newcommand{\esssup}{\mathop{\text{\normalfont ess sup}}}
\newcommand{\Part}{\mathfrak{P}}
\newcommand{\dd}{\mathrm{d}}

\title{Detecting Random Geometric Graphs with General Kernel}
\author{Dong\ Huang\thanks{
D.\ Huang is with the Department of Statistics and Data Science, Tsinghua
University. Email: \href{mailto:hd23@mails.tsinghua.edu.cn}{\color{black}{\nolinkurl{hd23@mails.tsinghua.edu.cn}}}.}\ ~and
Shuangping\ Li\thanks{S. Li is with the Department of Statistics and Data Science, Yale University. Email: \href{mailto:shuangping.li@yale.edu}{\color{black}{\nolinkurl{shuangping.li@yale.edu}}}.}}
\date{}
\hypersetup{
  pdfsubject={Detection in high-dimensional latent geometric graphs},
  pdfkeywords={random geometric graph, hypothesis testing, nonsmooth kernel,
  spherical harmonics, posterior overlap, cavity method}
}

\begin{document}
\maketitle

\begin{abstract}
We study testing a random geometric graph with a fixed kernel against an \ER graph with the same edge density. Under the null hypothesis, the graph has $n$ vertices and independent edges with probability $p$. Under the alternative hypothesis, latent points $X_1,\ldots,X_n$ are independent and uniform on $\mathbb{S}^{d-1}$, and edges are conditionally independent with probabilities $K(\langle X_i,X_j\rangle)$, where $p=\E[K(\langle X_1,X_2\rangle)]$. The goal is to determine a critical dimension $d_\star = d_\star(n,K)$: the total variation distance between the two graph laws tends to one for $d\ll d_\star$ and to zero for $d\gg d_\star$.

For dense graphs, \cite{bubeck2016testing} established $d_\star\asymp n^3$ for hard threshold kernels, whereas \cite{mao2026random} proved $d_\star\asymp n^{3/4}$ for smooth kernels with $K'(0)\ne0$. We connect these results by characterizing the detection thresholds for kernels that need not be differentiable at zero. We also determine the thresholds for smooth kernels whose first nonzero derivative at zero has finite order greater than one. Throughout these classes, signed triangle or quadrilateral counts attain the optimal detection threshold. For certain nondifferentiable kernels, quadrilateral counts are optimal while triangle counts are suboptimal. 
\end{abstract}

{\keywords{Random geometric graph; hypothesis testing; spherical harmonics; subgraph counts; posterior expansion}}

\tableofcontents

\section{Introduction}\label{sec:introduction}

\subsection{Model and testing problem}


Network data arise in various fields, including social networks~\cite{wasserman1994social}, biological networks~\cite{barabasi2004network}, and communication networks~\cite{faloutsos1999power}. The \ER model $\maG(n,p)$, in which each pair of vertices is connected independently with probability $p$, is a basic model with a well-developed mathematical theory~\cite{bollobas2001random}. However, its independence assumption does not capture dependence arising from unobserved vertex attributes in real network data. Latent position models~\cite{hoff2002latent} address this limitation by assigning each vertex an unobserved feature vector and allowing connection probabilities to depend on the features of its endpoints. Specifically, edges are independent conditional on these features.

Indeed, inner products provide a natural way to describe interactions between
latent feature vectors~\cite{young2007random}. In this
paper, we consider a spherical \emph{random geometric graph}
(RGG). Let \(X_1,\ldots,X_n\) be sampled independently from the uniform probability
measure \(\sigma_d\) on the unit sphere \(\Sph^{d-1}\).
Conditional on these latent points, the \(d\)-dimensional RGG on the vertex set \([n]=\{1,\ldots,n\}\) is generated
by connecting each pair of distinct vertices \(i,j\) independently with
probability \(K(\langle X_i,X_j\rangle)\).
We consider dense graphs with a fixed measurable kernel
\(K:[-1,1]\to[c,1-c]\) for some constant \(0<c<1/2\).

In the \ER model, all edges are mutually independent.
A natural question is whether the latent geometry in an RGG can be
detected by comparing it with an \ER graph having the same edge density.
As \(d\) grows, independent uniform points on \(\Sph^{d-1}\) become nearly
orthogonal, and the RGG becomes more similar to an \ER graph. This motivates studying how large \(d\) must be for the geometric structure to become statistically undetectable. 
Let \(P_{n,d}=P_{n,d}(K)\) denote the law of the RGG, and let \(Q_{n,d}\)
denote the law of the \ER graph on \([n]\) with edge probability
\(p=\E[K(\langle X_1,X_2\rangle)]\).
The two models therefore have the same marginal distribution on each edge. Let $\TV(P,Q)$ denote the total variation distance between two distributions $P$ and $Q$.
We seek a critical dimension \(d_\star=d_\star(n,K)\) such that, as
\(n,d\to\infty\),
\begin{equation}
 \TV(P_{n,d},Q_{n,d})\longrightarrow
 \begin{cases}
  1,&d\ll d_\star,\\
  0,&d\gg d_\star.
 \end{cases}
 \label{eq:testing-threshold}
\end{equation}

Despite its simple construction, this testing problem is already
challenging. The two models agree not only on individual edges, but
on the joint distribution of the edge indicators along any fixed
forest. Intuitively, one can integrate out the latent points at
the leaves one at a time: rotational symmetry makes the probability
of each incident edge equal to $p$, regardless of the position
of its neighbor. Thus, ordinary tree subgraph counts have the same
expectation under the two models. A triangle is the smallest
configuration that can reveal geometric dependence. For the
classical hard RGG (i.e. $K_d(t)=\indc{t\ge\tau_{d,p}}$ for some threshold $\tau_{d,p}$), where nearby latent points are connected,
edges between $i$ and $j$ and between $j$ and $k$ place $X_i$
and $X_k$ near a common point $X_j$. This makes them more likely
to be close to each other, increasing the chance of the remaining
edge between $i$ and $k$. 
This observation motivates the signed triangle count, which centers
each edge indicator at $p$ before multiplying around a triangle
and summing over all triples. 
Indeed, the signed triangle count has played a central role in previous
studies of spherical RGGs and their
variants~\cite{bubeck2016testing,liu2023phase,liu2022testing,
bangachev2024fourier,mao2026random,bok2026detection,
hunter2026distinguishability}. However, proving that this statistic
is optimal is challenging, since one must rule out every other
test beyond its detection threshold. For the simple hard threshold kernel
$K_d(t)=\indc{t\ge\tau_{d,p}}$, where $\tau_{d,p}$ is chosen to
give edge density $p$, this question has required a sustained
sequence of advances. Its optimality was established for fixed $p\in(0,1)$
in~\cite{bubeck2016testing}.
For $p=\Theta(1/n)$, \cite{liu2022testing} proved
indistinguishability for $d\gg\log^{36}n$, matching the
triangle bound up to polylogarithmic factors.
More recently, \cite{du2026resolution} established optimality
for vanishing $p\gtrsim n^{-2/3}/\log n$.
See also~\cite{hunter2026distinguishability} for another recent result.
Very recently, \cite{du2026sharp} established the sharp detection
threshold for $p\ge n^{-1+\delta}$ with $\delta>0$.

For more general kernels, \cite{mao2026random} highlights the
importance of the behavior of $K$ near zero. Since
$\langle X_i,X_j\rangle=O_{\P}(d^{-1/2})$, this is the region
that determines most connection probabilities as $d$ grows.
For fixed smooth dense kernels with $K'(0)\neq0$,
\cite{mao2026random} proves that $d_\star\asymp n^{3/4}$,
again attained by the signed triangle count, and further
conjectures that this statistic governs the detection threshold
for general kernels. This threshold differs substantially
from $d_\star\asymp n^3$ for dense hard
RGGs~\cite{bubeck2016testing}. What connects these two regimes,
and what happens when $K$ is not differentiable at zero
or satisfies $K'(0)=0$? 
These results motivate us to study the relationship between the detection threshold and the local regularity of $K$.

To address these questions, we study fixed kernel families that allow
nondifferentiability at zero or a vanishing first derivative. For some
of these kernels, contributions from the even and odd parts cancel
in the signed triangle mean, so the triangle signal alone need not
determine the optimal detection threshold. In contrast, the signed
quadrilateral mean can be written as an integral of a square and has
no such cancellation. We show that quadrilaterals attain the optimal
threshold in regimes where cancellation makes the triangle signal too
weak. Our main results can be summarized as follows.

\begin{informaltheorem}[Informal]
For the kernel families considered in this paper, the optimal detection
threshold is determined jointly by the signed triangle and quadrilateral
counts.
\end{informaltheorem}

Related phenomena have also been observed in different settings.
For RGGs on the torus with the $L_\infty$ metric,
\cite{bangachev2024linfty} shows that signed quadrilateral counts
outperform triangles and are optimal, up to logarithmic factors,
among polynomial tests of low degree. For certain stochastic block
models, \cite{bangachev2025graph} shows that signed star and
quadrilateral counts suffice whenever a polynomial test of
constant degree succeeds.

\subsection{Main results}

Throughout the paper, we assume that $K:[-1,1]\to[c,1-c]$ is fixed and measurable with constant
$0<c<1/2$. Write
\begin{equation}
 T_d=\langle X_1,X_2\rangle,\quad p_d=\E[K(T_d)],\quad
 s_d=\sqrt{p_d(1-p_d)},\quad
 \kappa_d(t)=\frac{K(t)-p_d}{s_d}.
 \label{eq:model-normalization}
\end{equation}
Let $T_{\kappa_d}$ be the integral operator on $L^2(\sigma_d)$
with kernel $\kappa_d(\langle x,y\rangle)$, and define
$\tau_{\ell,d}=\tr(T_{\kappa_d}^{\ell})$ for $\ell\ge2$.
For $\ell\ge3$, these traces are the centered moments of simple cycles
of length $\ell$. Between the hard threshold case with $d_\star\asymp n^3$
\cite{bubeck2016testing} and the smooth case with $K'(0)\ne0$ and
$d_\star\asymp n^{3/4}$ \cite{mao2026random}, there are kernels that are
continuous but not differentiable at zero, such as
$K(x)=c+(1-2c)\sqrt{|x|}$. To study such kernels, we first introduce
the cusp class by separating the even and odd parts of $K$. Indeed, we will show that this class of kernels unifies the two previously known results.
For $p_0\in(0,1)$ and $0<r_0<1$, write
\begin{align}\label{eq:def-KeKo}
 K_e(r)&=\frac{K(r)+K(-r)}2-p_0,\quad
 K_o(r)=\frac{K(r)-K(-r)}2,\quad 0<r<r_0.
\end{align}

\begin{definition}[The cusp class]\label{gen:def:local}\label{def:cusp}
A fixed dense function $K$ belongs to $\maK_{\mathrm{cusp}}$ if, for some
$p_0,r_0$ as in~\eqref{eq:def-KeKo}, it satisfies one of the following
conditions. All expansions are as $r\downarrow0$.
\begin{itemize}
\item \emph{Pure even.} For some $0<\alpha\le1$ and $a_e\ne0$,
$K_e(r)=a_er^\alpha+o(r^\alpha)$ and $K_o\equiv0$.
\item \emph{Pure odd.} For some $0<\beta\le1$ and $a_o\ne0$,
$K_e\equiv0$ and $K_o(r)=a_or^\beta+o(r^\beta)$.
\item \emph{Mixed.} For some $0<\alpha\le1$, $\beta>0$, and $a_e,a_o\ne0$,
$K_e(r)=a_er^\alpha+o(r^\alpha)$ and $K_o(r)=a_or^\beta+o(r^\beta)$.
\item \emph{Jump.} For some $\varepsilon\in\{-1,1\}$ and $c_1>0$,
$\varepsilon\{K(r)-K(-r)\}\ge c_1$ for every $0<r<r_0$.
\end{itemize}
\end{definition}

\begin{figure}[htbp]
\centering
\includegraphics[width=\textwidth]{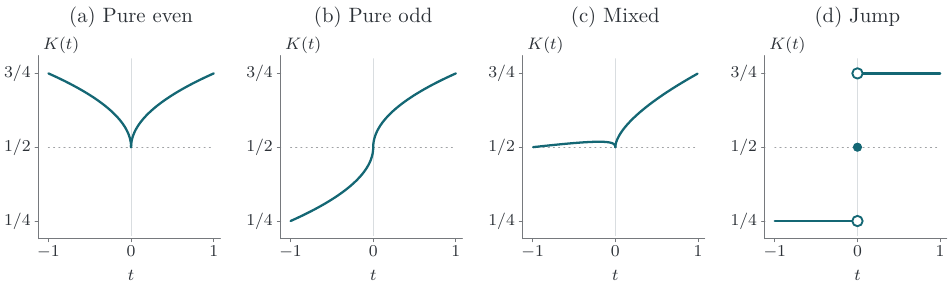}
\caption{Examples of the four cases in Definition~\ref{def:cusp}:
\textup{(a)} $K(t)=\frac12+\frac14|t|^{1/2}$;
\textup{(b)} $K(t)=\frac12+\frac14\operatorname{sgn}(t)|t|^{1/2}$;
\textup{(c)} $K(t)=\frac12+\frac18|t|^{1/2}
+\frac18\operatorname{sgn}(t)|t|^{3/4}$; and
\textup{(d)} $K(t)=\frac12+\frac14\operatorname{sgn}(t)$.
The dotted line marks $p_0=1/2$, and $\operatorname{sgn}(0)=0$.
In \textup{(d)}, the open circles mark the one-sided limits.}
\label{fig:cusp-kernel-examples}
\end{figure}


The first three cases describe continuous cusps, whereas the last describes discontinuous behavior at zero (see Figure~\ref{fig:cusp-kernel-examples} for examples).
The exponents $\alpha$ and $\beta$ describe the leading local
behavior of the even and odd parts of $K$ without requiring
derivatives at zero. Smaller exponents correspond to sharper cusps.
The pure odd case connects the two regimes discussed above:
the profile $a_o\operatorname{sgn}(t)|t|^\beta$ is continuous but
not differentiable at zero for $0<\beta<1$, gives a jump at
$\beta=0$, and becomes a nonzero linear term at $\beta=1$,
matching the local behavior of a smooth kernel with $K'(0)\ne0$. 
These exponents also determine the scale of the local signal.
Since $\langle X_1,X_2\rangle=O_{\P}(d^{-1/2})$, the corresponding
even and odd power terms have typical sizes $d^{-\alpha/2}$ and
$d^{-\beta/2}$. Their rates of decay, together with how they
contribute to cycle counts, determine the detection threshold.
This concentration near zero also motivates imposing the
conditions only for $r\in(0,r_0)$: outside $(-r_0,r_0)$, $K$ may
be any extension satisfying the dense bound.
The following theorem characterizes the detection threshold
for the cusp class.

\begin{theorem}[Detection for the cusp class]\label{thm:cusp}
Fix $K\in\maK_{\mathrm{cusp}}$ and let
$\Lambda_{n,d}=n^3\tau_{3,d}^2+n^4\tau_{4,d}^2$.
As $n,d\to\infty$,
\begin{equation}
 \TV(P_{n,d},Q_{n,d})\longrightarrow
 \begin{cases}
  1,&\Lambda_{n,d}\to\infty,\\
  0,&\Lambda_{n,d}\to0.
 \end{cases}
 \label{eq:cusp-TV}
\end{equation}
\end{theorem}

In view of Theorem~\ref{thm:cusp}, the two terms
$n^3\tau_{3,d}^2$ and $n^4\tau_{4,d}^2$ determine the detection
threshold. They are comparable to the squared signal-to-noise
ratios of the signed triangle and quadrilateral counts, respectively.
Here the noise is measured under $Q_{n,d}$ (see
Section~\ref{sec:cusp-spectrum} for details).
The models are indistinguishable when $\Lambda_{n,d}\to0$
and distinguishable when $\Lambda_{n,d}\to\infty$.
Thus these two cycle counts together attain the optimal threshold
throughout the cusp class.
Except in the mixed case with $\beta=\alpha+1/3$,
we have $\Lambda_{n,d}\asymp n^3\tau_{3,d}^2$
whenever $\Lambda_{n,d}=O(1)$, and the threshold is
$d_\star=n^s$, where
\begin{equation}
s=
\begin{cases}
3/(3\alpha+2),&\text{pure even},\\
3/(3\beta+1),&\text{pure odd},\\
\max\{3/(3\alpha+2),3/(3\beta+1)\},&\text{mixed},\\
3,&\text{jump}.
\end{cases}
\label{eq:cusp-dimension}
\end{equation}
On the line $\beta=\alpha+1/3$, the leading contributions to
$\tau_{3,d}$ can cancel, weakening the triangle signal.
Both traces may then be needed to determine the threshold;
see Section~\ref{sec:cusp-converse}.
For pure odd cusps, \eqref{eq:cusp-dimension} gives
$d_\star=n^{3/(3\beta+1)}$.
The exponent approaches $3$ as $\beta\downarrow0$ and equals
$3/4$ at $\beta=1$, connecting the jump threshold $n^3$
to the smooth threshold $n^{3/4}$.
Signed triangle counts attain the optimal threshold throughout
this family and in the jump case.

Beyond the cusp class, a natural question is how the detection threshold
changes for smooth kernels with $K'(0)=0$. The linear term then vanishes,
so higher derivatives determine the local behavior of $K$. We focus
on kernels whose first nonzero derivative at zero has a finite order
$m\ge2$. Their leading nonconstant Taylor term is proportional to $t^m$.
Since $\langle X_1,X_2\rangle=O_{\P}(d^{-1/2})$, this term has typical size $d^{-m/2}$,
so $m$ directly affects the strength of the geometric signal.
We therefore classify these kernels by $m$, which measures their
flatness at zero. We first introduce the kernel class $\maK_{\mathrm{smooth}}(m)$.

\begin{definition}[The class $\maK_{\mathrm{smooth}}(m)$]
\label{def:finite-flat}
For a fixed integer $m\ge2$, let $\maK_{\mathrm{smooth}}(m)$
consist of fixed dense links $K\in C^{m+6}([-1,1])$ satisfying
$K'(0)=\cdots=K^{(m-1)}(0)=0$ and $K^{(m)}(0)\ne0$.
\end{definition}

Let $q_m=1$ for odd $m$ and $q_m=2$ for even $m$. Specifically, we have the following theorem.

\begin{theorem}[Detection for finitely flat smooth kernels]
\label{thm:finite-flat}
Fix an integer $m\ge2$ and $K\in\maK_{\mathrm{smooth}}(m)$.
As $n,d\to\infty$,
\begin{equation}
 \TV(P_{n,d},Q_{n,d})\longrightarrow
 \begin{cases}
  1,&d\ll n^{3/(3m+q_m)},\\
  0,&d\gg n^{3/(3m+q_m)}.
 \end{cases}
 \label{eq:finite-flat-threshold}
\end{equation}
Specifically, a test based on the signed triangle count attains this optimal detection threshold.
\end{theorem}

In view of Theorem~\ref{thm:finite-flat}, the critical dimension
$d_\star=n^{3/(3m+q_m)}$ is attained by the signed triangle count (see Section~\ref{sec:detection} for details).
At the level of the threshold formulas, odd $m$ corresponds to
the odd cusp expression with $\beta=m$, while even $m$ corresponds
to the even cusp expression with $\alpha=m$. The resulting exponents
are $3/(3m+1)$ and $3/(3m+2)$, respectively.

By Theorems~\ref{thm:cusp} and~\ref{thm:finite-flat}, signed triangles
and quadrilaterals suffice for optimal detection across these two
classes of fixed kernels (see Figure~\ref{fig:threshold-exponents} for details). Together, the results show how the local
behavior of $K$ near zero determines the critical dimension for
a broad range of kernels. In the smooth class, the leading Taylor
monomial $t^m$ already identifies the relevant parity: odd $m$
produces an odd leading term, and even $m$ an even leading term.
Thus, the two cases are captured by the single order $m$, without
separate assumptions on the even and odd components.

These results extend the conclusions of~\cite{mao2026random}, which
proves that signed triangle counts are optimal for fixed smooth
kernels with $K'(0)\ne0$.
Theorem~\ref{thm:finite-flat} shows that triangles remain optimal
when the first nonzero derivative at zero has finite order $m\ge2$.
Theorem~\ref{thm:cusp} shows that this need not hold for continuous
kernels that are not differentiable at zero.
For certain mixed cusps, cancellation weakens the triangle signal,
and signed quadrilateral counts attain the optimal threshold.


\begin{figure}[htbp]
\centering
\begingroup
\definecolor{cyclepossible}{RGB}{223,239,237}
\definecolor{cycleimpossible}{RGB}{245,229,216}
\definecolor{cycleink}{RGB}{44,49,55}
\resizebox{.9\textwidth}{!}{%
\begin{tikzpicture}[
  x=1cm,y=1cm,font=\small,text=cycleink,
  axis/.style={
    draw=cycleink,line width=.65pt,
    -{Stealth[length=4pt]}
  },
  boundary/.style={draw=cycleink,line width=1.1pt},
  secondary/.style={
    draw=cycleink,line width=.85pt,dashed
  },
  triangles/.style={
    pattern={Lines[angle=90,distance=5.5pt,line width=.35pt]},
    pattern color=cycleink!60
  },
  quadrilaterals/.style={
    pattern={Lines[angle=0,distance=5.5pt,line width=.35pt]},
    pattern color=cycleink!60
  }
]

\begin{scope}
  \fill[cyclepossible] (0,0) rectangle (3.55,2.80);
  \fill[cycleimpossible]
    (0,0) -- (0,2.80) -- (2.35,2.80) -- cycle;

  \path[triangles]
    (0,0) -- (2.35,2.80) -- (3.55,2.80)
    -- (3.55,0) -- cycle;
  \path[quadrilaterals]
    (0,0) -- (3.55,1.22) -- (3.55,0) -- cycle;

  \draw[secondary] (0,0) -- (3.55,1.22);
  \draw[boundary] (0,0) -- (2.35,2.80)
    node[pos=.63,sloped,above=3pt,
         fill=cycleimpossible,inner sep=1pt] {slope $s$};

  \draw[axis] (0,0) -- (3.82,0)
    node[right] {$\log n$};
  \draw[axis] (0,0) -- (0,3.05)
    node[above] {$\log d$};
  \node[anchor=north] at (1.80,-.35)
    {\textup{(a)} No cancellation};
\end{scope}

\begin{scope}[xshift=5.05cm]
  \fill[cyclepossible] (0,0) rectangle (3.55,2.80);
  \fill[cycleimpossible]
    (0,0) -- (0,2.80) -- (2.35,2.80) -- cycle;

  \path[quadrilaterals]
    (0,0) -- (2.35,2.80) -- (3.55,2.80)
    -- (3.55,0) -- cycle;
  \path[triangles]
    (0,0) -- (3.55,1.22) -- (3.55,0) -- cycle;

  \draw[secondary] (0,0) -- (3.55,1.22);
  \draw[boundary] (0,0) -- (2.35,2.80);

  \draw[axis] (0,0) -- (3.82,0)
    node[right] {$\log n$};
  \draw[axis] (0,0) -- (0,3.05)
    node[above] {$\log d$};
  \node[anchor=north] at (1.80,-.35)
    {\textup{(b)} Cancellation};
\end{scope}

\begin{scope}[
  xshift=9.9cm,yshift=2.46cm,
  every node/.style={anchor=west,inner sep=0pt},
  legend border/.style={draw=cycleink!35,line width=.4pt}
]
  \foreach \y/\shade/\entry in {
    0/cycleimpossible/Impossible,
    -.50/cyclepossible/Detectable
  } {
    \filldraw[fill=\shade,legend border]
      (0,\y-.18) rectangle (.48,\y+.18);
    \node at (.65,\y) {\entry};
  }

  \foreach \y/\texture/\entry in {
    -1.12/triangles/Triangles,
    -1.62/quadrilaterals/Quadrilaterals
  } {
    \fill[cyclepossible]
      (0,\y-.18) rectangle (.48,\y+.18);
    \path[\texture]
      (0,\y-.18) rectangle (.48,\y+.18);
    \draw[legend border]
      (0,\y-.18) rectangle (.48,\y+.18);
    \node at (.65,\y) {\entry};
  }

  \fill[cyclepossible]
    (0,-2.30) rectangle (.48,-1.94);
  \path[triangles]
    (0,-2.30) rectangle (.48,-1.94);
  \path[quadrilaterals]
    (0,-2.30) rectangle (.48,-1.94);
  \draw[legend border]
    (0,-2.30) rectangle (.48,-1.94);
  \node at (.65,-2.12) {Triangles + quadrilaterals};
\end{scope}

\end{tikzpicture}%
}
\endgroup
\caption{Solid lines separate detection from impossibility; dashed lines mark
the boundary for the other cycle count. Vertical and horizontal
hatching indicate detection by signed triangles and quadrilaterals,
respectively.
\textup{(a)} Triangles attain $d_\star=n^s$.
\textup{(b)} Cancellation with quadrilateral dominance.}
\label{fig:threshold-exponents}
\end{figure}
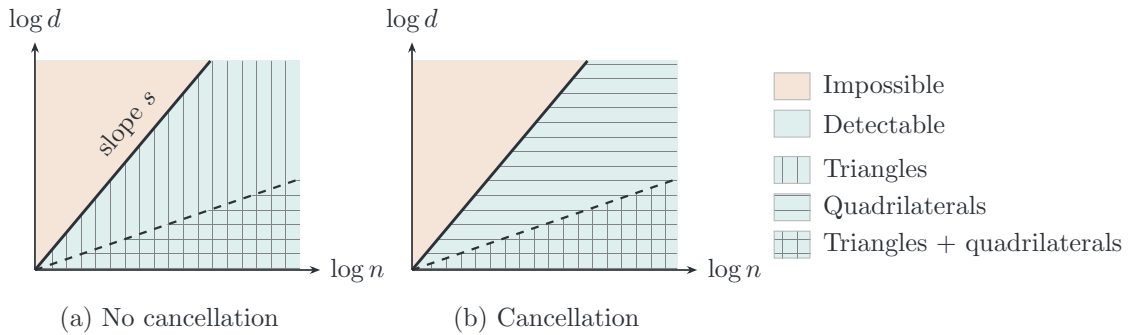

\subsection{Proof overview}

For notational simplicity, we write $P=P_{n,d}$ and $Q=Q_{n,d}$
throughout the rest of the paper.
The detection results in Theorems~\ref{thm:cusp}
and~\ref{thm:finite-flat} use the signed cycle counts $S_3$ and $S_4$.
Let $A=(A_{ij})$ be the adjacency matrix and let
$Y_{ij}=(A_{ij}-p_d)/s_d$. The counts $S_3$ and $S_4$ sum products
of these variables over all triangles and all distinct quadrilaterals,
respectively. Their means satisfy
\[
 \E_Q[S_3]=\E_Q[S_4]=0,\quad
 \E_P[S_3]=\binom n3\tau_{3,d},\quad
 \E_P[S_4]=3\binom n4\tau_{4,d}.
\]
Independence under $Q$ gives $\var_Q(S_3)=\binom n3$ and
$\var_Q(S_4)=3\binom n4$. Under $P$, we bound the variances by
classifying pairs of overlapping cycles. The test selects the model
whose mean is closer to $S_\ell$. Chebyshev's inequality gives
vanishing testing error whenever
\[
 \frac{\var_P(S_\ell)+\var_Q(S_\ell)}
 {\big|\E_P[S_\ell]-\E_Q[S_\ell]\big|^2}\longrightarrow0,
 \quad \ell\in\{3,4\}.
\]
We verify this condition for the appropriate count in each
detection regime (see Section~\ref{sec:detection}).

For the impossibility results, let $\KL$ denote the
Kullback--Leibler divergence. By Pinsker's inequality,
$\TV(P,Q)^2\le\KL(P\Vert Q)/2$,
so it suffices to prove $\KL(P\Vert Q)=o(1)$.
Let $A^{(r)}$ be the graph induced by $[r]$, and let
$B^{(r+1)}=(A_{i,r+1})_{i\le r}$ contain the edges joining
the new vertex to $[r]$. Following~\cite[Section~3]{mao2026random},
we add one vertex at a time and apply the chain rule
(see Figure~\ref{fig:sequential-kl}):
\[
 \KL(P\Vert Q)
 =\sum_{r=1}^{n-1}\E_P\Big[
 \KL\big(P(B^{(r+1)}\mid A^{(r)})
 \Vert\Bern(p_d)^{\otimes r}\big)\Big].
\]
Integrating out the new latent point gives
\[
 h_{k,d}(x_1,\ldots,x_k)
 =\E_Z\Big[\prod_{i=1}^k
 \kappa_d(\langle x_i,Z\rangle)\Big],
 \quad Z\sim\sigma_d.
\]
Expanding the conditional likelihood and bounding KL by the
chi-squared divergence gives $\KL(P\Vert Q)\le\sum_{k=2}^{n-1}G_k$, where
\[
 G_k=\sum_{r=k}^{n-1}\binom rk
 \E_P\Big[\E_P[h_{k,d}(X_1,\ldots,X_k)\mid A^{(r)}]^2\Big].
\]
Centering removes order one, and exchangeability gives $\binom rk$
(see Lemma~\ref{lem:cavity}). Thus $G_k$ measures how well the
observed graph predicts an interaction among $k$ latent points.
Conditional Jensen's inequality gives
\[
 G_k\le\binom n{k+1}
 \E[h_{k,d}(X_1,\ldots,X_k)^2],
 \quad G_2\le\binom n3\tau_{4,d}.
\]
Prior moment bounds therefore control sufficiently high orders
(see Lemma~\ref{lem:contracted} and Proposition~\ref{prop:cavity-tail}).
However, the displayed bound on $G_2$ need not vanish when both
cycle signals do. The lower orders require sharper posterior estimates.

\

\noindent\textbf{Cusp kernels.}
For continuous cusps, we first modify the link outside a small
neighborhood of zero. Concentration of the inner products makes
the changes in KL and the cycle traces exponentially small
(see Lemma~\ref{lem:localization} and Figure~\ref{fig:cusp-localization}).
For mixed cusps with $\beta=\alpha+1/3$, the leading even and odd
contributions to $\tau_{3,d}$ can cancel (see Remark~\ref{rmk:cancel}).
This shrinks the triangle signal, while our target remains
\[
 \KL(P\Vert Q)\lesssim
 \Lambda_{n,d}=n^3\tau_{3,d}^2+n^4\tau_{4,d}^2
 \quad\text{when }\Lambda_{n,d}=o(1).
\]
Bounding the even and odd contributions separately loses this
cancellation. We must preserve the actual signed trace before
squaring and estimate the remaining terms more carefully.

\begin{itemize}
\item \textbf{Order two.} The interaction $h_{2,d}$ represents a path of length two.
The direct edge closes a triangle; a common neighbor closes a
quadrilateral. At the triangle scale, the conditional estimate gives
\[
 G_2\lesssim\Lambda_{n,d}+\frac{n^5\tau_{4,d}^2}{d}
 \quad\text{(Proposition~\ref{prop:cusp-triangle-two}).}
\]
The extra term can exceed $n^4\tau_{4,d}^2$ in the cancellation
regime. We then retain the posterior of the remaining latent points.
Integrating the queried points produces $H_{4,d}$, the kernel
profile of $T_{\kappa_d}^4$. A second posterior estimate gives
\[
 G_2\lesssim\Lambda_{n,d}
 +n^5\sup_{2\le r\le n}
 \E_P\Big[\E_P[H_{4,d}(\langle X_1,X_2\rangle)\mid A^{(r)}]^2\Big]
 \lesssim\Lambda_{n,d}
\]
on the critical relation (see Proposition~\ref{prop:cusp-critical-two}
and Figure~\ref{fig:order-two-comparison}).

\item \textbf{Order three.} Parity improves the moment of $h_{3,d}$:
only products with zero or two odd factors survive.
This suffices at the triangle scale. On the critical relation,
we retain the three possible internal edges and expand their
eight configurations. Bounding their means and centered
conditional contributions separately gives
$G_3\lesssim n^4\tau_{4,d}^2$
(see Section~\ref{sec:order-3-csup} and Appendix~\ref{app:cusp-order-three}).

\item \textbf{Higher orders.} Without cancellation, conditional estimates
handle $4\le k\le9$, and prior moments control $k\ge10$.
On the critical relation, perfect matchings at fixed orders contribute
powers of the actual $\tau_{3,d}$, while graph estimates control
the other configurations. Uniform bounds then handle orders up
to logarithmic size and the remaining tail
(see Section~\ref{sec:order-high-csup}).
\end{itemize}
For jumps, a comparison of Gram matrix laws and data processing
give $\KL(P\Vert Q)\lesssim n^3/d$
(see Proposition~\ref{prop:jump}).

\

\noindent\textbf{Finitely flat kernels.}
We first work with finite harmonic expansions, keeping the
individual coefficient sizes to use the vanishing derivatives.
Write $\kappa(x,y)$ for the auxiliary centered kernel.
Draw $A$ from the model under consideration, then draw two latent
configurations $X,X'$ independently from the posterior given $A$.
For any square-integrable $f$,
\[
 \E\big[\E[f(X)\mid A]^2\big]=\E[f(X)f(X')],
\]
where both expectations use this joint sampling.
The products on the right are easier to estimate under two
independent prior samples. Transferring these estimates to the
posterior has an exponential cost (see Lemma~\ref{lem:planting}).

\begin{itemize}
\item \textbf{Order two.} The replica identity gives a positive semidefinite
quadratic form. A direct Gram matrix bound is too large, so we
restrict to an event of high probability where its operator norm
is smaller. Prior exponential moments then control the posterior
quadratic form (see Section~\ref{sec:smooth-order-two}).

\item \textbf{Order three.} Conditioning on one prior sample gives a
$U$-statistic whose kernel is centered in each variable.
A consequence of~\cite[Theorem~7]{adamczak2006moment} bounds its
tail through four terms. We estimate these terms and truncate
large kernel values to obtain a tail strong enough for the
posterior comparison (see Lemma~\ref{lem:adamczak-cor} and
Section~\ref{sec:smooth-order-three}).

\item \textbf{Higher orders.} For each fixed $k\ge4$, Newton's identities rewrite sums over
distinct vertices in terms of power sums. Let $Z,Z'$ be independent
uniform points, independent of the replicas, and let
$r_i=\kappa(X_i,Z)\kappa(X_i',Z')$. The formal identity
\[
 \prod_{i=1}^n(1+tr_i)
 =\exp\Big(\sum_{j\ge1}\frac{(-1)^{j-1}t^j}{j}
               \sum_{i=1}^n r_i^j\Big)
\]
expresses the coefficient of $t^k$ through power sums.
Only the first power sum needs a posterior estimate, which reduces
to a quadratic overlap. Higher power sums admit direct moment
bounds (see Section~\ref{sec:smooth-higher-orders}).
\end{itemize}

To return to the original kernel, let $\Pi$ be the centered Taylor polynomial of $\kappa_d$
through degree $m+5$ and let $R=\kappa_d-\Pi$.
For each fixed $1\le q<\infty$,
\[
 \|R(T_d)\|_q\lesssim_q d^{-(m+6)/2}=d^{-3}d^{-m/2}.
\]
This gain controls the errors from replacing the kernel in both
the likelihood weight and $h_{k,d}$ under the original posterior
(see Section~\ref{sec:smooth-transfer}). Together with the tail
bound in Proposition~\ref{prop:cavity-tail}, these estimates give
$\KL(P\Vert Q)=o(1)$.



\subsection{Related work}

Latent position models describe network dependence through
unobserved feature vectors associated with the vertices.
Conditional on these vectors, edges are independent, with
connection probabilities determined by the features of their
endpoints. One natural choice is to let vertices with similar
features connect more often, as in latent distance
models~\cite{hoff2002latent}.
Inner products provide another way to describe feature interactions.
Eigenmodels use weighted inner products~\cite{hoff2007modeling},
while random dot product graphs use the inner product itself as
the connection probability~\cite{young2007random}.
The latter admit a rich theory of spectral estimation and
inference for latent positions~\cite{athreya2018statistical}.
In a random dot product graph, the latent distribution must be
supported on a set whose pairwise inner products lie in $[0,1]$.
This condition can hold in any dimension, although much of the
classical inference theory takes $d$ to be fixed as the graph
grows~\cite{athreya2018statistical}.
When many latent attributes may affect connections, it is also
natural to study a growing dimension.
The spherical model provides a simple setting for this question.
Inner products of unit vectors lie in $[-1,1]$ in every dimension,
and a fixed function $K:[-1,1]\to[0,1]$ gives valid connection
probabilities. Uniform sampling on the sphere also gives every
latent position the same expected connection probability.
As the dimension grows, inner products concentrate near zero.
Related estimation problems for spherical graphs include recovering
an unknown connection function and latent distances~\cite{decastro2019adaptive,arayavaldivia2019latent}.

Early work studied similarity between random geometric graphs and \ER graphs through clique
numbers~\cite{devroye2011high}.
Even the hard threshold rule
$K_d(t)=\indc{t\ge\tau_{d,p}}$, with $\tau_{d,p}$ chosen to give
edge density $p$, leads to a difficult detection problem.
For fixed $p\in(0,1)$, \cite{bubeck2016testing} established
$d_\star\asymp n^3$ using signed triangles and a matching converse
based on a comparison between Wishart and Gaussian orthogonal
ensemble (GOE) matrices.
Related matrix comparisons address eigenvalue laws,
non-Gaussian entries, and the critical Wishart--GOE
window~\cite{jiang2015approximation,bubeck2018entropic,racz2019smooth}.
The sparse case required further work~\cite{brennan2020phase}.
For $p=a/n$ with fixed $a\ge1$, \cite{liu2022testing} proved
indistinguishability when $d\gg(\log n)^{36}$.
This leaves a logarithmic gap from the conjectured scale
$(nh(p))^3\asymp(\log n)^3$, where
$h(p)=p\log(1/p)+(1-p)\log(1/(1-p))$.
More recently, \cite{du2026resolution} proved indistinguishability
when $d\gg(nh(p))^3$ and $d\ge(1+\epsilon)n$, for fixed
$\epsilon>0$ and $\epsilon/n\le p\le1-\epsilon$.
Together with signed triangle tests, this resolves the threshold
for $p\gtrsim n^{-2/3}/\log n$.
Their proof expands the KL divergence into posterior moments.
An information comparison and strong log-concavity of a truncated
posterior then control these moments.
Independently, \cite{hunter2026distinguishability} established
the same threshold for
$n^{-1/5}\operatorname{polylog}(n)\le p\le1/3$.
Their proof uses Gram--Schmidt orthogonalization and a sequential
analysis of the latent vectors to estimate the probabilities of
prescribed graphs.
Very recently, \cite{du2026sharp} established the sharp detection
threshold when $p\to0$ and $p\ge n^{-1+\delta}$ for fixed $\delta>0$.
Their proof combines posterior estimates with Fourier and spherical
harmonic expansions.

The latent distribution also affects detection.
For dense anisotropic Gaussian geometric graphs,
\cite{eldan2020information,brennan2024threshold} studied how the
covariance spectrum determines the threshold.
Our model keeps the latent distribution uniform on the sphere
and varies the connection kernel.
Soft connection functions also arise in models of latent features
and in spatial random geometric graphs, where much of the work
concerns connectivity~\cite{dettmann2016random,penrose2016connectivity}.
For detecting geometry, \cite{liu2023probabilistic} studied
smooth monotone connection functions with Gaussian latent points.
Their bounds depend on both dimension and a smoothing parameter.
In the spherical setting, \cite{liu2023phase} studied independent
edge noise that interpolates between a hard geometric graph
and its matched \ER graph.
Other work studies geometry confined to a hidden set of
vertices~\cite{silber2026testing,bok2026detection}, or observed through
randomly masked edges in bipartite Gaussian latent graphs~\cite{gobel2026information}.
For fixed smooth kernels with $K'(0)\ne0$,
\cite{mao2026random} established the sharp threshold
$d_\star\asymp n^{3/4}$, attained by signed triangle counts.
Together with the hard threshold $n^3$, this result motivates
studying how the local behavior of the kernel determines detectability.
Very recently, \cite{luo2026cubic} constructed kernels depending
on $d$ for which the triangle signal cancels but signed
quadrilateral counts still detect geometry.
The same work identifies an additional detection scale for dense
polynomial kernels of bounded degree, beyond those given by
triangles and quadrilaterals.

The repeated appearance of signed triangles suggests studying
geometric detection through subgraph statistics.
More generally, the Fourier coefficients of the graph likelihood
ratio are signed subgraph moments.
By bounding these moments, \cite{bangachev2024fourier} characterized
testing with polynomials of low degree for hard spherical and
Gaussian geometric graphs, including partial observations.
However, the most effective statistic can change with the latent
geometry. For the $L_\infty$ model on a torus,
\cite{bangachev2024linfty} proved that signed quadrilateral counts
are optimal among tests of low degree, up to logarithmic factors,
while signed triangles are suboptimal.
Related Fourier methods give conditions for convergence of random
algebraic graphs to \ER graphs~\cite{bangachev2025random}.
For several classes of stochastic block models,
\cite{bangachev2025graph} showed that signed stars and quadrilaterals
suffice whenever a polynomial test of constant degree succeeds.
Subgraph statistics also provide moment estimators for network
models~\cite{bickel2011method} and tests for community
structure~\cite{gao2017testing}.
Signed cycles appear in contiguity results for planted partition
models~\cite{banerjee2018contiguity} and in tests allowing degree
heterogeneity and mixed memberships~\cite{jin2021optimal}.
The role of quadrilaterals also connects to classical graph
quasirandomness~\cite{chung1989quasi}.

\subsection{Notation and paper organization}

Let $\sigma_d$ denote the uniform probability measure on $\Sph^{d-1}$.
Unless stated otherwise, latent points are independent with law $\sigma_d$.
We use the normalization in~\eqref{eq:model-normalization} and write
$v_d=\E[\kappa_d(T_d)^2]$ for the kernel energy.
For a zonal profile $f$ with finite $L^2$ norm, its integral operator is
$T_f g(x)=\int f(\langle x,y\rangle)g(y)\sigma_d(dy)$.
The power $T_f^\ell$ denotes operator composition. For every
integer $\ell\ge2$,
\[
 \tr(T_f^\ell)
 =\E\big[f(\langle X_1,X_2\rangle)\cdots
          f(\langle X_\ell,X_1\rangle)\big].
\]
When writing $f(x,y)$, we mean $f(\langle x,y\rangle)$; integral
operators evaluated at $(x,y)$ denote their kernels.
For probability measures $P$ and $Q$ on the same measurable space,
write $\TV(P,Q)=\sup_B|P(B)-Q(B)|$ and
$\KL(P\|Q)=\int\log(dP/dQ)dP$.
For integers $r\ge1$ and $0\le N\le r$, write
$[r]=\{1,\ldots,r\}$ and $(r)_N=r(r-1)\cdots(r-N+1)$,
with $(r)_0=1$. Also, $x_+=\max\{x,0\}$.
We use $\|\cdot\|_p$ for the $L^p$ norm under the indicated measure;
for functions of latent points, the default measure is their independent prior.
The norms $\|\cdot\|_{\op}$ and $\|\cdot\|_{\HS}$ are the operator and
Hilbert--Schmidt norms, respectively.
For positive sequences, $a_n\lesssim b_n$ means $a_n=O(b_n)$,
$a_n\asymp b_n$ means $a_n\lesssim b_n\lesssim a_n$, and
$a_n\ll b_n$ means $a_n=o(b_n)$; $\gtrsim$ and $\gg$ denote the reverse
relations. Implicit constants may depend on fixed model parameters
and fixed moment orders, but not on $n$ or $d$. Constants in estimates
may change between lines unless their values have been fixed.

\

\noindent\textbf{Paper organization.}
Section~\ref{sec:detection} proves the detection results using signed
triangle and quadrilateral counts.
Section~\ref{sec:cavity} introduces the KL expansion and posterior
formulation used to prove impossibility.
Sections~\ref{sec:cusp-converse} and~\ref{sec:polynomial} establish
the matching impossibility results for cusp kernels and finitely
flat smooth kernels, respectively.
Section~\ref{sec:discussion} discusses the implications of our results
and their connections to previous work.
Some technical proofs are deferred to the appendices.

\section{Detection by cycle counts}\label{sec:detection}

\subsection{Detection guarantees for triangle and quadrilateral counts}

We prove the detection results using centered cycle counts, starting
with signed triangles.
A single edge has the same distribution under $P$ and $Q$.
Moreover, any two distinct edges have zero covariance: disjoint edges
are independent, while for edges sharing a vertex, integrating out
either nonshared latent point gives the fixed edge probability $p_d$.
A triangle is the shortest cycle whose product of centered edges
can reveal the latent geometry. Let $A=(A_{ij})$ be the adjacency
matrix and $Y=(Y_{ij})$ its standardized version, with
$Y_{ij}=(A_{ij}-p_d)/\sqrt{p_d(1-p_d)}$ for $1\le i<j\le n$.
Set $Y_{ji}=Y_{ij}$ and write $T_{ij}=\langle X_i,X_j\rangle$
for all $1\le i,j\le n$.
We first consider the signed triangle count
\begin{equation}
 S_3=\sum_{1\le i<j<k\le n}Y_{ij}Y_{jk}Y_{ki}.
 \label{eq:signed-triangle}
\end{equation}

Under $Q$, the edges $A_{ij}$ are independent Bernoulli variables
with parameter $p_d$, so the standardized variables $Y_{ij}$ are
independent with mean zero and variance one. Thus
$\E_Q[S_3]=0$ and $\var_Q(S_3)=\binom n3$.
Under $P$, the latent geometry can induce a nonzero mean.
Intuitively, as illustrated in Figure~\ref{fig:triangle-signal},
observing the edges $\{i,j\}$ and $\{i,k\}$ can change the
probability that $\{j,k\}$ is also present.
This triangle signal is captured by
$\tau_{3,d}=\tr(T_{\kappa_d}^3)$, with
$\E_P[S_3]=\binom n3\tau_{3,d}$.
The following proposition gives the moments under $P$.

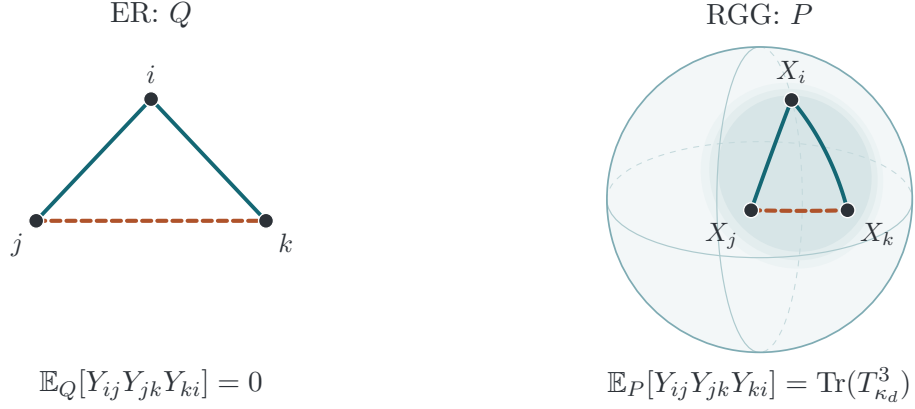
\begin{figure}[!htbp]
\centering
\begingroup
\definecolor{triink}{RGB}{44,49,55}
\definecolor{trigeometry}{RGB}{20,103,116}
\definecolor{triclose}{RGB}{175,83,43}

\begin{tikzpicture}[
  x=1cm,y=1cm,line cap=round,line join=round,
  tri point/.style={
    circle,fill=triink,draw=white,line width=.5pt,
    minimum size=5.6pt,inner sep=0pt
  },
  tri label/.style={
    font=\small,text=triink,inner sep=1pt
  }
]
\path[use as bounding box] (0,0) rectangle (16,5.6);

\node[font=\normalsize,text=triink] at (3.9,5.25)
  {ER: $Q$};
\node[font=\normalsize,text=triink] at (11.95,5.25)
  {RGG: $P$};

\coordinate (qi) at (3.9,4.13);
\coordinate (qj) at (2.38,2.52);
\coordinate (qk) at (5.42,2.52);

\draw[trigeometry,line width=1.45pt] (qj)--(qi)--(qk);
\draw[
  triclose,line width=1.55pt,
  dash pattern=on 3.4pt off 2.3pt
] (qj)--(qk);

\foreach \v in {i,j,k}{
  \node[tri point] at (q\v) {};
}
\node[tri label,above=5pt] at (qi) {$i$};
\node[tri label,below left=5pt] at (qj) {$j$};
\node[tri label,below right=5pt] at (qk) {$k$};

\fill[trigeometry!5] (11.95,2.8) circle[radius=2.02];

\foreach \caprad/\captint in {.67/8,.61/12,.55/17}{
  \pgfmathsetmacro{\capcx}
    {11.95+2.02*.270473744*cos(deg(\caprad))}
  \pgfmathsetmacro{\capcy}
    {2.8+2.02*.200350921*cos(deg(\caprad))}
  \pgfmathsetmacro{\capmajor}
    {2.02*sin(deg(\caprad))}
  \pgfmathsetmacro{\capminor}
    {\capmajor*.941649331}
  \fill[
    trigeometry!\captint,
    rotate around={36.528855:(\capcx,\capcy)}
  ]
    (\capcx,\capcy)
    ellipse[x radius=\capminor cm,y radius=\capmajor cm];
}

\draw[
  trigeometry!25,line width=.4pt,
  dash pattern=on 2pt off 2.5pt
]
  (9.93,2.8)
  arc[start angle=180,end angle=0,
      x radius=2.02,y radius=.7675904];

\draw[trigeometry!35,line width=.45pt]
  (9.93,2.8)
  arc[start angle=180,end angle=360,
      x radius=2.02,y radius=.7675904];

\draw[trigeometry!35,line width=.45pt]
  (11.95,4.82)
  arc[start angle=90,end angle=270,
      x radius=.5656,y radius=2.02];

\draw[
  trigeometry!25,line width=.4pt,
  dash pattern=on 2pt off 2.5pt
]
  (11.95,4.82)
  arc[start angle=90,end angle=-90,
      x radius=.5656,y radius=2.02];

\draw[trigeometry!55,line width=.65pt]
  (11.95,2.8) circle[radius=2.02];

\coordinate (pi) at (12.372287,4.119331);
\coordinate (pj) at (11.837516,2.659943);
\coordinate (pk) at (13.113074,2.659943);

\def\rggarc#1#2#3#4#5#6{%
  \draw[#1,domain=0:1,samples=60,variable=\t]
    plot (
      {11.95+
       (sin((1-\t)*#6)*(#2-11.95)
        +sin(\t*#6)*(#4-11.95))/sin(#6)},
      {2.8+
       (sin((1-\t)*#6)*(#3-2.8)
        +sin(\t*#6)*(#5-2.8))/sin(#6)}
    );
}

\rggarc{trigeometry,line width=1.45pt}
  {12.372287}{4.119331}
  {11.837516}{2.659943}{48.086836}

\rggarc{trigeometry,line width=1.45pt}
  {12.372287}{4.119331}
  {13.113074}{2.659943}{48.086836}

\rggarc{
  triclose,line width=1.55pt,
  dash pattern=on 3.4pt off 2.3pt
}
  {11.837516}{2.659943}
  {13.113074}{2.659943}{38.355353}

\foreach \v in {i,j,k}{
  \node[tri point] at (p\v) {};
}
\node[tri label,above=5pt] at (pi) {$X_i$};
\node[tri label,below left=5pt] at (pj) {$X_j$};
\node[tri label,below right=5pt] at (pk) {$X_k$};

\node[font=\normalsize,text=triink] at (3.9,.35)
  {$\E_Q[Y_{ij}Y_{jk}Y_{ki}]=0$};
\node[font=\normalsize,text=triink] at (11.95,.35)
  {$\E_P[Y_{ij}Y_{jk}Y_{ki}]=\tr(T_{\kappa_d}^3)$};

\end{tikzpicture}
\endgroup
\caption{Triangle closure for an increasing link.
Observing the two solid edges favors latent configurations with
$X_j$ and $X_k$ near $X_i$, making proximity between them, and hence
the dashed closing edge, more likely.
Under $Q$, the closing edge is independent of the first two.}
\label{fig:triangle-signal}
\end{figure}

\begin{proposition}[Moments of signed triangles]
\label{prop:triangle-moments}
For any $K:[-1,1]\to[c,1-c]$ with fixed $0<c<1/2$, there is a
constant $C_{\ref{prop:triangle-moments}}(c)$ such that
\begin{equation}
 \E_P[S_3]=\binom n3\tau_{3,d},\quad
 \var_P(S_3)\le C_{\ref{prop:triangle-moments}}(c)
 \big(n^3+n^4\tau_{4,d}\big).
 \label{eq:tri-planted}
\end{equation}
The test
$\phi_3=\mathbf1\{\sgn(\E_P[S_3])S_3\ge|\E_P[S_3]|/2\}$
has type I and type II errors tending to zero if
\begin{equation}
 \frac{n^6\tau_{3,d}^2}{n^3+n^4\tau_{4,d}}\longrightarrow\infty.
 \label{eq:tri-snr}
\end{equation}
\end{proposition}

\begin{proof}
Write $Z_{ijk}=Y_{ij}Y_{jk}Y_{ki}$ and $b_c=(1-c)/c$.
Since $p_d\in[c,1-c]$, we have $|Y_{ij}|^2\le b_c$.
Conditional independence and integration around the triangle give
$\E_P[Z_{123}]=\tau_{3,d}$, hence
$\E_P[S_3]=\binom n3\tau_{3,d}$.
Rotational invariance also gives
$\E_P[Z_{123}\mid X_1]=\tau_{3,d}$.

For the variance, disjoint triangles are independent. Triangles
sharing just one vertex are conditionally independent given that
vertex and have constant conditional means, so their covariance is
zero. For identical triangles,
$\var_P(Z_{123})\le\E_P[Z_{123}^2]\le b_c^3$.
For triangles sharing an edge, define the profile $H_d$ by
\begin{equation}
 H_d(\langle x,y\rangle)
 =\int\kappa_d(\langle x,z\rangle)
       \kappa_d(\langle z,y\rangle)\sigma_d(dz).
 \label{eq:contracted-profile}
\end{equation}
This is the integral kernel of $T_{\kappa_d}^2$. Condition on $(X_1,X_2,A_{12})$.
The variables $Z_{123}$ and $Z_{124}$ are then independent, with
common conditional mean $Y_{12}H_d(T_{12})$, where
$T_{12}=\langle X_1,X_2\rangle$. Therefore
\begin{align*}
 \operatorname{Cov}_P(Z_{123},Z_{124})
 &=\var_P\big(Y_{12}H_d(T_{12})\big)\le b_c\E[H_d(T_d)^2]=b_c\tau_{4,d}.
\end{align*}
The last equality uses~\eqref{eq:contracted-profile} and the symmetry
of $T_{\kappa_d}$. There are $12\binom n4$ ordered pairs of distinct
triangles sharing an edge, so
\begin{align*}
 \var_P(S_3)
 &\le b_c^3\binom n3+12b_c\binom n4\tau_{4,d}\le C_{\ref{prop:triangle-moments}}(c)
       \big(n^3+n^4\tau_{4,d}\big),
\end{align*}
where one may take
$C_{\ref{prop:triangle-moments}}(c)=b_c^3/6+b_c/2$.
This proves~\eqref{eq:tri-planted}.

Condition~\eqref{eq:tri-snr} makes $\E_P[S_3]\ne0$ eventually.
Since $\E_Q[S_3]=0$ and $\var_Q(S_3)=\binom n3$,
Chebyshev's inequality gives
\[
 \P_Q[\phi_3=1]+\P_P[\phi_3=0]
 \le\frac{4\{\var_Q(S_3)+\var_P(S_3)\}}{\E_P[S_3]^2}
 \lesssim
 \frac{n^3+n^4\tau_{4,d}}{n^6\tau_{3,d}^2}
 \longrightarrow0.
\qedhere
\]
\end{proof}

We then consider signed cycles of length four, whose mean is proportional to $\tau_{4,d}$:
\begin{equation}
 S_4=\sum_{1\le i<j<k<\ell\le n}
 \big(Y_{ij}Y_{jk}Y_{k\ell}Y_{\ell i}
     +Y_{ij}Y_{j\ell}Y_{\ell k}Y_{ki}
     +Y_{ik}Y_{kj}Y_{j\ell}Y_{\ell i}\big).
 \label{eq:signed-square}
\end{equation}
Under $Q$, distinct cycle monomials are orthogonal. Hence
$\E_Q[S_4]=0$ and $\var_Q(S_4)=3\binom n4$.
The moments under $P$ again depend on cycle overlaps and involve
the same contracted profile $H_d$ from~\eqref{eq:contracted-profile}.

\begin{proposition}[Moments of signed cycles of length four]
\label{prop:square-moments}
For any $K:[-1,1]\to[c,1-c]$ with fixed $0<c<1/2$, there is a
constant $C_{\ref{prop:square-moments}}(c)$ such that
\begin{equation}
 \E_P[S_4]=3\binom n4\tau_{4,d},\quad
 \var_P(S_4)\le C_{\ref{prop:square-moments}}(c)
 \big(n^4+n^5\tau_{4,d}
       +n^6\big(\tau_{6,d}+\E[H_d(T_d)^4]\big)\big).
 \label{eq:square-variance}
\end{equation}
The test $\phi_4=\mathbf1\{S_4\ge\E_P[S_4]/2\}$
has type I and type II errors tending to zero if
\begin{equation}
 \frac{n^8\tau_{4,d}^2}
 {n^4+n^5\tau_{4,d}
       +n^6\big(\tau_{6,d}+\E[H_d(T_d)^4]\big)}
 \longrightarrow\infty.
 \label{eq:square-snr}
\end{equation}
\end{proposition}

\begin{proof}
For distinct $i,j,k,\ell$, write
$Z_{ijk\ell}=Y_{ij}Y_{jk}Y_{k\ell}Y_{\ell i}$.
Let $b_c=(1-c)/c$, so both $|Y_{ij}|^2$ and
$|\kappa_d(t)|^2$ are bounded by $b_c$.
Under $P$, conditional independence and integration around the
cycle give $\E_P[Z_{ijk\ell}]=\tau_{4,d}$, hence
$\E_P[S_4]=3\binom n4\tau_{4,d}$.

For the variance, disjoint cycles are independent.
For cycles sharing only one vertex, such as $Z_{1234}$ and
$Z_{1567}$, conditioning on $X_1$ makes the two products
independent. Rotational invariance gives the same constant
conditional mean $\tau_{4,d}$, so their covariance is zero.

\

 \noindent\textbf{Case 1.}
We first suppose that exactly two vertices are shared.
First, consider $Z_{1234}$ and $Z_{1256}$, which share the edge
$\{1,2\}$. Given $(X_1,X_2,A_{12})$, they are independent with
the same conditional mean. Hence,
\[
\operatorname{Cov}_P(Z_{1234},Z_{1256})
=
\var_P\big(\E_P[Z_{1234}\mid X_1,X_2,A_{12}]\big)
\le b_c\tau_{6,d}.
\]
Indeed, after squaring the conditional mean and using
$Y_{12}^2\le b_c$, the two remaining paths of length three
form a cycle of length six.
Next, consider $Z_{1234}$ and $Z_{1356}$. Their shared vertices
$1,3$ are opposite in the first cycle and adjacent in the second.
Given $(X_1,X_3)$, the two products are independent, with
\begin{align*}
 \E_P[Z_{1234}\mid X_1,X_3]
 &=H_d(T_{13})^2,\quad
 \E_P[Z_{1356}\mid X_1,X_3]
 =\kappa_d(T_{13})\E_P[Y_{35}Y_{56}Y_{61}\mid X_1,X_3].
\end{align*}
The second conditional mean has squared $L^2$ norm at most
$b_c\tau_{6,d}$ by the same calculation with two paths of
length three. Cauchy--Schwarz therefore gives
\[
\big|\operatorname{Cov}_P(Z_{1234},Z_{1356})\big|
\le
\sqrt{b_c\tau_{6,d}\E[H_d(T_d)^4]}.
\]
Finally, consider $Z_{1234}$ and $Z_{1536}$, whose shared
vertices $1,3$ are opposite in both cycles.
Given $(X_1,X_3)$, the two products are independent with
common conditional mean $H_d(T_{13})^2$. Thus,
\[
\operatorname{Cov}_P(Z_{1234},Z_{1536})
=
\var\big(H_d(T_{13})^2\big)
\le \E[H_d(T_d)^4].
\]
All three bounds are at most
$b_c\tau_{6,d}+\E[H_d(T_d)^4]$.

\

 \noindent\textbf{Case 2.}
Suppose now that exactly three vertices are shared.
For $Z_{1234}$ and $Z_{1235}$, the common edges are
$\{1,2\}$ and $\{2,3\}$.
Conditioning on $(X_1,X_2,X_3,A_{12},A_{23})$ makes the
two products independent with common conditional mean
$Y_{12}Y_{23}H_d(T_{13})$. Hence,
\[
\operatorname{Cov}_P(Z_{1234},Z_{1235})
=
\var_P\big(Y_{12}Y_{23}H_d(T_{13})\big)
\le b_c^2\tau_{4,d},
\]
where we used $\E[H_d(T_d)^2]=\tau_{4,d}$.
The other possibility is represented by $Z_{1234}$ and
$Z_{1325}$, which share only the edge $\{2,3\}$.
Given $(X_1,X_2,X_3,A_{23})$, the two products are independent
with conditional means
$Y_{23}\kappa_d(T_{12})H_d(T_{13})$ and
$Y_{23}\kappa_d(T_{13})H_d(T_{12})$, respectively.
Each has second moment at most $b_c^2\tau_{4,d}$.
Cauchy--Schwarz therefore gives
\[
\big|\operatorname{Cov}_P(Z_{1234},Z_{1325})\big|
\le b_c^2\tau_{4,d}.
\]

\

 \noindent\textbf{Case 3.}
Pairs sharing all four vertices have absolute covariance at most
$b_c^4$. 

There are $O(n^6)$, $O(n^5)$, and $O(n^4)$ ordered
pairs sharing exactly two, three, and four vertices, respectively.
Summing these bounds yields
\[
\var_P(S_4)
\le C_{\ref{prop:square-moments}}(c)
\Big(n^4+n^5\tau_{4,d}
+n^6\big(\tau_{6,d}+\E[H_d(T_d)^4]\big)\Big),
\]
which proves~\eqref{eq:square-variance}.
Condition~\eqref{eq:square-snr} ensures that $\tau_{4,d}>0$
eventually, so $\E_P[S_4]>0$.
Since $\E_Q[S_4]=0$ and $\var_Q(S_4)=3\binom n4$,
Chebyshev's inequality gives
\begin{align*}
\P_Q[\phi_4=1]+\P_P[\phi_4=0]
&\le
\frac{4\big(\var_Q(S_4)+\var_P(S_4)\big)}{\E_P[S_4]^2}\lesssim
\frac{n^4+n^5\tau_{4,d}
+n^6\big(\tau_{6,d}+\E[H_d(T_d)^4]\big)}
{n^8\tau_{4,d}^2}
\to0.
\qedhere
\end{align*}
\end{proof}

We next use the spherical harmonic expansion to control the traces in
the two variance bounds (see, e.g., \cite[Chapter~1]{dai2013approximation}).
For $d\ge3$ and an integer $q\ge0$, let $\maH_{q,d}$ denote the
restrictions to $\Sph^{d-1}$ of homogeneous harmonic polynomials
of degree $q$ in $d$ variables. Its dimension is
\begin{equation}
N_{q,d}
=
\binom{d+q-1}{q}-\binom{d+q-3}{q-2}
=
\frac{d^q}{q!}\big(1+O_q(d^{-1})\big),
\label{eq:multiplicity}
\end{equation}
where the second binomial coefficient is zero for $q=0,1$,
and the asymptotic holds for fixed $q$ as $d\to\infty$.
Let $P_{q,d}$ be the Gegenbauer polynomial of degree $q$ with parameter
$(d-2)/2$ (i.e., orthogonal to all polynomials of lower degree under
the law of $T_d$), normalized by $P_{q,d}(1)=1$. Let $\psi_{q,d}(t)=\sqrt{N_{q,d}}P_{q,d}(t)$ and define $\gamma_{q,d}
\triangleq
\E\big[\kappa_d(T_d)\psi_{q,d}(T_d)\big].$
The polynomials $\{\psi_{q,d}\}_{q\ge0}$ form an orthonormal basis
with respect to the law of $T_d$. Since $\kappa_d$ is centered,
its coefficient of degree zero vanishes, and
\begin{equation}
\kappa_d(t)=\sum_{q\ge1}\gamma_{q,d}\psi_{q,d}(t),
\quad
v_d=\sum_{q\ge1}\gamma_{q,d}^2,
\label{eq:harmonic-expansion}
\end{equation}
where the expansion converges in mean square and the second identity
follows from Parseval's identity.
By the Funk--Hecke formula (see, e.g., ~\cite[Chapter~1, Funk--Hecke formula]{dai2013approximation}),
\[
T_{\kappa_d}f
=
\E\big[\kappa_d(T_d)P_{q,d}(T_d)\big]f
=
\frac{\gamma_{q,d}}{\sqrt{N_{q,d}}}f,
\quad f\in\maH_{q,d}.
\]
Thus degree $q$ contributes $N_{q,d}$ copies of the eigenvalue
$\gamma_{q,d}/\sqrt{N_{q,d}}$ to the spectrum.
Since the spherical harmonic spaces form an orthogonal decomposition
of $L^2(\sigma_d)$, summing the $\ell$th powers of these eigenvalues gives
\begin{equation}
\tau_{\ell,d}=\tr(T_{\kappa_d}^{\ell})
=\sum_{q\ge1}N_{q,d}
\Big(\frac{\gamma_{q,d}}{\sqrt{N_{q,d}}}\Big)^\ell
=\sum_{q\ge1}\gamma_{q,d}^{\ell}N_{q,d}^{1-\ell/2},
\quad \ell\ge2.
\label{eq:trace-formula}
\end{equation}

We are now ready to combine the two tests into a single detection criterion.

\begin{corollary}
\label{cor:cycle-detection}
Suppose that $\E[H_d(T_d)^4]\lesssim \tau_{4,d}^2$.
If $n^3\tau_{3,d}^2+n^4\tau_{4,d}^2\to\infty$, then
$\TV(P,Q)\to1$. For each $(n,d)$, use $\phi_3$ when
$n^3\tau_{3,d}^2\ge n^4\tau_{4,d}^2$ and $\phi_4$ otherwise.
\end{corollary}

\begin{proof}
If $n^3\tau_{3,d}^2\ge n^4\tau_{4,d}^2$, use the midpoint test for $\sgn(\tau_{3,d})S_3$.
Chebyshev's inequality and Proposition~\ref{prop:triangle-moments} give
\begin{align}
 \P_Q[\phi_3=1]+\P_P[\phi_3=0]
 &\le C_K\Big(\frac1{n^3\tau_{3,d}^2}
       +\frac{\tau_{4,d}}{n^2\tau_{3,d}^2}\Big)
 \le C_K\Big(\frac1{n^3\tau_{3,d}^2}
       +\frac1{n^{5/2}|\tau_{3,d}|}\Big).
 \label{eq:triangle-error}
\end{align}
If $n^4\tau_{4,d}^2>n^3\tau_{3,d}^2$, use the midpoint test for $S_4$, whose alternative
mean is positive. The variance bound for cycles of length four and
the assumed bound on the fourth moment yield
\begin{equation}
 \P_Q[\phi_4=1]+\P_P[\phi_4=0]
 \le C_K\Big(
 \frac1{n^4\tau_{4,d}^2}+\frac1{n^3\tau_{4,d}}
 +\frac{\tau_{6,d}}{n^2\tau_{4,d}^2}+\frac1{n^2}\Big).
 \label{eq:square-error}
\end{equation}
By~\eqref{eq:trace-formula}, every nonconstant eigenvalue has
multiplicity at least $d$, so
$\|T_{\kappa_d}\|_{\op}^4\le\tau_{4,d}/d$. Also
$\tau_{6,d}\le\|T_{\kappa_d}\|_{\op}^2\tau_{4,d}$. Therefore the right side
of \eqref{eq:square-error} is at most
\[
 C_K\Big(\frac1{n^4\tau_{4,d}^2}+\frac1{n^3\tau_{4,d}}
           +\frac1{n^2\sqrt{d\tau_{4,d}}}+\frac1{n^2}\Big).
\]
When $n^3\tau_{3,d}^2+n^4\tau_{4,d}^2\to\infty$, the larger of $n^3\tau_{3,d}^2$ and $n^4\tau_{4,d}^2$ diverges.
Select the corresponding test for each $(n,d)$. The bounds in~\eqref{eq:triangle-error} and~\eqref{eq:square-error}
show that its two errors tend to zero, even if the selected
branch changes along the sequence. Hence
$\TV(P,Q)\to1$.
\end{proof}

\subsection{Possibility results for cusp kernels}
\label{sec:cusp-spectrum}

We apply Corollary~\ref{cor:cycle-detection} to the four cases in
Definition~\ref{def:cusp}. For continuous cusps, the local
expansions are sampled at the scale $T_d=O_{\P}(d^{-1/2})$.
Let $G\sim N(0,1)$ be standard normal and define $M_t=\E[|G|^t]$.
The limit $\sqrt dT_d\Rightarrow G$, together with the moment
bounds in Lemma~\ref{gen:lem:gap-transfer}, gives
\begin{equation}
 s_d\gamma_{1,d}
 =a_oM_{\beta+1}d^{-\beta/2}(1+o(1)),\quad
 s_d\gamma_{2,d}
 =\frac{a_e\alpha M_\alpha}{\sqrt2}d^{-\alpha/2}(1+o(1)).
 \label{eq:cusp-low-coefficients}
\end{equation}
The first relation applies to pure odd and mixed cusps, and the
second to pure even and mixed cusps. The factor $\alpha M_\alpha$
comes from $\E[|G|^\alpha(G^2-1)]=\alpha M_\alpha$.
For a jump, the signed difference condition instead gives
$|\gamma_{1,d}|\asymp 1$ and $v_d\asymp 1$.
Thus the relevant degrees are one for the odd part and two for
the centered even part. Their multiplicities in~\eqref{eq:trace-formula}
give the following trace scales.

\begin{proposition}[Cycle traces for cusp kernels]
\label{gen:prop:spectral-scales}
Let $K\in\maK_{\mathrm{cusp}}$ be fixed.
\begin{itemize}
\item \emph{Pure even.} We have
$\tau_{3,d}^2\asymp d^{-3\alpha-2}$,
$\tau_{4,d}\asymp d^{-2\alpha-2}$, and
$\|T_{\kappa_d}\|_{\op}^2\asymp d^{-\alpha-2}$.
\item \emph{Pure odd.} We have
$\tau_{3,d}^2\asymp d^{-3\beta-1}$,
$\tau_{4,d}\asymp d^{-2\beta-1}$, and
$\|T_{\kappa_d}\|_{\op}^2\asymp d^{-\beta-1}$.
\item \emph{Mixed.} We have
\begin{equation}
 \tau_{4,d}\asymp d^{-2\beta-1}+d^{-2\alpha-2},\quad
 \|T_{\kappa_d}\|_{\op}^2\asymp d^{-\beta-1}+d^{-\alpha-2}.
 \label{gen:eq:even-trace-scales}
\end{equation}
If $\beta\ne\alpha+1/3$, then, with $s$ defined
in~\eqref{eq:cusp-dimension},
\begin{equation}
 \tau_{3,d}^2\asymp d^{-3/s}.
 \label{gen:eq:nonbalanced-cubic}
\end{equation}
If $\beta=\alpha+1/3$, then
\begin{equation}
 s_d^3d^{(3\alpha+2)/2}\tau_{3,d}
 =a_o^3M_{\beta+1}^3+\frac{(a_e\alpha M_\alpha)^3}{2}+o(1).
 \label{eq:cusp-critical-coefficient}
\end{equation}
\item \emph{Jump.} We have
$\tau_{3,d}^2\asymp \tau_{4,d}
\asymp \|T_{\kappa_d}\|_{\op}^2\asymp d^{-1}$.
\end{itemize}
\end{proposition}

The proof is in Appendix~\ref{app:cusp-spectrum}.
The first odd and even cubic terms have orders
$d^{-(3\beta+1)/2}$ and $d^{-(3\alpha+2)/2}$, respectively.
The leading triangle signal disappears when these orders agree
and the two constants in~\eqref{eq:cusp-critical-coefficient} sum
to zero.
The fourth trace has nonnegative spectral terms and remains
of order $d^{-2\alpha-5/3}$ on this relation.
Section~\ref{sec:cusp-converse} determines the resulting sharp
thresholds, including the effect of the local remainders.

To apply Corollary~\ref{cor:cycle-detection}, it remains to verify
its condition on the fourth moment. 
By~\eqref{eq:H-moments} in Lemma~\ref{lem:contracted},
\[
 \|H_d(T_d)\|_4
 \lesssim \frac{\gamma_{1,d}^2}{\sqrt d}+\frac{v_d}{d}
 \overset{\mathrm{(a)}}{\lesssim} \sqrt{\tau_{4,d}}.
\]
Here $\mathrm{(a)}$ uses Lemma~\ref{gen:lem:gap-transfer} and
Proposition~\ref{gen:prop:spectral-scales} for the continuous cases;
for jumps, $v_d\asymp\gamma_{1,d}^2\asymp1$ and
$\tau_{4,d}\asymp d^{-1}$.
Hence the selection rule in Corollary~\ref{cor:cycle-detection}
achieves $\TV(P,Q)\to1$ whenever
$n^3\tau_{3,d}^2+n^4\tau_{4,d}^2\to\infty$.
This proves the detection assertion of Theorem~\ref{thm:cusp}.
For a continuous cusp away from the critical relation,
$\tau_{4,d}=o(d^{-1/s})$. Thus $n\gg d^{1/s}$ implies both
$n^3\tau_{3,d}^2\to\infty$ and
$n^2\tau_{3,d}^2/\tau_{4,d}\to\infty$, so
Proposition~\ref{prop:triangle-moments} shows that signed triangles
alone achieve detection for $d\ll n^s$.
The same argument gives $d\ll n^3$ for jumps.
On the critical relation, the criterion retains both traces and
allows signed cycles of length four to detect a signal suppressed in the
triangle mean.

\subsection{Possibility results for smooth kernels}

We now consider $K\in\maK_{\mathrm{smooth}}(m)$.
The leading nonconstant Taylor term is proportional to $t^m$.
Specifically, write
\[
t^m=\sum_{i=0}^m d_{m,i}P_{i,d}(t),
\quad
d_{m,i}=N_{i,d}\E\big[T_d^mP_{i,d}(T_d)\big].
\]
By symmetry, $d_{m,i}=0$ whenever $m-i$ is odd.
Centering removes the component of degree zero, leaving degrees
$m,m-2,\ldots,q_m$.

To determine the coefficient scales, the Gaussian representation
of a uniform point on the sphere and Slutsky's theorem give
$\sqrt dT_d\Rightarrow G\sim N(0,1)$.
Hence, $T_d^m-\E[T_d^m]$ has $L^2$ norm of order $d^{-m/2}$.
Under the same rescaling, the normalized Gegenbauer polynomials
converge to normalized Hermite polynomials
(see Lemma~\ref{lem:gegen-hermite} for details).
The corresponding Gaussian projections are nonzero at degrees
$m,m-2,\ldots,q_m$. Combining these facts with Taylor's theorem
and convergence of the relevant moments gives $v_d\asymp d^{-m}$
and $|\gamma_{q,d}|\asymp d^{-m/2}$ for
$q\in\{m,m-2,\ldots,q_m\}$.

For even $m$, the coefficient of degree two has this leading scale.
Only Taylor terms of higher order contribute to degree one, and
the Taylor remainder gives
$\gamma_{1,d}=O(d^{-(m+1)/2})$.
Since $N_{1,d}=d$ and $N_{2,d}\asymp d^2$,
\eqref{eq:trace-formula} shows that the contribution of degree one
to the cubic trace is smaller in absolute value than the contribution of degree two by a factor $O(d^{-1})$.
The energy bound and $N_{q,d}\ge N_{3,d}\asymp d^3$ control the
remaining degrees. Thus degree two dominates, giving
$|\tau_{3,d}|\asymp d^{-(3m+2)/2}$.
For odd $m$, degree one already has a nonzero leading coefficient,
with $|\gamma_{1,d}|\asymp d^{-m/2}$.
Using~\eqref{eq:trace-formula} and the energy bound to control
degrees $q\ge2$, the term of degree one dominates, giving
$|\tau_{3,d}|\asymp d^{-(3m+1)/2}$.
The following proposition records both cases.

\begin{proposition}
\label{prop:spectrum}
Fix an integer $m\ge2$ and $K\in\maK_{\mathrm{smooth}}(m)$.
As $d\to\infty$, we have $|\gamma_{q,d}|\asymp d^{-m/2}$ for $q=m,m-2,\ldots,q_m$.
For every fixed integer $\ell\ge3$,
\begin{equation}
\tau_{\ell,d}
=\gamma_{q_m,d}^\ell N_{q_m,d}^{1-\ell/2}(1+O(d^{-1/2})).
\label{eq:prop-finite-positive-flat-1}
\end{equation}
In particular,
\begin{equation}
|\tau_{3,d}|\asymp d^{-(3m+q_m)/2},\quad
\tau_{4,d}\asymp d^{-(2m+q_m)}.
\label{eq:prop-finite-positive-flat-2}
\end{equation}
\end{proposition}

The proof is in Appendix~\ref{app:proof-spectrum}.
For detection in Theorem~\ref{thm:finite-flat}, suppose
$d\ll n^{3/(3m+q_m)}$. Both $n^3d^{-(3m+q_m)}$ and $n^2d^{-m}$
then diverge. By~\eqref{eq:prop-finite-positive-flat-2},
\[
 \frac{n^6\tr(T_{\kappa_d}^3)^2}
 {n^3+n^4\tr(T_{\kappa_d}^4)}
 \asymp n^3d^{-(3m+q_m)}\wedge n^2d^{-m}\longrightarrow\infty,
\]
so Proposition~\ref{prop:triangle-moments} applies.

\begin{remark}[Triangles versus cycles of length four]\label{rmk:cancel}
We compare the leading contributions to the triangle signal
in the two kernel classes.
\begin{itemize}
\item \emph{Cusp kernels.}
Recall from~\eqref{eq:tri-planted} that
$\E_P[S_3]=\binom n3\tau_{3,d}$, so $\tau_{3,d}$ determines
the mean of the signed triangle count.
By~\eqref{eq:trace-formula},
\[
 \tau_{3,d}
 =\frac{\gamma_{1,d}^3}{\sqrt d}
 +\frac{\gamma_{2,d}^3}{\sqrt{N_{2,d}}}
 +\sum_{q\ge3}\frac{\gamma_{q,d}^3}{\sqrt{N_{q,d}}}.
\]
For a mixed cusp, \eqref{eq:cusp-low-coefficients} gives
\begin{align*}
 \frac{s_d^3\gamma_{1,d}^3}{\sqrt d}
 &=a_o^3M_{\beta+1}^3d^{-(3\beta+1)/2}
   +o(d^{-(3\beta+1)/2}),\\
 \frac{s_d^3\gamma_{2,d}^3}{\sqrt{N_{2,d}}}
 &=\frac{(a_e\alpha M_\alpha)^3}{2}d^{-(3\alpha+2)/2}
   +o(d^{-(3\alpha+2)/2}).
\end{align*}
The remaining degrees contribute a smaller order than the larger
of these two terms.
If $\beta<\alpha+1/3$, the first term dominates; if
$\beta>\alpha+1/3$, the second dominates.
If $\beta=\alpha+1/3$, the two terms have the same order,
and their leading coefficients cancel precisely when
\[
 a_o^3M_{\alpha+4/3}^3+\frac{(a_e\alpha M_\alpha)^3}{2}=0.
\]
In contrast, the fourth trace satisfies
\[
 \tau_{4,d}
 =\frac{\gamma_{1,d}^4}{d}
 +\frac{\gamma_{2,d}^4}{N_{2,d}}
 +\sum_{q\ge3}\frac{\gamma_{q,d}^4}{N_{q,d}}.
\]
Every term is nonnegative, so no such cancellation occurs.
On this relation, $\tau_{4,d}\asymp d^{-2\alpha-5/3}$.
After the leading terms cancel, the local remainders
determine whether cycles of length four dominate.

\item \emph{Finitely flat smooth kernels.}
For $K\in\maK_{\mathrm{smooth}}(m)$, the first nonconstant
Taylor term has the parity of $m$, whereas the opposite parity
can first appear at order $m+1$.
Consequently, for odd $m$,
\[
 \frac{|\gamma_{1,d}|^3}{\sqrt d}
 \asymp d^{-(3m+1)/2},
 \quad
 \frac{|\gamma_{2,d}|^3}{\sqrt{N_{2,d}}}
 =O(d^{-(3m+5)/2}),
\]
and for even $m$,
\[
 \frac{|\gamma_{1,d}|^3}{\sqrt d}
 =O(d^{-(3m+4)/2}),
 \quad
 \frac{|\gamma_{2,d}|^3}{\sqrt{N_{2,d}}}
 \asymp d^{-(3m+2)/2}.
\]
Thus degree one dominates for odd $m$, and degree two dominates
for even $m$. The remaining degrees are also smaller, as in
Proposition~\ref{prop:spectrum}.
Hence the leading cubic contribution cannot cancel, and signed
triangle counts attain the optimal detection threshold.
\end{itemize}
\end{remark}

\section{Posterior cavity reduction}\label{sec:cavity}

In this section, we prove a general KL bound for the impossibility
results. To show $\TV(P,Q)=o(1)$, it suffices to prove
$\KL(P\Vert Q)=o(1)$ by Pinsker's inequality
$\TV(P,Q)^2\le\KL(P\Vert Q)/2$ (see, e.g.,~\cite[Chapter~7]{polyanskiy2025information}).
Following~\cite[Section~3]{mao2026random}, we expand the KL
divergence by adding one vertex at a time and observing its edges to the existing graph.
Recall that $A^{(r)}$ is the graph induced by $[r]$, and define
$B^{(r+1)}=(A_{1,r+1},\ldots,A_{r,r+1})$.
At step $r+1$, we observe $B^{(r+1)}$ conditional on $A^{(r)}$ (see Figure~\ref{fig:sequential-kl} for an illustration).

\begin{figure}[htbp]
\centering
\begingroup
\definecolor{klold}{RGB}{20,103,116}
\definecolor{klnew}{RGB}{175,83,43}
\definecolor{klink}{RGB}{44,49,55}
\resizebox{0.88\textwidth}{!}{%
\begin{tikzpicture}[
  x=1cm,y=1cm,
  line cap=round,line join=round,
  every node/.style={text=klink},
  old block/.style={
    fill=klold!15,draw=klold,line width=.8pt
  },
  new block/.style={
    fill=klnew!23,draw=klnew,line width=.8pt
  },
  next step/.style={
    -{Stealth[length=5pt]},draw=klink!70,line width=.8pt
  }
]
\node at (0,-1.15) {$\cdots$};

\begin{scope}[shift={(.7,0)}]
  \path[old block] (0,0) rectangle (2,-2);
  \node at (1,.45) {$A^{(r)}$};
\end{scope}

\draw[next step] (2.95,-1.15)--(4.2,-1.15);
\node[font=\small] at (3.575,-.75) {$B^{(r+1)}$};

\begin{scope}[shift={(4.45,0)}]
  \path[old block] (0,0) rectangle (2,-2);
  \path[new block]
    (2,0)--(2.45,0)--(2.45,-2.45)
    --(0,-2.45)--(0,-2)--(2,-2)--cycle;
  \node at (1.225,.45) {$A^{(r+1)}$};
\end{scope}

\draw[next step] (7.15,-1.15)--(8.4,-1.15);
\node[font=\small] at (7.775,-.75) {$B^{(r+2)}$};

\begin{scope}[shift={(8.65,0)}]
  \path[old block] (0,0) rectangle (2.45,-2.45);
  \draw[klold!40,line width=.5pt]
    (2,0)--(2,-2.45)
    (0,-2)--(2.45,-2);
  \path[new block]
    (2.45,0)--(2.9,0)--(2.9,-2.9)
    --(0,-2.9)--(0,-2.45)--(2.45,-2.45)--cycle;
  \node at (1.45,.45) {$A^{(r+2)}$};
\end{scope}
\node at (12.2,-1.15) {$\cdots$};

\path[old block] (3.25,-3.50) rectangle (3.60,-3.22);
\node[anchor=west,font=\small] at (3.75,-3.36)
  {Already observed};
\path[new block] (7,-3.50) rectangle (7.35,-3.22);
\node[anchor=west,font=\small] at (7.50,-3.36)
  {Newly observed};

\end{tikzpicture}%
}
\endgroup
\caption{Sequential observation of the adjacency matrix.
Each new row and column represent the same edge variables.}
\label{fig:sequential-kl}
\end{figure}
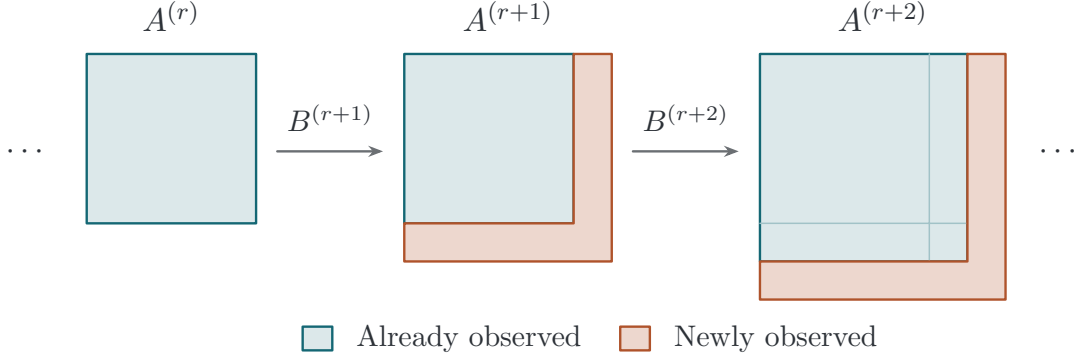

Under $Q$, the new edges are independent of $A^{(r)}$ and have
law $\Bern(p_d)^{\otimes r}$.
The chain rule and $\KL\le\chi^2$, where $\chi^2$ denotes
the chi-squared divergence, give
\begin{align}
 \KL(P\Vert Q)
 &=\sum_{r=1}^{n-1}\E_P\Big[
 \KL\big(P(B^{(r+1)}\mid A^{(r)})
 \Vert\Bern(p_d)^{\otimes r}\big)\Big]\notag\\
 &\le\sum_{r=1}^{n-1}
 \sum_{\substack{S\subseteq[r]\\|S|\ge2}}
 \E_P\Big[\E_P\Big[
 \prod_{i\in S}\kappa_d(\langle X_i,X_{r+1}\rangle)
 \mid A^{(r)}\Big]^2\Big].
 \label{eq:cavity-chain-expansion}
\end{align}
To obtain the last inequality, fix $b\in\{0,1\}^r$ and let
$y_i(b)=(b_i-p_d)/s_d$.
Given the latent points, the likelihood ratio for edge $\{i,r+1\}$ is
$1+\kappa_d(\langle X_i,X_{r+1}\rangle)y_i(b)$.
Conditional independence therefore gives
\begin{align*}
 \frac{P(B^{(r+1)}=b\mid A^{(r)})}{Q(B^{(r+1)}=b)}
 &=\E_P\Big[\prod_{i=1}^r
 \big(1+\kappa_d(\langle X_i,X_{r+1}\rangle)y_i(b)\big)
 \mid A^{(r)}\Big]\\
 &=1+\sum_{\substack{S\subseteq[r]\\|S|\ge2}}
 \E_P\Big[\prod_{i\in S}
 \kappa_d(\langle X_i,X_{r+1}\rangle)\mid A^{(r)}\Big]
 \prod_{i\in S}y_i(b).
\end{align*}
The terms with $|S|=1$ vanish by centering.
For $b\sim\Bern(p_d)^{\otimes r}$, the variables $y_i(b)$ are
independent with mean zero and variance one. Hence
\[
 \E_{b\sim\Bern(p_d)^{\otimes r}}\Big[
 \prod_{i\in S}y_i(b)\prod_{j\in T}y_j(b)\Big]
 =\mathbf{1}_{\{S=T\}},\quad S,T\subseteq[r].
\]
Thus squaring the likelihood ratio minus one and averaging over $b$
removes all cross terms, giving~\eqref{eq:cavity-chain-expansion}.

We now group the terms by $k=|S|$.
Since $X_{r+1}$ is independent of $(A^{(r)},X_1,\ldots,X_r)$,
we can first integrate out this new point. For $k\ge1$, define
\begin{equation}
 h_{k,d}(x_1,\ldots,x_k)
 =\int_{\Sph^{d-1}}\prod_{i=1}^k
 \kappa_d(\langle x_i,z\rangle)\sigma_d(dz).
 \label{eq:star-kernel}
\end{equation}
Centering gives $h_{1,d}=0$ and
$\int h_{k,d}(x_1,\ldots,x_k)\sigma_d(dx_i)=0$ for each $i$.
We call a kernel with this property canonical in each argument.
For $S=\{i_1<\cdots<i_k\}$, let $X_S=(X_{i_1},\ldots,X_{i_k})$.
The inner expectation in~\eqref{eq:cavity-chain-expansion} is then
$\E_P[h_{k,d}(X_S)\mid A^{(r)}]$.
After averaging over $A^{(r)}$, vertex exchangeability makes all
subsets of size $k$ contribute equally. There are $\binom rk$
such subsets. We therefore define
\begin{equation}
 g_{r,k}=\E_P\Big[
 \E_P[h_{k,d}(X_1,\ldots,X_k)\mid A^{(r)}]^2\Big],
 \quad 2\le k\le r\le n.
 \label{eq:4-2}
\end{equation}
This quantity measures how well the observed graph predicts
$h_{k,d}$. To collect the contributions of order $k$ over all
exposure steps, define
\begin{equation}
 G_k=\sum_{r=k}^{n-1}\binom rk g_{r,k},\quad
 G_{\ge k_0}=\sum_{k=k_0}^{n-1}G_k.
 \label{eq:cavity-costs}
\end{equation}
We also write $\bar g_k=g_{n,k}$ and
$\bar G_k=\binom n{k+1}\bar g_k$.
Conditioning on the full graph gives $g_{r,k}\le g_{n,k}$.
Together with $\sum_{r=k}^{n-1}\binom rk=\binom n{k+1}$,
this yields $G_k\le\bar G_k$.
The following lemma summarizes the KL bound and its direct
comparison with prior moments.

\begin{lemma}[Sequential cavity inequality]\label{lem:cavity}
For every dense spherical latent kernel graph with its centered kernel,
\begin{equation}
 \KL(P\Vert Q)
 \le\sum_{r=1}^{n-1}\sum_{k=2}^r\binom rk g_{r,k}
 =\sum_{k=2}^{n-1}G_k\le\sum_{k=2}^{n-1}\bar G_k.
 \label{eq:cavity-double}
\end{equation}
Moreover, with $H_d$ defined in~\eqref{eq:contracted-profile},
\begin{equation}
 g_{r,k}\le\E[h_{k,d}(X_1,\ldots,X_k)^2]
 =\E[H_d(T_d)^k],\quad 2\le k\le r\le n.
 \label{eq:4-4}
\end{equation}
\end{lemma}

\begin{proof}
Grouping~\eqref{eq:cavity-chain-expansion} by $|S|$ and using
\eqref{eq:4-2} and~\eqref{eq:cavity-costs} gives
\eqref{eq:cavity-double}, since $G_k\le\bar G_k$.
For~\eqref{eq:4-4}, conditional Jensen's inequality gives
$g_{r,k}\le\E[h_{k,d}(X_1,\ldots,X_k)^2]$.
Let $Z,Z'\sim\sigma_d$ be independent of each other and of
$X_1,\ldots,X_k$.
By~\eqref{eq:star-kernel}, Fubini's theorem and independence give
\begin{align*}
 \E[h_{k,d}(X_1,\ldots,X_k)^2]
 &=\E_{Z,Z'}\Big[
 \Big(\int\kappa_d(\langle x,Z\rangle)
 \kappa_d(\langle x,Z'\rangle)\sigma_d(dx)\Big)^k\Big]\\
 &=\E[H_d(\langle Z,Z'\rangle)^k]
 =\E[H_d(T_d)^k],
\end{align*}
where the last line uses~\eqref{eq:contracted-profile}
and $\langle Z,Z'\rangle\stackrel{d}{=}T_d$.
\end{proof}

We next bound the terms $g_{r,k}$ in~\eqref{eq:cavity-double}.
The prior comparison~\eqref{eq:4-4} already bounds each term by
$\E[H_d(T_d)^k]$. To estimate these moments, we first examine
$H_d$. Recall its definition in~\eqref{eq:contracted-profile}
and the harmonic expansion~\eqref{eq:harmonic-expansion}.
The Funk--Hecke formula and orthogonality give
\[
 \int_{\Sph^{d-1}}
 \psi_{q,d}(\langle x,z\rangle)
 \psi_{j,d}(\langle z,y\rangle)\sigma_d(\dd z)
 =\indc{q=j}P_{q,d}(\langle x,y\rangle).
\]
Substituting the harmonic expansion into both factors of
\eqref{eq:contracted-profile} therefore yields
\begin{equation}
 H_d(t)=\sum_{q\ge1}\gamma_{q,d}^2P_{q,d}(t).
 \label{eq:contracted-kernel}
\end{equation}
This calculation follows first for finite sums and then by $L^2$
convergence and Cauchy--Schwarz. Since $P_{1,d}(t)=t$, we separate
the degree one term and bound the remaining degrees uniformly.

\begin{lemma}[A uniform contracted kernel bound]\label{lem:contracted}
For $d\ge3$, $t\in[-1,1]$, and every integer $k\ge2$,
\begin{align}
 |H_d(t)-\gamma_{1,d}^2t|
 &\le v_d\Big(t^2+\frac{1-t^2}{d-1}\Big),\quad 
 \|H_d(T_d)\|_k
\le C\Big(\gamma_{1,d}^2\sqrt{\frac{k}{d}}+
                      v_d\frac{k}{d}\Big),\label{eq:H-moments}
\end{align}
where $C$ is an absolute constant.
For every fixed $K\in\maK_{\mathrm{cusp}}$, we also have
$\E[H_d(T_d)^4]\le C_K\tau_{4,d}^2$ for all sufficiently large $d$.
\end{lemma}

\begin{proof}
Let $V$ be the first coordinate of a uniform point on $\Sph^{d-2}$.
The Laplace representation~\cite[Theorem 6.7.4]{Andrews1999special} gives
\[
 P_{q,d}(t)=\E\big[
   \big(t+\mathrm{i}\sqrt{1-t^2}V\big)^q\big].
\]
Since $t^2+(1-t^2)V^2\le1$ and $\E[V^2]=1/(d-1)$, for $q\ge2$,
\[
 |P_{q,d}(t)|
 \le\E\big[\big(t^2+(1-t^2)V^2\big)^{q/2}\big]
 \le t^2+\frac{1-t^2}{d-1}.
\]
Summing against $\gamma_{q,d}^2$ in~\eqref{eq:contracted-kernel}
and using $\sum_{q\ge2}\gamma_{q,d}^2\le v_d$ proves
\eqref{eq:H-moments}. The spherical tail bound
\eqref{eq:kernel-bridge-sphere-tail} gives
$\|T_d\|_r\le C\sqrt{r/d}$ for $r\ge2$; for $r\ge d$,
this also follows from $|T_d|\le1$. Hence Minkowski's inequality gives
\begin{align*}
 \|H_d(T_d)\|_k
 &\le\gamma_{1,d}^2\|T_d\|_k
       +v_d\Big(\|T_d\|_{2k}^2+\frac1{d-1}\Big)\le C\Big(\gamma_{1,d}^2\sqrt{\frac{k}{d}}+
                  v_d\frac{k}{d}\Big),
\end{align*}
which proves~\eqref{eq:H-moments}.

For continuous cusps, Lemma~\ref{gen:lem:gap-transfer} gives
$v_d\le C_K(\gamma_{1,d}^2+\gamma_{2,d}^2)$ for all sufficiently
large $d$. Recall from~\eqref{eq:trace-formula} that
$\tau_{4,d}\ge\gamma_{1,d}^4/d+\gamma_{2,d}^4/N_{2,d}$,
where $N_{2,d}\asymp d^2$. Thus
\[
 \frac{\gamma_{1,d}^2}{\sqrt d}+\frac{v_d}{d}
 \le C_K\Big(\frac{\gamma_{1,d}^2}{\sqrt d}
             +\frac{\gamma_{2,d}^2}{d}\Big)
 \le C_K\sqrt{\tau_{4,d}}.
\]
For jumps, the same bound follows from $v_d\le1$ and
$|\gamma_{1,d}|\asymp1$, $\tau_{4,d}\asymp d^{-1}$ in
Proposition~\ref{gen:prop:spectral-scales}.
Taking $k=4$ in~\eqref{eq:H-moments} proves the final assertion.
\end{proof}

We now apply these moment bounds to the cavity costs.
By~\eqref{eq:4-4} and the definition of $\bar G_k$,
\begin{equation}
 G_k\le\bar G_k
 \le\binom n{k+1}\E[H_d(T_d)^k]
 \le\binom n{k+1}\E\big[|H_d(T_d)|^k\big].
 \label{eq:prior-cavity-cost}
\end{equation}
For cusp kernels, Lemma~\ref{lem:contracted} therefore gives
\begin{equation}
 \bar G_k\le\binom n{k+1}
 \big[C\big(\gamma_{1,d}^2\sqrt{k/d}+v_dk/d\big)\big]^k.
 \label{eq:general-prior-cavity}
\end{equation}
Away from cubic cancellation, this controls all orders $k\ge10$
at the threshold scale. On the critical relation $\beta=\alpha+1/3$,
the sample size can be larger. The same prior estimate then controls
the orders above $C_K\log d$, and Section~\ref{sec:cusp-converse}
bounds the retained posterior terms.
For smooth kernels, \eqref{eq:growing-moment} gives
$\bar G_k\lesssim_k n^{k+1}d^{-k(m+q_m/2)}$ at each fixed $k$.
Proposition~\ref{prop:cavity-tail} controls the sum over
$k>k_\star$ for any fixed $k_\star>6m/q_m+2$.
Thus both kernel families require sharper posterior estimates at low
orders, while their growing order bounds use different features
of the kernel.


At low orders, we retain the conditional averaging in $g_{r,k}$
by representing squared conditional means with two posterior samples.
First draw $A\sim P$ and then draw
$X,X'$ conditionally independently from the posterior given $A$.
We call these samples \emph{posterior replicas} and denote their
joint law, after averaging over $A$, by $\P_{\post}$, with
expectation $\E_{\post}$.
For every latent statistic $f$ with finite second moment,
conditional independence gives
\[
\E_{A\sim P}\Big[\E[f(X)\mid A]^2\Big]
=
\E_{\post}[f(X)f(X')].
\]
Taking $f(X)=h_{k,d}(X_1,\ldots,X_k)$ represents $\bar g_k$;
the same construction with $A^{(r)}$ represents $g_{r,k}$.

We estimate these expectations through posterior tail probabilities.
To use concentration under the independent prior, we express these
probabilities by Bayes' rule. With $Q$ as the reference measure,
define the conditional and marginal graph likelihoods
\[
L_X(A)
=
\frac{\mathrm dP(\cdot\mid X)}{\mathrm dQ}(A),
\quad
L(A)
=
\E_X[L_X(A)]
=
\frac{\mathrm dP}{\mathrm dQ}(A),
\]
where $\E_X$ averages over $X\sim\sigma_d^{\otimes n}$.
The posterior density relative to this prior is $L_X(A)/L(A)$.
Each of the two replicas contributes one such factor, while
averaging $A$ under $P$ contributes a factor $L(A)$.
Consequently, for any event $\mathcal E$ depending only on $(X,X')$,
\begin{equation}
\P_{\post}[\mathcal E]
=
\E_Q\Bigg[
\frac{
\E_{X,X'}[L_X(A)L_{X'}(A)\indc{\mathcal E}]
}{L(A)}
\Bigg],
\label{eq:posterior-event-bayes}
\end{equation}
where $X,X'$ on the right side are independent prior
configurations.

To use concentration for the independent prior pair, we first remove
the denominator $L(A)$. The numerator in~\eqref{eq:posterior-event-bayes}
is bounded by $L(A)^2$, which yields the square root comparison below.
The remaining weight is the expectation under $Q$ of the product of the
two conditional likelihoods.
Define the \emph{likelihood overlap}
\begin{equation}
\Gamma_d(X,X')
:=
\E_Q[L_X(A)L_{X'}(A)]
=
\prod_{i<j}
\big\{1+\kappa_d(T_{ij})\kappa_d(T_{ij}')\big\},
\label{eq:4-5}
\end{equation}
where $T_{ij}=\langle X_i,X_j\rangle$ and
$T_{ij}'=\langle X_i',X_j'\rangle$.
The product formula follows from independence of the edges
under $Q$.
We record the resulting comparison in the following lemma.

\begin{lemma}
\label{lem:planting}
For every measurable event $\mathcal E$ depending only on
the latent pair $(X,X')$,
\begin{equation}
\P_{\post}[\mathcal E]
\le
\sqrt{
\E_{X,X'}\big[
\Gamma_d(X,X')\indc{\mathcal E}
\big]
}.
\label{eq:4-6}
\end{equation}
Here $X,X'$ on the right side are independent with law
$\sigma_d^{\otimes n}$.
Moreover, the marginal law of either replica under $\P_{\post}$
is $\sigma_d^{\otimes n}$.
\end{lemma}

\begin{proof}
    Set $ M_{\mathcal E}(A)
 =\E_{X,X'}[L_X(A)L_{X'}(A)\1_{\mathcal E}].$
Bayes' rule gives $ \P_{\post}[\mathcal E]
 =\E_{A\sim Q}\Big[\tfrac{M_{\mathcal E}(A)}{L(A)}\Big]. $
Since $\1_{\mathcal E}\le1$,
$0\le M_{\mathcal E}(A)\le L(A)^2$; hence
$M_{\mathcal E}(A)/L(A)\le\sqrt{M_{\mathcal E}(A)}$.
By Jensen's inequality, 
\[
 \P_{\post}[\mathcal E]
 \le\sqrt{\E_Q[M_{\mathcal E}(A)]}
 =\sqrt{\E_{X,X'}[\Gamma_d(X,X')\1_{\mathcal E}]},
\]
which is~\eqref{eq:4-6}. For every bounded $F$, the tower property gives
$\E_{A\sim P}[\E[F(X)\mid A]]=\E_X[F(X)]$.
Thus an event depending on one replica has the same probability
under the averaged posterior and the prior.
\end{proof}
For the smooth class, we will apply~\eqref{eq:4-6} to tail events of replica statistics,
bound the corresponding weighted prior expectations, and
integrate the resulting posterior tail bounds. Section~\ref{sec:polynomial} carries this out using a finite Taylor expansion.
The cusp argument in Section~\ref{sec:cusp-converse} uses the same
conditional likelihoods and bounds negative moments of their likelihood denominators using
reverse hypercontractivity.

\section{Impossibility results for cusp kernels}
\label{sec:cusp-converse}\label{gen:sec:converse}

In this section, we prove the impossibility results for
$K\in\maK_{\mathrm{cusp}}$ in Theorem~\ref{thm:cusp}.
By Pinsker's inequality and Lemma~\ref{lem:cavity},
$\TV(P,Q)^2\le\KL(P\Vert Q)/2
\le\frac12\sum_{k=2}^{n-1}G_k$.
It therefore suffices to show that this sum is $o(1)$.
Recall that $\Lambda_{n,d}=n^3\tau_{3,d}^2+n^4\tau_{4,d}^2$
combines the signals from signed triangle and quadrilateral counts. The following proposition gives the required bound in terms of
$\Lambda_{n,d}$.

\begin{proposition}[KL bound for cusp kernels]\label{prop:cusp-KL}
For every $K\in\maK_{\mathrm{cusp}}$, there are constants
$c_{\ref{prop:cusp-KL}}(K),C_K,d_K>0$ such that, for $d\ge d_K$ and $n\ge2$,
\begin{equation}
 \Lambda_{n,d}\le c_{\ref{prop:cusp-KL}}(K)
 \quad\Longrightarrow\quad
 \KL(P\Vert Q)\le C_K\Lambda_{n,d}.
 \label{eq:cusp-KL}
\end{equation}
\end{proposition}

For continuous cusps, let $\epsilon_n=nd^{-1/s}$, with $s$ given
by~\eqref{eq:cusp-dimension}. When $\epsilon_n$ is small, the basic estimates
control $G_2$, $G_3$, and $\sum_{k=4}^{n-1}G_k$ by
$C\epsilon_n^3$ (see Propositions~\ref{prop:cusp-triangle-two},
\ref{prop:cusp-triangle-three}, and~\ref{prop:cusp-triangle-remaining},
respectively). Without cancellation, these estimates suffice since
\[
 \Lambda_{n,d}\asymp n^3\tau_{3,d}^2
 \asymp\epsilon_n^3.
\]
However, cancellation reduces $\tau_{3,d}$, and $\Lambda_{n,d}$
can be much smaller than $\epsilon_n^3$ (see Remark~\ref{rmk:cancel}).
We must then retain the actual triangle contribution
$n^3\tau_{3,d}^2$ and control the remaining terms at the finer scale
$n^4\tau_{4,d}^2$. Obtaining these sharper estimates is the main
difficulty of the proof.
Recall from~\eqref{eq:cavity-costs} that
$G_k=\sum_{r=k}^{n-1}\binom rk g_{r,k}$.
This section is organized as follows:\begin{itemize}
    \item Section~\ref{sec:localization}: Localization near zero
    and control of the approximation error.
    \item Section~\ref{sec:cusp-posterior}: Order two estimates ($G_2$)
    retaining the triangle and quadrilateral signals.
    \item Section~\ref{sec:order-3-csup}: Order three estimates ($G_3$)
    using parity and the edges among the three roots.
        \item Section~\ref{sec:order-high-csup}: Higher order estimates ($G_k$ with $k\ge 4$)
    and completion of the proof of Proposition~\ref{prop:cusp-KL}.
\end{itemize}

\subsection{Preliminary: localization of the kernel}\label{sec:localization}

We first consider the three continuous cases in Definition~\ref{def:cusp}. Since $T_d=\langle X_1,X_2\rangle=O_{\P}(d^{-1/2})$,
the edge probability $K(T_d)$ mainly depends on the behavior of $K$ near zero. We may therefore replace $K$ by $p_0$ outside a slightly
larger interval before estimating the posterior means. Let
$\delta_d=d^{-11/24}$ and define
\[
 \widetilde K_d(t)=p_0+(K(t)-p_0)\indc{0<|t|\le\delta_d}.
\]
Use tildes for the corresponding density, centered kernel, graph laws, and
traces. The modified link still takes values in $[c,1-c]$.
Figure~\ref{fig:cusp-localization} illustrates this change.
The next lemma shows that it preserves the KL divergence and the
cycle signals up to exponentially small errors.

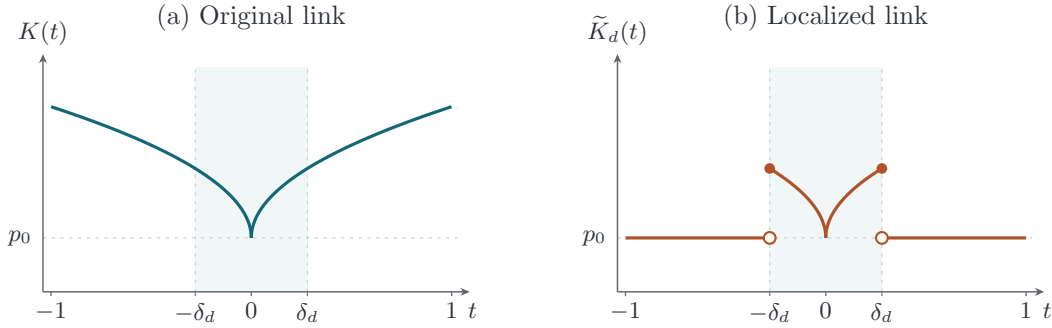
\begin{figure}[!htbp]
\centering
\begingroup
\definecolor{locteel}{RGB}{20,103,116}
\definecolor{loccopper}{RGB}{175,83,43}
\definecolor{locink}{RGB}{44,49,55}
\definecolor{locrule}{RGB}{196,203,207}
\begin{tikzpicture}[
  x=1cm,y=1cm,
  loc axis/.style={draw=locink!75,line width=.55pt,
    -{Stealth[length=3.5pt,width=3pt]}},
  loc guide/.style={draw=locrule,line width=.4pt,
    dash pattern=on 1.2pt off 2.1pt},
  loc original/.style={draw=locteel,line width=1.2pt},
  loc changed/.style={draw=loccopper,line width=1.2pt,solid},
  loc label/.style={font=\footnotesize,text=locink,inner sep=1pt},
  loc title/.style={font=\small,text=locink,inner sep=0pt},
  loc closed/.style={circle,draw=loccopper,fill=loccopper,
    line width=.6pt,inner sep=0pt,minimum size=3.4pt},
  loc open/.style={circle,draw=loccopper,fill=white,
    line width=.8pt,inner sep=0pt,minimum size=4.4pt}
]
\def\locwindow{.28}
\def\locbase{.23}
\def\locheight{.56}
\foreach \panelx in {3.15,10.75}{
  \begin{scope}[shift={(\panelx,0)},x=2.65cm,y=3.1cm]
    \fill[locteel!6] (-\locwindow,0) rectangle (\locwindow,.96);
    \draw[loc guide] (-\locwindow,0)--(-\locwindow,.96);
    \draw[loc guide] (\locwindow,0)--(\locwindow,.96);
    \draw[loc guide] (-1.04,\locbase)--(1.03,\locbase);
    \draw[loc axis] (-1.04,0)--(1.09,0);
    \draw[loc axis] (-1.04,0)--(-1.04,1.015);
    \node[loc label,below=3pt] at (1.105,0) {$t$};
    \node[loc label,left=3pt] at (-1.04,\locbase) {$p_0$};
    \foreach \xx/\lab in {-1/{-1},0/{0},1/{1}}{
      \draw[locink!70,line width=.5pt] (\xx,0)--(\xx,-.018);
      \node[loc label,below=3pt] at (\xx,0) {$\lab$};
    }
    \foreach \xx/\lab in {-\locwindow/{-\delta_d},\locwindow/{\delta_d}}{
      \draw[locink!65,line width=.45pt] (\xx,0)--(\xx,-.018);
      \node[loc label,below=3pt] at (\xx,0) {$\lab$};
    }
  \end{scope}
}
\begin{scope}[shift={(3.15,0)},x=2.65cm,y=3.1cm]
  \node[loc title] at (0,1.17) {(a) Original link};
  \node[loc label,above=2pt] at (-1.04,1.015) {$K(t)$};
  \draw[loc original] plot[domain=0:1,samples=91,variable=\u]
    ({-\u*\u},{\locbase+\locheight*\u});
  \draw[loc original] plot[domain=0:1,samples=91,variable=\u]
    ({\u*\u},{\locbase+\locheight*\u});
\end{scope}
\begin{scope}[shift={(10.75,0)},x=2.65cm,y=3.1cm]
  \node[loc title] at (0,1.17) {(b) Localized link};
  \node[loc label,above=2pt] at (-1.04,1.015) {$\widetilde K_d(t)$};
  \draw[loc changed] (-1,\locbase)--(-\locwindow,\locbase);
  \draw[loc changed] (\locwindow,\locbase)--(1,\locbase);
  \draw[loc changed] plot[domain=0:sqrt(\locwindow),samples=71,variable=\u]
    ({-\u*\u},{\locbase+\locheight*\u});
  \draw[loc changed] plot[domain=0:sqrt(\locwindow),samples=71,variable=\u]
    ({\u*\u},{\locbase+\locheight*\u});
  \foreach \sgn in {-1,1}{
    \node[loc closed] at
      ({\sgn*\locwindow},{\locbase+\locheight*sqrt(\locwindow)}) {};
    \node[loc open] at ({\sgn*\locwindow},\locbase) {};
  }
\end{scope}
\end{tikzpicture}
\endgroup
\caption{Localization keeps $K(t)=p_0+b\sqrt{|t|}$ on
$[-\delta_d,\delta_d]$ and replaces it by $p_0$ outside, where
$\delta_d=d^{-11/24}$ (window enlarged).
Lemma~\ref{lem:localization} bounds the resulting changes in the
centered kernel, cycle traces, and KL divergence.}
\label{fig:cusp-localization}
\end{figure}

\begin{lemma}[Localization in relative entropy]\label{lem:localization}
Fix a continuous cusp $K$ and $C_0>0$. There are constants
$C,c_2>0$ such that, for $n\le C_0d^2$ and all sufficiently large $d$,
\begin{align}
 \|\kappa_d-\widetilde\kappa_d\|_2
 +|\tau_{3,d}-\widetilde\tau_{3,d}|
 +|\tau_{4,d}-\widetilde\tau_{4,d}|+ |\KL(P\Vert Q)
 -\KL(\widetilde P\Vert\widetilde Q)|
 &\le C\exp(-c_2d^{1/12}).\label{eq:cusp-localized-KL}
\end{align}
\end{lemma}

\begin{proof}
By the spherical tail bound~\eqref{eq:kernel-bridge-sphere-tail},
$\epsilon_d:=\P[|T_d|>\delta_d]\le C\exp(-cd^{1/12})$.
The links agree on $0<|t|\le\delta_d$, so
$|p_d-\widetilde p_d|\le\epsilon_d$ and
$\|\kappa_d-\widetilde\kappa_d\|_2\lesssim\sqrt{\epsilon_d}$.
Their operator norms and Hilbert--Schmidt norms are bounded by
constants depending only on $c$.
For $\ell=3,4$, we expand $T_{\kappa_d}^\ell-T_{\widetilde\kappa_d}^\ell$ and use cyclicity
of the trace to obtain
\begin{align*}
|\tau_{\ell,d}-\widetilde\tau_{\ell,d}|
&=\Big|\sum_{j=0}^{\ell-1}
\operatorname{Tr}\big((T_{\kappa_d}-T_{\widetilde\kappa_d})
T_{\widetilde\kappa_d}^jT_{\kappa_d}^{\ell-1-j}\big)\Big|\\
&\le \|T_{\kappa_d}-T_{\widetilde\kappa_d}\|_{\mathrm{HS}}
\sum_{j=0}^{\ell-1}\|T_{\widetilde\kappa_d}^jT_{\kappa_d}^{\ell-1-j}\|_{\mathrm{HS}}\lesssim \|\kappa_d-\widetilde\kappa_d\|_2.
\end{align*}
Here we used $\|T_{\kappa_d}-T_{\widetilde\kappa_d}\|_{\mathrm{HS}}
=\|\kappa_d-\widetilde\kappa_d\|_2$
and the uniform bounds on the operator and Hilbert--Schmidt norms
of $T_{\kappa_d}$ and $T_{\widetilde\kappa_d}$.
It remains to bound $|\KL(P\Vert Q)
 -\KL(\widetilde P\Vert\widetilde Q)|$.

Let $N=\binom n2$ and $h(x)=-x\log x-(1-x)\log(1-x)$.
Coupling the two graphs using the same latent points and edge uniforms
shows that $\eta:=\TV(P,\widetilde P)\le N\epsilon_d$.
In the sums below, $a$ ranges over $\{0,1\}^N$.
Since $Q=\Bern(p_d)^{\otimes N}$ has the same edge marginals as
$P$,
\[
 \KL(P\Vert Q)
 =\sum_a P(a)\log P(a)+Nh(p_d),
\]
and the same identity holds for the modified model.

Let $\pi(a,b)$ be a maximal coupling of $P$ and $\widetilde P$,
so that $\sum_{a\ne b}\pi(a,b)=\eta$.
For each graph $a$, define
$\eta_a=\sum_{b\ne a}\pi(a,b)/P(a)$.
Then $\pi(a,a)=P(a)(1-\eta_a)$ and
$\sum_a P(a)\eta_a=\eta$.
Separating the terms with $a=b$ and $a\ne b$, we obtain
\begin{align*}
-\sum_{a,b}\pi(a,b)\log\frac{\pi(a,b)}{P(a)}
&=-\sum_a\pi(a,a)\log\frac{\pi(a,a)}{P(a)}
-\sum_{a\ne b}\pi(a,b)\log\frac{\pi(a,b)}{P(a)}\\
&=\sum_a P(a)h(\eta_a)
-\sum_{a\ne b}\pi(a,b)
\log\frac{\pi(a,b)}{P(a)\eta_a}\\
&\le \sum_a P(a)h(\eta_a)+\eta N\log 2\le h(\eta)+\eta N\log 2,
\end{align*}
where $h(t)=-t\log t-(1-t)\log(1-t)$.
For the first inequality, whenever $\eta_a>0$, the numbers
$\pi(a,b)/(P(a)\eta_a)$, indexed by $b\ne a$, form a probability
distribution on at most $2^N$ graphs. Its entropy is therefore
at most $N\log 2$, and its contribution is weighted by
$P(a)\eta_a$. The last inequality follows from the concavity of $h$.
Terms with zero mass are interpreted as zero.
Since $\pi(a,b)\le\widetilde P(b)$, it follows that
\begin{align*}
\sum_a P(a)\log P(a)
-\sum_b\widetilde P(b)\log\widetilde P(b)
&=\sum_{a,b}\pi(a,b)\log\frac{\pi(a,b)}{\widetilde P(b)}
-\sum_{a,b}\pi(a,b)\log\frac{\pi(a,b)}{P(a)}\\
&\le-\sum_{a,b}\pi(a,b)\log\frac{\pi(a,b)}{P(a)}\le h(\eta)+\eta N\log 2.
\end{align*}
Interchanging $P$ and $\widetilde P$ gives $\Big|\sum_a P(a)\log P(a)
-\sum_a\widetilde P(a)\log\widetilde P(a)\Big|
\le h(\eta)+\eta N\log 2.$
Since $h$ is Lipschitz on $[c,1-c]$, we obtain
\begin{align*}
 |\KL(P\Vert Q)-\KL(\widetilde P\Vert\widetilde Q)|
 &\le CN\epsilon_d+h(\eta)+\eta N\log2
 \le C\exp(-c_2d^{1/12}),
\end{align*}
where the last inequality uses $n\le C_0d^2$,
$h(x)\le x\log(e/x)$, and a smaller $c_2$.
This proves~\eqref{eq:cusp-localized-KL}.
\end{proof}

Recall from Proposition~\ref{gen:prop:spectral-scales} that
$\sqrt{\tau_{4,d}}\gtrsim d^{-2}$. Thus $\Lambda_{n,d}\le c_{\ref{prop:cusp-KL}}(K)$
implies $n\lesssim d^2$, and the errors in
Lemma~\ref{lem:localization} are negligible compared with
$n^4\tau_{4,d}^2$. In the posterior estimates below, all quantities
refer to the localized model and its own matched null. We omit tildes
and return to the original model in the proof of
Proposition~\ref{prop:cusp-KL}. For the triangle scale, let $\epsilon_n=nd^{-1/s}$, taking $1/s=\alpha+2/3$ also on the critical relation.
The spectral estimates in Proposition~\ref{gen:prop:spectral-scales}
give
\begin{equation}
 \sqrt{\tau_{4,d}}\lesssim d^{-1/s-1/6},\quad
 v_d\lesssim\min\{1,d^{2/3-1/s}\},\quad
 |\tau_{3,d}|\lesssim d^{-3/(2s)},\quad 1/s\le5/3.
 \label{eq:cusp-triangle-upper-scales}
\end{equation}

\subsection{Order two}
\label{sec:cusp-posterior}

In this section, we focus on the term $G_2$ introduced in~\eqref{eq:cavity-double}.
Recall that
$Y_{ij}=(A_{ij}-p_d)/s_d$ and
$H_d(T_{12})=\E_Z[\kappa_d(\langle X_1,Z\rangle)
\kappa_d(\langle Z,X_2\rangle)]$, where $Z\sim\sigma_d$ is independent.
Thus $H_d$ integrates the two edges of the path $1\to Z\to2$.
Adding edge $12$ closes a triangle; adding the path $1\to j\to2$
through a distinct vertex closes a quadrilateral. Conditional
independence therefore gives
\begin{equation}
 \E_P[H_d(T_{12})Y_{12}]=\tau_{3,d},\quad
 \E_P[H_d(T_{12})Y_{1j}Y_{2j}]=\tau_{4,d}\quad(j\notin\{1,2\}).
 \label{eq:cusp-two-basic-signals}
\end{equation}
These two cycle moments enter the posterior bound as follows.

\begin{proposition}[Order two at the triangle scale]
\label{prop:cusp-triangle-two}
For a fixed continuous cusp and $\epsilon_n = n d^{-1/s}\le c_{\ref{prop:cusp-triangle-two}}(K)$,
\begin{equation}
 G_2\lesssim n^3\tau_{3,d}^2+n^4\tau_{4,d}^2
                 +\frac{n^5\tau_{4,d}^2}{d}
 \lesssim\epsilon_n^3.
 \label{eq:cusp-triangle-two}
\end{equation}
\end{proposition}

\begin{proof}
Recall that $G_2 = \sum_{r=2}^{n-1}\binom{r}{2}g_{r,2}$.
By~\eqref{eq:cusp-triangle-upper-scales},
$n\sqrt{\tau_{4,d}}\lesssim\epsilon_nd^{-1/6}$, thus the
condition of Lemma~\ref{lem:cusp-query} holds.
Applying its consequence for two roots,~\eqref{eq:cusp-query-two},
with $f=H_d$ and the internal edge $12$ retained, we obtain the following bound, where $c_1,\delta>0$ are the constants supplied by that lemma:
\begin{align}
 g_{r,2}&\le C\Big(\E[H_d(T_d)\kappa_d(T_d)]^2
        +n\tau_{4,d}\|H_d\|_4^2\Big)\\
 &\quad{}+\frac{Cn^2\tau_{4,d}}{d}
       \Big(\|H_d\|_2^2+\|H_d\kappa_d-\tau_{3,d}\|_2^2\Big)
       +\|H_d\|_\infty^2\exp(-c_1d^\delta).
 \label{eq:cusp-two-root}
\end{align}
Note that $\E[H_d(T_d)\kappa_d(T_d)] = \E[\kappa_d(T_{12})\kappa_d(T_{23})\kappa_d(T_{13})] = \tau_{3,d}$. Similarly, we have $\Vert H_d\Vert^2_2 = \tau_{4,d}$. It follows from Lemma~\ref{lem:contracted} that $\Vert H_{d}\Vert_4^2\lesssim \tau_{4,d}$.
The dense bound on $K$ gives $\|\kappa_d\|_\infty\lesssim1$, hence
\[
 \|H_d\kappa_d-\tau_{3,d}\|_2^2
 =\operatorname{Var}(H_d(T_d)\kappa_d(T_d))
 \le\|\kappa_d\|_\infty^2\|H_d\|_2^2
 \lesssim\tau_{4,d}.
\]
Also, $\|H_d\|_\infty\le v_d\lesssim1$ by Cauchy--Schwarz.
Since $\tau_{4,d}\gtrsim d^{-4}$ by
Proposition~\ref{gen:prop:spectral-scales}, the exponential error in
\eqref{eq:cusp-two-root} is absorbed into $n\tau_{4,d}^2$.
Substituting these bounds and summing the weights
$\sum_{r=2}^{n-1}\binom r2=\binom n3$ in~\eqref{eq:cavity-costs} yields
\begin{align*}
 G_2&\lesssim n^3\tau_{3,d}^2+n^4\tau_{4,d}^2
                       +n^5\tau_{4,d}^2/d\lesssim\epsilon_n^3+\epsilon_n^4d^{-2/3}
                       +\epsilon_n^5d^{1/s-5/3}
 \lesssim\epsilon_n^3,
\end{align*}
where the inequality uses~\eqref{eq:cusp-triangle-upper-scales} and
$1/s\le5/3$.
\end{proof}

We now consider $\beta=\alpha+1/3$ and
$\Lambda_{n,d}\le c_{\ref{prop:cusp-KL}}(K)$.
Since $\tau_{4,d}\asymp d^{-2\alpha-5/3}$, this permits
$n\lesssim d^{\alpha+5/6}$. The extra term
$n^5\tau_{4,d}^2/d$ in~\eqref{eq:cusp-triangle-two} can then exceed
$n^4\tau_{4,d}^2$. The direct estimate conditions on the remaining
latent positions. We instead retain their posterior given the observed graph.
Integrating the queried vertices in the likelihood expansion gives
$H_{4,d}$, the profile of $T_{\kappa_d}^4$.
Its smaller eigenvalues allow us to bound the second moment of its
posterior mean more sharply. Figure~\ref{fig:order-two-comparison}
illustrates the two steps in the following proposition.

\begin{figure}[!htbp]
\centering
\begingroup
\resizebox{.95\textwidth}{!}{%
\definecolor{proofteal}{RGB}{20,103,116}
\definecolor{proofcopper}{RGB}{175,83,43}
\definecolor{proofink}{RGB}{44,49,55}
\begin{tikzpicture}[
  font=\small,text=proofink,
  vertex/.style={circle,minimum size=3mm,inner sep=0pt,line width=.9pt},
  query/.style={vertex,draw=proofink,fill=white},
  fixed/.style={vertex,fill=proofteal,draw=proofteal},
  latent/.style={vertex,fill=white,draw=proofteal},
  hub/.style={vertex,fill=proofcopper,draw=proofcopper},
  graph/.style={draw=proofteal!45,line width=.8pt},
  factor/.style={draw=proofteal,line width=1.25pt},
  pathfactor/.style={draw=proofcopper,line width=1.25pt,dashed},
  process/.style={draw=proofink!65,line width=.8pt,-{Stealth[length=2mm]}},
  smalllabel/.style={font=\footnotesize,align=center},
  heading/.style={font=\bfseries\normalsize,anchor=base}
]
\path[use as bounding box] (-.2,.4) rectangle (16.0,-9.2);
\draw[proofink!15,line width=.6pt] (7.9,.1)--(7.9,-9.02);

\foreach \offset/\mode in {0/fixed,8.3/latent}{
\begin{scope}[xshift=\offset cm,yshift=.4cm]
  \fill[proofteal!5,rounded corners=5pt] (.25,-2.65) rectangle (7.25,-4.5);
  \node[query,label=left:$X_1$] (q1) at (1.05,-1.7) {};
  \node[query,label=right:$X_2$] (q2) at (6.45,-1.7) {};
  \node[hub,label={[text=proofcopper]above:$Z$}] (z) at (3.75,-1.2) {};
  \draw[pathfactor] (q1)--(z)--(q2);
  \node[text=proofcopper] at (3.75,-1.97) {$H_d(T_{12})$};
  \node[\mode,label=left:$i$] (i) at (1.05,-3.25) {};
  \node[\mode,label=right:$j$] (j) at (6.45,-3.25) {};
  \node[\mode] (a) at (2.5,-3.95) {};
  \node[\mode] (b) at (4.95,-3.95) {};
  \node[\mode] (c) at (3.75,-2.98) {};
  \draw[graph] (i)--(a)--(c)--(b)--(j) (i)--(c)--(j) (a)--(b);
  \draw[factor] (q1)--(i) node[midway,right,font=\footnotesize] {$\kappa_d(T_{i1})$};
  \draw[factor] (q2)--(j) node[midway,left,font=\footnotesize] {$\kappa_d(T_{2j})$};
\end{scope}
}

\node[heading] at (3.75,0) {(a) Fix the remaining latent points};
\node[heading] at (12.05,0) {(b) Keep their posterior};
\node[smalllabel,text=proofteal] at (3.75,-3.88) {latent positions given};
\node[smalllabel,text=proofteal] at (12.05,-3.88) {only the graph observed};

\draw[process] (3.75,-4.3)--(3.75,-7.43)
 node[midway,right,font=\footnotesize] {direct estimate};

\draw[process] (12.05,-4.3)--(12.05,-4.93)
 node[midway,right,font=\footnotesize] {integrate $X_1,Z,X_2$};
\fill[proofteal!5,rounded corners=5pt] (8.55,-5.15) rectangle (15.55,-6.68);
\node[latent,label=left:$i$] (ri) at (9.35,-5.6) {};
\node[latent,label=right:$j$] (rj) at (14.75,-5.6) {};
\node[latent] (ra) at (10.8,-6.35) {};
\node[latent] (rb) at (13.25,-6.35) {};
\node[latent] (rc) at (12.05,-6.05) {};
\draw[graph] (ri)--(ra)--(rc)--(rb)--(rj) (ri)--(rc)--(rj) (ra)--(rb);
\draw[pathfactor] (ri)--(rj)
 node[midway,above,font=\small,text=proofcopper] {$H_{4,d}(T_{ij})$};
\draw[process] (12.05,-6.88)--(12.05,-7.43)
 node[midway,right,font=\footnotesize] {bound its posterior mean};

\fill[proofcopper!7,rounded corners=4pt] (.65,-7.65) rectangle (6.85,-9.05);
\fill[proofteal!7,rounded corners=4pt] (8.95,-7.65) rectangle (15.15,-9.05);
\node[smalllabel] at (3.75,-7.95) {Extra term in $G_2$};
\node[text=proofcopper] at (3.75,-8.6) {$\displaystyle \frac{n^5\tau_{4,d}^2}{d}$};
\node[smalllabel] at (12.05,-7.95) {Extra term in $G_2$};
\node[text=proofteal] at (12.05,-8.6) {$\displaystyle n^4\tau_{4,d}^2$};

\end{tikzpicture}%
}
\endgroup
\caption{Two bounds for the additional contribution to $G_2$.
The right panel integrates $X_1,Z,X_2$ and then bounds the posterior
mean of $H_{4,d}(T_{ij})$.}
\label{fig:order-two-comparison}
\end{figure}
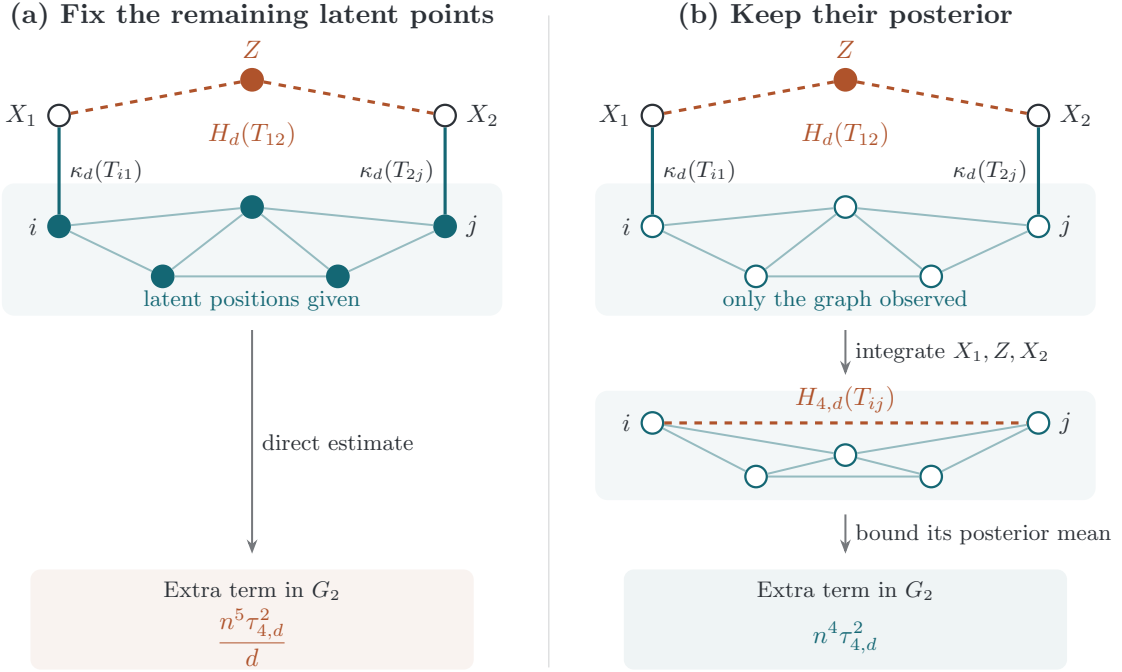

\begin{proposition}[Order two on the critical relation]
\label{prop:cusp-critical-two}
For a fixed mixed cusp with $\beta=\alpha+1/3$ and
$\Lambda_{n,d}\le c_{\ref{prop:cusp-KL}}(K)$, uniformly over $2\le r\le n$,
\begin{align}
 &~g_{r,2}\lesssim\tau_{3,d}^2+n\tau_{4,d}^2
       +n^2\E_P\big[\E_P[H_{4,d}(T_{12})\mid A^{(r)}]^2\big],
 \label{eq:second-star-decomposition}\\
 &~\E_P\big[\E_P[H_{4,d}(T_{12})\mid A^{(r)}]^2\big]
       \lesssim\tau_{5,d}^2+n\tau_{4,d}\tau_{8,d}
                   +n^2\tau_{4,d}\|T_{\kappa_d}\|_{\op}^8.
 \label{eq:fourth-kernel-prediction}
\end{align}
Consequently, $n^3\sup_{2\le r\le n}g_{r,2}\lesssim\Lambda_{n,d}$
and $G_2\lesssim\Lambda_{n,d}$.
\end{proposition}

\begin{proof}[Proof sketch]
We first explain~\eqref{eq:second-star-decomposition}.
Recall that $H_d$ represents a path of length two.
Adding edge $12$ closes a triangle, giving
$\E[H_d(T_d)\kappa_d(T_d)]=\tau_{3,d}$.
Separating this edge and writing
$H_d\kappa_d=\tau_{3,d}+(H_d\kappa_d-\tau_{3,d})$
therefore gives the contribution $\tau_{3,d}^2$.
For the other edges, retain the posterior of the remaining
latent points and expand the incident likelihood.
Selecting edges $1i$ and $2j$, with $i,j\in\{3,\ldots,r\}$,
gives the coefficient
\begin{equation}
 \E_{X_1,X_2}\big[
 \kappa_d(T_{i1})H_d(T_{12})\kappa_d(T_{2j})\big]
 =H_{4,d}(T_{ij}),
 \label{eq:cusp-fourth-profile}
\end{equation}
where the two roots are integrated against their independent
priors, keeping the other latent points fixed.
The two selected edges extend the original path to length four.
If $i=j$, it closes and $H_{4,d}(1)=\tau_{4,d}$,
giving $n\tau_{4,d}^2$ after summing the squared coefficients.
If $i\ne j$, the coefficient is averaged over the posterior
given the graph on the remaining vertices.
There are at most $n^2$ such choices, giving the last term
in~\eqref{eq:second-star-decomposition}.
Lemma~\ref{lem:cusp-retained-design} justifies this reduction.
The root moment and query bounds control all remaining terms
by $Cn\tau_{4,d}^2$.

We next apply the same estimate as in~\eqref{eq:cusp-two-root},
with $H_d$ replaced by $H_{4,d}$.
The corresponding terms become
\begin{align*}
 \E[H_{4,d}(T_d)\kappa_d(T_d)]
 &=\tau_{5,d},\quad \|H_{4,d}\|_2^2=\tau_{8,d},\quad
 \|H_{4,d}\|_4^2+\|H_{4,d}\kappa_d-\tau_{5,d}\|_2^2
 \lesssim\tau_{8,d}.
\end{align*}
Then we have 
\[
 \E_P\big[\E_P[H_{4,d}(T_{12})\mid A^{(r)}]^2\big]
 \lesssim \tau_{5,d}^2+n\tau_{4,d}\tau_{8,d}
                  +\frac{n^2\tau_{4,d}\tau_{8,d}}{d},
\]
with the exponential error absorbed as in
Proposition~\ref{prop:cusp-triangle-two}.
Since $\tau_{8,d}/d\asymp\|T_{\kappa_d}\|_{\op}^8$
on the critical relation, this proves
\eqref{eq:fourth-kernel-prediction}.

Finally, the spectral estimates~\eqref{eq:cusp-app-critical-scales}
in Appendix~\ref{app:cusp-local} and $n\lesssim d^{\alpha+5/6}$ give
\begin{align*}
 n^5\tau_{5,d}^2
 &\lesssim n^4\tau_{4,d}^2
              \big(nd^{-\alpha-4/3}\big)
 \lesssim n^4\tau_{4,d}^2d^{-1/2},\\
 n^6\tau_{4,d}\tau_{8,d}
 &\lesssim n^4\tau_{4,d}^2
              \big(n^2d^{-2\alpha-8/3}\big)
 \lesssim n^4\tau_{4,d}^2d^{-1},\\
 n^7\tau_{4,d}\|T_{\kappa_d}\|_{\op}^8
 &\lesssim n^4\tau_{4,d}^2
              \big(n^3d^{-2\alpha-11/3}\big)
 \lesssim n^4\tau_{4,d}^2d^{\alpha-7/6}.
\end{align*}
Since $\alpha\le1$, these three terms are bounded by
$Cn^4\tau_{4,d}^2$.
Multiplying~\eqref{eq:second-star-decomposition} by $n^3$
therefore gives
$n^3\sup_{2\le r\le n}g_{r,2}\lesssim\Lambda_{n,d}$
and hence $G_2\lesssim\Lambda_{n,d}$.
The full proof is in Appendix~\ref{app:cusp-order-two}.
\end{proof}

\subsection{Order three}\label{sec:order-3-csup}

We next bound $G_3$. At the triangle scale, the moment estimate in
Lemma~\ref{lem:root-moments} and the general bound in
Lemma~\ref{lem:cusp-query} suffice. On the critical relation
$\beta=\alpha+1/3$, cancellation may weaken the triangle signal,
so we retain the internal edges and estimate their contributions
more carefully.

\begin{proposition}[Order three]
\label{prop:cusp-triangle-three}
For a fixed continuous cusp and $\epsilon_n\le c_{\ref{prop:cusp-triangle-three}}(K)$,
\begin{equation}
 G_3\lesssim\big(nd^{-1/s-1/6}\big)^5
       +n^4v_d^{19/8}d^{-2/s-4/3}
 \lesssim\epsilon_n^3.
 \label{eq:cusp-triangle-three}
\end{equation}
For a mixed cusp with $\beta=\alpha+1/3$ and
$\Lambda_{n,d}\le c_{\ref{prop:cusp-KL}}(K)$, we have $G_3\lesssim n^4\tau_{4,d}^2$.
\end{proposition}

\begin{proof}[Proof sketch]
Recall that $\epsilon_n=nd^{-1/s}\le c_{\ref{prop:cusp-triangle-three}}(K)$.
By~\eqref{eq:cusp-triangle-upper-scales} and
$v_d\gtrsim d^{-1}$ from Lemma~\ref{gen:lem:gap-transfer},
\[
 n\sqrt{\tau_{4,d}}\lesssim\epsilon_nd^{-1/6},
 \quad
 n\sqrt{\tau_{4,d}/d}
 \lesssim\epsilon_nd^{-2/3}\lesssim v_d^{11/24}.
\]
Thus Lemma~\ref{lem:cusp-query} applies.
Using~\eqref{eq:cusp-query-rough} with $f=h_{3,d}$
and summing the exposure weights gives
\[
 G_3\lesssim
 n^5\tau_{4,d}\|h_{3,d}\|_4^2
 +n^4\|h_{3,d}\|_4^2
 \big(v_d^{11/24}+Cn\sqrt{\tau_{4,d}/d}\big)^3.
\]
We suppress the exponentially small error, which is absorbed
into $(nd^{-1/s-1/6})^5$ below.
By~\eqref{eq:cusp-app-root-useful},
\[
 \|h_{3,d}\|_4^2
 \lesssim\gamma_{2,d}^6d^{-3}
         +\gamma_{2,d}^2\gamma_{1,d}^4d^{-2}
 \lesssim v_dd^{-2/s-4/3}
 \lesssim d^{-3/s-1/2},
\]
where we used $\gamma_{2,d}^2\le v_d$,
$\gamma_{1,d}^2\lesssim d^{-1/s+1/3}$, and
$v_d\lesssim d^{2/3-1/s}$.
We then obtain
\begin{align*}
 G_3
 &\lesssim(nd^{-1/s-1/6})^5
          +n^4v_d^{19/8}d^{-2/s-4/3}\lesssim\epsilon_n^5d^{-5/6}
          +\epsilon_n^4v_d^{3/8}
 \lesssim\epsilon_n^3.
\end{align*}

On $\beta=\alpha+1/3$, cancellation can reduce $\tau_{3,d}$
and permit larger $n$. The second term in the coarse estimate
does not give $G_3\lesssim n^4\tau_{4,d}^2$.
We instead use~\eqref{eq:cusp-query-tensor} in
Lemma~\ref{lem:cusp-query}.
Let $X=(X_1,X_2,X_3)$ and $X'=(X_1',X_2',X_3')$
be independent prior triples, and write
$T_{ij}'=\langle X_i',X_j'\rangle$.
For the fixed amplification factor $b_*$ in that lemma,
let $A(t)=\exp(nb_*^2H_d(t))$. Then
\begin{align*}
 g_{r,3}
 &\lesssim n\tau_{4,d}\|h_{3,d}\|_4^2
 +\E\Big[h_{3,d}(X)h_{3,d}(X')
 \prod_{i=1}^3A(\langle X_i,X_i'\rangle)
 \prod_{1\le i<j\le3}\big(1+b_*^2\kappa_d(T_{ij})\kappa_d(T_{ij}')\big)\Big].
\end{align*}
The first term contributes at most
$Cn^5\tau_{4,d}^{5/2}
=C(n\sqrt{\tau_{4,d}})^5\lesssim n^4\tau_{4,d}^2$
after summing the exposure weights.

For the second term, expand the three internal factors
and group the eight choices of internal edges as follows.
\begin{itemize}
\item \emph{Zero or one internal edge.}
A root remains connected only to the integrated center,
so the corresponding coefficient has mean zero.
The weighted contribution is controlled using centering,
parity, and conditional moment estimates.
\item \emph{Two internal edges.}
The internal edges form a path.
Adding the center from $h_{3,d}$ produces $K_4$
with one edge removed.
\item \emph{Three internal edges.}
The internal edges form a triangle.
Adding the center produces the complete graph $K_4$.
\end{itemize}
Lemma~\ref{lem:cusp-critical-fixed-orders} controls these
weighted contributions using the graph estimates in
Lemma~\ref{lem:fixed-graph-budget}, giving
$G_3\lesssim n^4\tau_{4,d}^2$.
The critical case is proved in
Appendix~\ref{app:cusp-order-three}.
\end{proof}

\subsection{Higher orders and completion}\label{sec:order-high-csup}

In this section we focus on the higher order $k\ge 4$. We first introduce the following proposition.

\begin{proposition}[The remaining orders at the triangle scale]
\label{prop:cusp-triangle-remaining}
For a fixed continuous cusp and $\epsilon_n\le c_{\ref{prop:cusp-triangle-remaining}}(K)$, each $4\le k\le9$
satisfies
\begin{equation}
 G_k\lesssim_k\big(nd^{-1/s-1/6}\big)^{k+2}
       +n\big(nd^{-1/s-1/6}\big)^k v_d^{11k/24}
 \lesssim_k\epsilon_n^3.
 \label{eq:cusp-triangle-fixed}
\end{equation}
Moreover,
\begin{equation}
 \sum_{k=10}^{n-1}G_k\lesssim\epsilon_n^3.
 \label{eq:cusp-triangle-tail}
\end{equation}
\end{proposition}

\begin{proof}
Lemma~\ref{gen:lem:gap-transfer} gives $v_d\gtrsim d^{-1}$.
Together with~\eqref{eq:cusp-triangle-upper-scales}, this implies
$n\sqrt{\tau_{4,d}/d}\lesssim\epsilon_nd^{-2/3}
\lesssim v_d^{11/24}$.
Thus~\eqref{eq:cusp-query-rough} in Lemma~\ref{lem:cusp-query}
gives, for fixed $k\ge4$,
\begin{align*}
 G_k&\lesssim_k n^{k+1}\|h_{k,d}\|_4^2
       \big(nd^{-2/s-1/3}+v_d^{11k/24}\big)
       +n^{k+1}\exp(-cd^\delta).
\end{align*}
The moment bound~\eqref{eq:cusp-app-root-growing}, a consequence of
Lemma~\ref{lem:root-moments}, gives
$\|h_{k,d}\|_4^2\lesssim_k d^{-k/s-k/6}$.
Substitute this bound and absorb the exponentially small term into
$\big(nd^{-1/s-1/6}\big)^{k+2}$ yields the first inequality
in~\eqref{eq:cusp-triangle-fixed}.
Its first term is at most $\epsilon_n^{k+2}$.
For the second term, at $k=4$ we have
\[
 n\big(nd^{-1/s-1/6}\big)^4v_d^{11/6}
 =\epsilon_n^5d^{1/s-2/3}v_d^{11/6}
 \lesssim\epsilon_n^5,
\]
since~\eqref{eq:cusp-triangle-upper-scales} gives
$d^{1/s-2/3}v_d\lesssim1$ and $v_d\lesssim1$.
Each additional order contributes the factor
$nd^{-1/s-1/6}v_d^{11/24}=o(1)$.
This proves~\eqref{eq:cusp-triangle-fixed} for $4\le k\le9$.
For $k\ge10$, using~\eqref{eq:general-prior-cavity} with
$n\gamma_{1,d}^2/\sqrt d\lesssim\epsilon_nd^{-1/6}$ and
$nv_d/d\lesssim\epsilon_nd^{-1/3}$ yields
\begin{align*}
 \sum_{k=10}^{n-1}G_k
 &\lesssim n\big[(C\epsilon_nd^{-1/6})^{10}
                  +(C\epsilon_nd^{-1/3})^{10}\big]\lesssim\epsilon_n^{11}d^{1/s-5/3}
 \lesssim\epsilon_n^3,
\end{align*}
where the last inequality uses $1/s\le5/3$ and small $\epsilon_n$.
This proves~\eqref{eq:cusp-triangle-tail}.
\end{proof}

We return to the localized mixed kernel with $\beta=\alpha+1/3$.
By~\eqref{eq:prior-cavity-cost} and Lemma~\ref{lem:contracted},
\begin{equation}
 G_k\le\binom n{k+1}\E[|H_d(T_d)|^k]
 \le\binom n{k+1}
 \big[C(\gamma_{1,d}^2\sqrt{k/d}+v_dk/d)\big]^k,
 \label{eq:cusp-prior-start}
\end{equation}
where $v_d=\E[\kappa_d(T_d)^2]$.
Under $\Lambda_{n,d}=o(1)$, \eqref{eq:cusp-prior-start} controls
the sum over $k>C_K\log d$ for a sufficiently large constant $C_K$.
Among the remaining orders, only finitely
many require the graph estimate of Lemma~\ref{lem:fixed-graph-budget}.
An internal edge closes a triangle after integrating the additional
point in $h_{k,d}$. Hence an internal matching on $2j$ roots gives
\begin{equation}
 \E\Big[h_{2j,d}(X_1,\ldots,X_{2j})
       \prod_{i=1}^{j}\kappa_d(T_{2i-1,2i})\Big]
 =\tau_{3,d}^{j}.
 \label{eq:cusp-matching-mean}
\end{equation}
Its squared contribution is at most
$n^{2j+1}\tau_{3,d}^{2j}=n^{1-j}(n^3\tau_{3,d}^2)^j$.
Larger connected components are controlled by the fourth trace.
Lemma~\ref{lem:cusp-critical-fixed-orders} combines these bounds
for every fixed range of orders. We use it with the general query
estimate to obtain the following result.

\begin{proposition}[Higher orders on the critical relation]
\label{prop:cusp-critical-higher}
For a fixed mixed cusp with $\beta=\alpha+1/3$ and
$\Lambda_{n,d}\le c_{\ref{prop:cusp-KL}}(K)$,
\begin{equation}
 \sum_{k=4}^{n-1}G_k\lesssim\Lambda_{n,d}.
 \label{eq:cusp-critical-higher}
\end{equation}
\end{proposition}

\begin{proof}
Let $L=\lfloor\log d\rfloor$. All sums below are restricted to
$m\le n-1$. Recall from~\eqref{eq:cusp-app-critical-scales} that
$v_d\asymp d^{-\alpha}$ and
$\tau_{4,d}\asymp d^{-2\alpha-5/3}$.
Since $n\sqrt{\tau_{4,d}}\le c_{\ref{prop:cusp-KL}}(K)^{1/4}$, we have
$n\lesssim d^{\alpha+5/6}$, and Lemma~\ref{lem:cusp-query}
applies with $C_0=1$ when $c_{\ref{prop:cusp-KL}}(K)$ is sufficiently small.

We first bound the orders $4\le m\le L$.
Choose a fixed integer $M\ge4$ such that
$11\alpha(M+1)/48>\alpha+5/6+2$.
For sufficiently large $d$, we have $M\le L$, and
Lemma~\ref{lem:cusp-critical-fixed-orders} gives
$\sum_{m=4}^{M}G_m\lesssim\Lambda_{n,d}$.
For the remaining orders, the consequence
\eqref{eq:cusp-app-root-growing} of
Lemma~\ref{lem:root-moments} gives
\begin{equation}
 \|h_{m,d}\|_4^2
 \le(Cm)^{m/2}\tau_{4,d}^{m/2},
 \quad M<m\le L.
 \label{eq:cusp-app-critical-root-moment}
\end{equation}
Its range includes $L$ because $\log d=o(d^{1/3})$.

Apply~\eqref{eq:cusp-query-rough} in
Lemma~\ref{lem:cusp-query} with $f=h_{m,d}$ and $E=K_m$.
Since $v_d\asymp d^{-\alpha}$, $0<\alpha\le1$, and
$n\sqrt{\tau_{4,d}}\lesssim1$, we have
$v_d^{11/24}+Cn\sqrt{\tau_{4,d}/d}
\lesssim d^{-11\alpha/24}$.
Using~\eqref{eq:cusp-app-critical-root-moment},
$m!\ge(m/e)^m$, and the exposure weights
$\sum_{r=m}^{n-1}\binom rm=\binom n{m+1}$ gives
\begin{align*}
 \sum_{m=M+1}^{L}G_m
 &\lesssim\sum_{m=M+1}^{L}
 \frac{C^m m^{m/2+2}}{m!}(n\sqrt{\tau_{4,d}})^{m+2}\\
 &\quad{}+n\sum_{m=M+1}^{L}(n\sqrt{\tau_{4,d}})^m
 (C\sqrt m d^{-11\alpha/24})^m
 +\exp(C(\log d)^2-cd^\delta).
\end{align*}
For the first sum, $n\sqrt{\tau_{4,d}}\le1$ gives
\[
 \sum_{m=M+1}^{L}
 \frac{C^m m^{m/2+2}}{m!}
 (n\sqrt{\tau_{4,d}})^{m+2}
 \lesssim n^6\tau_{4,d}^3
 \lesssim n^4\tau_{4,d}^2.
\]
For the second sum, $C\sqrt m\le d^{11\alpha/48}$
uniformly for $m\le L$ when $d$ is sufficiently large.
Factoring out four powers of $n\sqrt{\tau_{4,d}}$
and summing a geometric series gives
\[
 n\sum_{m=M+1}^{L}
 (n\sqrt{\tau_{4,d}})^m
 (C\sqrt m d^{-11\alpha/24})^m
 \lesssim n^5\tau_{4,d}^2d^{-11\alpha(M+1)/48}
 \lesssim n^4\tau_{4,d}^2d^{-2}.
\]
The last inequality uses the choice of $M$ and
$n\lesssim d^{\alpha+5/6}$.
The exponential error is also absorbed by $n^4\tau_{4,d}^2$.
Together with the fixed orders $4\le m\le M$, this gives
$\sum_{m=4}^{L}G_m\lesssim\Lambda_{n,d}$.

It remains to bound the orders $m>L$.
By~\eqref{eq:general-prior-cavity} and
$\gamma_{1,d}^2/\sqrt d+v_d/d\lesssim\sqrt{\tau_{4,d}}$,
\begin{equation}
 G_m\le n(Cn\sqrt{\tau_{4,d}})^m.
 \label{eq:cusp-app-critical-prior-tail}
\end{equation}
Decrease $c_{\ref{prop:cusp-KL}}(K)$ if needed so that
$Cn\sqrt{\tau_{4,d}}\le e^{-4}$.
Factoring out four powers and using
$L=\lfloor\log d\rfloor$ gives
\[
 \sum_{m>L}G_m
 \lesssim n^5\tau_{4,d}^2e^{-4L}
 \lesssim n^4\tau_{4,d}^2d^{-2},
\]
where the last inequality uses
$n\lesssim d^{\alpha+5/6}$ and $\alpha\le1$.
Together with the bound for $4\le m\le L$, this proves
\eqref{eq:cusp-critical-higher}.
\end{proof}

We now finish all of the estimates for continuous kernels. It remains to prove the jump case for Proposition~\ref{prop:cusp-KL}. We first introduce the following proposition.
\begin{proposition}\label{prop:jump}
    For every measurable
$K:[-1,1]\to[0,1]$ and $n\le d/2$,
\begin{equation}
 \KL(P\Vert Q)\le\frac1{2d}\binom n3+O(n^4/d^2).
 \label{gen:eq:universal-Gram-KL}
\end{equation}
\end{proposition}

For a jump, Proposition~\ref{gen:prop:spectral-scales} gives
$\tau_{3,d}^2\asymp d^{-1}$.
Thus $\Lambda_{n,d}\le c_{\ref{prop:cusp-KL}}(K)$ implies $n\lesssim d^{1/3}$,
and~\eqref{gen:eq:universal-Gram-KL} gives
$\KL(P\Vert Q)\lesssim n^3/d\lesssim\Lambda_{n,d}$.

\begin{proof}[Proof of Proposition~\ref{prop:jump}]
We first reduce the graph comparison to a comparison of latent
matrices, as in~\cite[Section~5]{bubeck2016testing}.
Let $R=(T_{ij})_{1\le i,j\le n}$ be the Gram matrix of
$X_1,\ldots,X_n$, where $T_{ij}=\langle X_i,X_j\rangle$
and $R_{ii}=1$.
We identify $R$ with its entries above the diagonal, and let
$\mu_d$ be the law of $T_d$.
Conditional on $R$, generating edges independently with
probabilities $K(R_{ij})$ gives $P$.
If the entries above the diagonal are instead independent
with law $\mu_d$, the same procedure gives $Q$, since
$\E[K(T_d)]=p_d$.
Thus data processing gives
\[
 \KL(P\Vert Q)
 \le \KL\big(\mathcal L(R)\Vert
                  \mu_d^{\otimes\binom n2}\big).
\]
It remains to bound the divergence on the right.

By~\cite[equation~2.1]{rouault2007asymptotic}, the density
of $R$ on positive definite matrices with diagonal entries one is
\[
 c_{n,d}(\det R)^{(d-n-1)/2},\quad
 c_{n,d}=\prod_{j=1}^{n-1}
 \frac{\Gamma(d/2)}{\pi^{j/2}\Gamma((d-j)/2)}.
\]
The density of $\mu_d$ is
$c_{2,d}(1-t^2)^{(d-3)/2}$ for $-1<t<1$.
Taking the expected log ratio of these densities therefore gives
\begin{align*}
 \KL\big(\mathcal L(R)\Vert\mu_d^{\otimes\binom n2}\big)
 &=\log\frac{c_{n,d}}{c_{2,d}^{\binom n2}}
 +\frac{d-n-1}{2}\E[\log\det R]
 -\frac{d-3}{2}\binom n2\E[\log(1-T_d^2)].
\end{align*}

We now compute the three terms.
The Gram--Schmidt decomposition expresses $\det R$ as a product
of independent $\operatorname{Beta}((d-j)/2,j/2)$ variables,
$1\le j\le n-1$; see
\cite[Proposition~2.1(2)]{rouault2007asymptotic}.
For $B\sim\operatorname{Beta}(a,b)$,
$\E[\log B]=\psi(a)-\psi(a+b)$, where
$\psi=\Gamma'/\Gamma$.
Using this identity and the Gamma function estimates
in~\cite[Section~8]{rouault2007asymptotic}, Taylor expansion gives
\begin{align*}
 \log\frac{c_{n,d}}{c_{2,d}^{\binom n2}}
 &=-\frac{\binom n3}{2d}+O(n^4/d^2),\\
 \E[\log\det R]
 &=-\frac{\binom n2}{d}
   -\frac{n(n-1)(2n+5)}{12d^2}+O(n^4/d^3),\\
 \E[\log(1-T_d^2)]
 &=-\frac1d-\frac{3}{2d^2}+O(d^{-3}).
\end{align*}
The remainders are uniform for $n\le d/2$, since
$(d-j)/2\ge d/4$ throughout the expansions.
Substituting these estimates, the two expectation terms together
contribute $\binom n3/d+O(n^4/d^2)$.
Hence
\[
 \KL\big(\mathcal L(R)\Vert\mu_d^{\otimes\binom n2}\big)
 =-\frac{\binom n3}{2d}+\frac{\binom n3}{d}
      +O(n^4/d^2)
 =\frac{\binom n3}{2d}+O(n^4/d^2).
\]
Together with data processing, this proves
\eqref{gen:eq:universal-Gram-KL}.
\end{proof}

We are now ready to combine the bounds for all four cusp cases.

\begin{proof}[Proof of Proposition~\ref{prop:cusp-KL}]
For pure cusps and mixed cusps away from $\beta=\alpha+1/3$,
Proposition~\ref{gen:prop:spectral-scales} gives
$n^3d^{-3/s}\asymp n^3\tau_{3,d}^2$.
Since $\epsilon_n^3\lesssim\Lambda_{n,d}$, choosing $c_{\ref{prop:cusp-KL}}(K)$ sufficiently small ensures the smallness conditions in the following three propositions.
Combining Propositions~\ref{prop:cusp-triangle-two},
\ref{prop:cusp-triangle-three}, and~\ref{prop:cusp-triangle-remaining},
\begin{equation}
 \KL(P\Vert Q)
 \le G_2+G_3+\sum_{k=4}^{9}G_k+\sum_{k=10}^{n-1}G_k
 \lesssim n^3d^{-3/s}\lesssim\Lambda_{n,d}.
 \label{gen:eq:triangle-scale-KL}
\end{equation}
Terms beyond $n-1$ are omitted.
On the critical relation, Propositions~\ref{prop:cusp-critical-two},
\ref{prop:cusp-triangle-three}, and~\ref{prop:cusp-critical-higher}
give the same KL bound. Lemma~\ref{lem:localization} transfers
both conclusions to the original link and its actual traces.

For a jump, the Gram bound~\eqref{gen:eq:universal-Gram-KL}
gives $\KL(P\Vert Q)\lesssim n^3/d\lesssim\Lambda_{n,d}$,
as shown in the jump case.
For $n=2$, the graph law equals its matched null.
This proves~\eqref{eq:cusp-KL} in all four cases.
\end{proof}

Pinsker's inequality now gives $\TV(P,Q)\to0$ when
$\Lambda_{n,d}\to0$. Together with the cycle tests in
Section~\ref{sec:cusp-spectrum}, this completes the proof of
Theorem~\ref{thm:cusp}.

\begin{remark}[Triangle cancellation]
\label{rem:cusp-cancellation}
For a mixed cusp, by~\eqref{eq:cusp-critical-coefficient}, the leading triangle contributions cancel when
\[
 \beta=\alpha+1/3,\quad
 a_o=-\frac{a_e\alpha M_\alpha}
              {2^{1/3}M_{\alpha+4/3}}.
\]
The remaining triangle signal depends on the smaller order terms
in the local expansions. For example, fix
$0<\alpha\le1$, choose $a_e\ne0$ sufficiently small, and define
\[
 K(t)=\frac12+a_e|t|^\alpha
       +a_o\operatorname{sgn}(t)|t|^{\alpha+1/3},
 \quad t\in[-1,1],
\]
with $a_o$ as above. This is a fixed dense kernel.
The trace estimates in Appendix~\ref{app:cusp-spectrum} give
$\tau_{3,d}=O(d^{-3\alpha/2-2})$ and
$\tau_{4,d}\asymp d^{-2\alpha-5/3}$.
Thus Theorem~\ref{thm:cusp} gives the sharp threshold
$d_\star=n^{6/(6\alpha+5)}$, attained by signed quadrilateral counts.
In particular, taking $n\asymp d^{\alpha+1}$ gives
\[
 n^3\tau_{3,d}^2=O(d^{-1})\to0,\quad
 n^4\tau_{4,d}^2\asymp d^{2/3}\to\infty.
\]
Hence $\TV(P,Q)\to1$ even though $n^3\tau_{3,d}^2\to0$.
This contradicts the indistinguishability direction
of~\cite[Conjecture~1]{mao2026random}.
\end{remark}
\section{Impossibility results for finite flat kernels}\label{sec:polynomial}

We prove the impossibility results for Theorem~\ref{thm:finite-flat} in this section.
Recall that $q_m=1$ for odd $m$ and $q_m=2$ for even $m$.
Fix $K\in\maK_{\mathrm{smooth}}(m)$ with $m\ge2$, and write
$n=d^{m+q_m/3}/\omega_d$, where $\omega_d\to\infty$.
Recall that $\bar G_k=\binom n{k+1}\bar g_k$, where
$\bar g_k=g_{n,k}$ corresponds to conditioning on the full graph
$A^{(n)}$. Since $g_{r,k}\le g_{n,k}$, \eqref{eq:cavity-costs}
gives $G_k\le\bar G_k$.
By Lemma~\ref{lem:cavity}, it therefore suffices to show that
$\sum_{k=2}^{n-1}\bar G_k=o(1)$.
This section is organized as follows:
\begin{itemize}
 \item Section~\ref{sec:smooth-order-two}: Order two estimates
 using replicas and truncation.
 \item Section~\ref{sec:smooth-order-three}: Order three estimates
 using a concentration inequality.
 \item Section~\ref{sec:smooth-higher-orders}: Higher order estimates
 using Newton's identities.
 \item Section~\ref{sec:smooth-transfer}: Return to the original
 distribution, control of the Taylor remainder, and completion of the proof.
\end{itemize}
For the fixed order arguments in
Sections~\ref{sec:smooth-order-two}--\ref{sec:smooth-higher-orders},
consider a centered kernel
\begin{equation}
 \kappa(x,y)=\sum_{q\in\mathcal I}b_qU_q(x,y),
 \quad U_q(x,y)=P_{q,d}(\langle x,y\rangle),
 \label{eq:finite-harmonic-model}
\end{equation}
where $\mathcal I$ is a fixed finite set of positive degrees,
the coefficients $b_q$ may depend on $d$, and
$p_d+s_d\kappa$ takes values in $[0,1]$. For ease of notation,
in these fixed order arguments, $P$ denotes the graph law of this
auxiliary kernel. We retain $Q=G(n,p_d)$ and use
$\E_{\post}$, $\Gamma_d$, $h_k$, and $\bar g_k$ for its posterior
replicas, likelihood overlap, interaction kernels, and posterior
moments, respectively, as in Section~\ref{sec:cavity}.
Unsubscripted expectations involving $(X,X')$ refer to two independent
prior samples. The higher order tail below concerns the original kernel;
Section~\ref{sec:smooth-transfer} returns to its posterior for the fixed orders.

\subsection{Order two}
\label{sec:smooth-order-two}

Recall that $h_2$ is the kernel of $T_\kappa^2$. By orthogonality,
$h_2=\sum_q w_qU_q$, where $w_q=b_q^2/N_{q,d}$.
The replica identity turns the squared posterior mean into a quadratic
form. Define
\begin{align}
 O_{q\ell}&=\sum_{i<j}U_q(X_i,X_j)U_\ell(X_i',X_j'),\notag\\
 \mathbf C_{q\ell}
 &=\sum_{i<j}\E_P\big[
   \E_P[U_q(X_i,X_j)\mid A]\E_P[U_\ell(X_i,X_j)\mid A]\big]
 =\E_{\post}[O_{q\ell}].
 \label{eq:smooth-quadratic-form}
\end{align}
Let $w=(w_q)_{q\in\mathcal I}$. Then $\mathbf C\succeq0$ and
$\binom n2\bar g_2=w^\top\mathbf Cw$.
It is enough to bound $\mathbf C_{qq}$, since
$w^\top\mathbf Cw\le|\mathcal I|\sum_q w_q^2\mathbf C_{qq}$.
By~\eqref{eq:4-5}, the likelihood weight satisfies
\begin{equation}
 \Gamma_d(X,X')
 =\prod_{i<j}\big(1+\kappa(X_i,X_j)\kappa(X_i',X_j')\big)
 \le e^{S_\kappa},
 \quad S_\kappa=\sum_{q,\ell}b_qb_\ell O_{q\ell}.
 \label{eq:finite-likelihood-overlap}
\end{equation}
Feasibility makes each factor nonnegative. To apply the planting
inequality in Lemma~\ref{lem:planting}, we therefore need exponential
moments of $S_\kappa$ together with one diagonal overlap $O_{qq}$.

Conditioning on one latent sample turns each overlap into a quadratic
form in the other sample. The deterministic bound on its Gram matrix
is of order $n$, which is too large for the needed exponential estimate.
We instead restrict to the following event. For fixed $D>0$, define
$N_{\max}=\max_{q\in\mathcal I}N_{q,d}$ and
\begin{equation}
 \mathcal G_D(X)=\bigcap_{q\in\mathcal I}
 \Big\{\|\mathbf U_q(X)\|_{\op}
 \le C_{\ref{lem:gram-event}}(D,\mathcal I)\big(1+n/N_{q,d}+\log(n+N_{q,d})\big)\Big\},
 \label{eq:smooth-gram-event}
\end{equation}
where $\mathbf U_q(X)=(U_q(X_i,X_j))_{i,j\le n}$.
Lemma~\ref{lem:gram-event} gives
$\P[\mathcal G_D^c]\le(n+N_{\max})^{-D}$ after increasing
$C_{\ref{lem:gram-event}}(D,\mathcal I)$ for the fixed set of degrees.
The auxiliary model will be used with the coefficient scales of the
Taylor expansion of $K$. Accordingly, in
Sections~\ref{sec:smooth-order-two}--\ref{sec:smooth-higher-orders},
we restrict it to expansions satisfying
\begin{align}
 \sum_q b_q^2&=O(1),
 \quad \sum_q\frac{b_q^2}{N_{q,d}}=O(d^{-m}),\quad 
 \max_q\frac{|b_q|}{N_{q,d}}=O(d^{-(m+q_m)/2}),
 \quad \sum_q\frac{b_q^2}{N_{q,d}^2}=O(d^{-(m+q_m)}).
 \label{eq:finite-bounds}
\end{align}
These bounds hold for the centered Taylor polynomial used in
Section~\ref{sec:smooth-transfer}, as verified in
Lemma~\ref{lem:taylor-finite}. Indeed, a Taylor monomial $t^j$,
$j\ge m$, contributes $O(d^{-(j-q)/2})$ to degree $q$ only if
$q\le j$ and $j-q$ is even. Since $N_{q,d}\asymp d^q$ and
$j+q\ge m+q_m$ for these positive degrees, summing the finitely
many Taylor terms gives~\eqref{eq:finite-bounds}.
Thus this restriction on the auxiliary model adds no assumption on
$K\in\maK_{\mathrm{smooth}}(m)$.
We also require
\begin{equation}
 \max_{q,\ell\in\mathcal I}|b_qb_\ell|
 \big(1+n/N_{q\vee\ell,d}+\log(n+N_{q\vee\ell,d})\big)=o(d),
 \label{eq:finite-field}
\end{equation}
where $q\vee\ell=\max\{q,\ell\}$.
Condition~\eqref{eq:finite-field} permits the conditional exponential
estimate of Lemma~\ref{lem:signed-block} on~\eqref{eq:smooth-gram-event}.
Section~\ref{sec:smooth-transfer} also verifies this condition for
the Taylor coefficients. The weighted bounds in~\eqref{eq:finite-bounds}
retain the smaller degree one component that can occur when $m$ is even.

Let $V_0=nd^{-m}\big(1+\log(n+N_{\max})\big)+n^2d^{-2m}$.
The same exponential bound will also be used at higher orders, so we
allow one additional pair: $F=0$ or
$F=t|b_qb_\ell|O_{q\ell}$, where $0<t\le t_0$ and $t_0>0$ is fixed.

\begin{lemma}[Exponential moments of harmonic overlaps]
\label{lem:finite-mgf}
Under~\eqref{eq:finite-bounds} and~\eqref{eq:finite-field},
for every fixed $a,D>0$,
\begin{equation}
 \log\E\big[\1_{\mathcal G_D(X)\cap\mathcal G_D(X')}
 e^{a(S_\kappa+F)}\big]\le C_{a,t,D}(1+V_0).
 \label{eq:finite-mgf}
\end{equation}
\end{lemma}

The proof in Appendix~\ref{app:proof-finite-mgf} conditions on the
sample carrying the larger degree in each pair, applies
Lemma~\ref{lem:signed-block}, and combines the finitely many pairs
by H\"older's inequality. We now use the bound to control
the diagonal entries of $\mathbf C$.

\begin{proposition}[Order two bound for finite harmonic expansions]
\label{prop:finite-order-two}
Under the assumptions above,
\begin{equation}
 \binom n2\bar g_2\le C(1+V_0)d^{-(m+q_m)}.
 \label{eq:order-two-summary}
\end{equation}
\end{proposition}

\begin{proof}
Fix $D>2$ and an active degree $q$, so $b_q\ne0$.
Lemma~\ref{lem:finite-mgf} with $F=tb_q^2O_{qq}$ gives
\[
 \E\big[\Gamma_d\1_{\mathcal G_D(X)\cap\mathcal G_D(X')}
 \1_{\{O_{qq}>x\}}\big]
 \le\exp\big(C(1+V_0)-tb_q^2x\big).
\]
Apply~\eqref{eq:4-6} in Lemma~\ref{lem:planting} and use the prior
marginal of each posterior replica. After adjusting constants,
\begin{equation}
 \P_{\post}\Big[O_{qq}>\frac{C(1+V_0)+u}{b_q^2}\Big]
 \le Ce^{-cu}+2(n+N_{\max})^{-D},\quad u\ge0.
 \label{eq:mc-11}
\end{equation}
Since $|O_{qq}|\le\binom n2$, integrating the upper tail yields
\begin{equation}
 0\le\mathbf C_{qq}=\E_{\post}[O_{qq}]
 \le\frac{C(1+V_0)}{b_q^2}+Cn^2(n+N_{\max})^{-D}
 \le\frac{C(1+V_0)}{b_q^2},
 \label{eq:mc-12}
\end{equation}
where the last inequality uses $D>2$ and $b_q^2=O(1)$.
Positive semidefiniteness and~\eqref{eq:finite-bounds} now give
\begin{equation}
 \binom n2\bar g_2
  =w^\top\mathbf Cw
 \le|\mathcal I|\sum_qw_q^2\mathbf C_{qq}  
  \le C(1+V_0)\sum_q\frac{b_q^2}{N_{q,d}^2}
 \le C(1+V_0)d^{-(m+q_m)}.
 \label{eq:mc-13}
\end{equation}
\end{proof}

Multiplying by $(n-2)/3$, we obtain
\begin{equation}
 \binom n3\bar g_2
 \lesssim nd^{-(m+q_m)}
 +n^2\log(n+N_{\max})d^{-(2m+q_m)}
 +n^3d^{-(3m+q_m)}=o(1).
 \label{eq:mc-14}
\end{equation}
The last term is $\omega_d^{-3}$, and the first two have additional
negative powers of $d$.

\subsection{Order three}
\label{sec:smooth-order-three}

Recall the replica identity from Section~\ref{sec:cavity}. Define
\begin{equation}
 T_3=\sum_{i<j<k}h_3(X_i,X_j,X_k)h_3(X_i',X_j',X_k'),
 \quad \E_{\post}[T_3]=\binom n3\bar g_3.
 \label{eq:smooth-cubic-replica}
\end{equation}
It suffices to show that $\E_{\post}[|T_3|]=o(n^{-1})$.
Under the prior, conditioning on $X'$ makes $T_3$ a canonical
$U$-statistic in the independent points $X_1,\ldots,X_n$.
To bound its tail, we use a consequence of
\cite[Theorem~7]{adamczak2006moment}.
Let $M(X')=\max_{i<j<k}|h_3(X_i',X_j',X_k')|$, and let
$a_r$, $0\le r\le3$, be the largest $L^2(\sigma_d^{\otimes r})$
norm of $h_3$ obtained by fixing its other $3-r$ arguments.
In particular, $a_0=\|h_3\|_\infty$ and $a_3=\|h_3\|_2$.
If $M(X')=0$, then $T_3=0$. Otherwise, for
$t\ge\max_{0\le r\le3}M(X')n^{r/2}a_r$, the inequality gives
\begin{equation}
 \P[|T_3|>t\mid X']
 \le C\exp\Big(-c\min_{0\le r\le3}
 \Big(\frac{t}{M(X')n^{r/2}a_r}\Big)^{2/(6-r)}\Big).
 \label{eq:smooth-cubic-adamczak}
\end{equation}
See Lemma~\ref{lem:adamczak-cor} for the general statement and proof.
The four terms correspond to integrating zero, one, two, or three
variables. We first control their common factor $M(X')$ on an event
of high probability.

Let $H(x,y)$ be the kernel of $T_\kappa^2$.
By Lemma~\ref{lem:spherical-hypercontractivity} and
\eqref{eq:finite-bounds},
\[
 \|h_3\|_2^2=\E[H(X_1,X_2)^3]
 \le\|H\|_3^3\lesssim d^{-3m-3q_m/2}.
\]
For a fixed $D>5$, let $L_n=C_{\ref{lem:cubic-envelope}}(D)(\log(n+d))^{C_{\ref{lem:cubic-envelope}}(D)}$ and define
\begin{equation}
 \mathcal A_D(X')=
 \big\{M(X')\le L_nd^{-3m/2-3q_m/4}\big\}.
 \label{eq:smooth-cubic-event}
\end{equation}
Lemma~\ref{lem:cubic-envelope} gives
$\P[\mathcal A_D^c]\le(n+d)^{-D}$.
On $\mathcal A_D(X')$, the estimates
\eqref{eq:hc-1}--\eqref{eq:hc-4} in the proof of
Proposition~\ref{prop:finite-order-three} give
$M(X')n^{r/2}a_r\lesssim D_r$ for $0\le r\le3$, where
\begin{align}
 D_0&=L_nd^{-3m-3q_m/4},
 \quad D_1=L_n\sqrt n d^{-3m-5q_m/4},\notag\\
 D_2&=L_nnd^{-3m-3q_m/2},
 \quad D_3=L_nn^{3/2}d^{-3m-3q_m/2}.
 \label{eq:smooth-cubic-scales}
\end{align}
Thus, on $\mathcal A_D(X')$, whenever $t\ge C\max_{0\le r\le3}D_r$,
substitution in~\eqref{eq:smooth-cubic-adamczak} gives
\begin{equation}
 \P[|T_3|>t\mid X']
 \le C\exp\Big(-c\min\Big\{
 (t/D_0)^{1/3},(t/D_1)^{2/5},
 (t/D_2)^{1/2},(t/D_3)^{2/3}\Big\}\Big).
 \label{eq:hc-9}
\end{equation}
We now transfer this prior tail to the posterior.

\begin{proposition}[Order three bound for finite harmonic expansions]
\label{prop:finite-order-three}
Under the finite harmonic model described above,
$\binom n4\bar g_3=o(1)$.
\end{proposition}

\begin{proof}[Proof sketch]
For $n=d^{m+q_m/3}/\omega_d$ with $\omega_d\to\infty$,
choose $\delta>0$ sufficiently small and let $t=n^{-1-\delta}$.
The four terms in~\eqref{eq:hc-9} give
\[
 \P[|T_3|>t,\mathcal A_D(X')]
 \le\exp\big(-d^{\zeta+o(1)}\big)
 \quad\text{for some }\zeta>2q_m/3.
\]
The exponent calculation is given in~\eqref{eq:hc-10}.
To apply the exponential moment estimate, let
$\mathcal G$ be the intersection of $\mathcal G_D$ and
$\mathcal A_D$ for both samples. For
$\mathcal B_t=\{|T_3|>t\}$, Lemma~\ref{lem:planting},
Cauchy--Schwarz, and~\eqref{eq:finite-mgf} give
\begin{align*}
 \P_{\post}[\mathcal B_t\cap\mathcal G]
 &\le\Big(\E[e^{2S_\kappa}
   \1_{\mathcal G_D(X)\cap\mathcal G_D(X')}]\Big)^{1/4}
   \P[\mathcal B_t\cap\mathcal G]^{1/4}\\
 &\le\exp\big(C(1+V_0)-d^{\zeta+o(1)}\big)
 =\exp\big(-d^{\zeta+o(1)}\big),
\end{align*}
because $V_0\le d^{2q_m/3+o(1)}$.
Each posterior replica has the prior marginal, so
$\P_{\post}[\mathcal G^c]\lesssim n^{-D}$.
Using $|T_3|\lesssim n^3d^{-3m}$, we obtain
\[
 \E_{\post}[|T_3|]
 \le n^{-1-\delta}
 +Cn^3d^{-3m}\Big(\exp\big(-d^{\zeta+o(1)}\big)+n^{-D}\Big)
 =o(n^{-1}).
\]
The conclusion follows from~\eqref{eq:smooth-cubic-replica}
and $\binom n4=\frac{n-3}{4}\binom n3$.
The full proof is in
Appendix~\ref{app:proof-finite-order-three}.
\end{proof}
\subsection{Higher fixed orders}
\label{sec:smooth-higher-orders}

We next bound $\bar G_k$ for each fixed $k\ge4$. Let $Z,Z'$ be
independent uniform points, independent of the posterior replicas, and define
$r_i=\kappa(X_i,Z)\kappa(X_i',Z')$ and $p_j=\sum_{i=1}^n r_i^j$.
If $e_k$ is the elementary symmetric polynomial of degree $k$, the
replica identity and Newton's identities give
\begin{align}
 \binom nk\bar g_k
 &=\E_{\post}\big[\E_{Z,Z'}[e_k(r_1,\ldots,r_n)]\big],
 \label{eq:hc-14}\\
 e_k(r_1,\ldots,r_n)
 &=\sum_{\substack{m_1,\ldots,m_k\ge0\\\sum_{j=1}^k jm_j=k}}
 \frac{(-1)^{k-\sum_{j=1}^k m_j}}
      {\prod_{j=1}^k j^{m_j}m_j!}
 \prod_{j=1}^k p_j^{m_j}.
 \label{eq:newton-main}
\end{align}
The second identity follows by taking the coefficient of $t^k$ in
the formal power series identity
\[
 \prod_{i=1}^n(1+tr_i)
 =\exp\Big(\sum_{j\ge1}\frac{(-1)^{j-1}}j p_jt^j\Big).
\]
For $j\ge2$, the triangle inequality already gives a useful bound
for $p_j$. Only $p_1$ requires an estimate under the posterior.
Let $H$ be the kernel of $T_\kappa^2$. The conditional second moment of $p_1$ is
\begin{equation}
 Y_H=\E_{Z,Z'}[p_1^2]
 =\sum_{i,j=1}^n H(X_i,X_j)H(X_i',X_j')\ge0.
 \label{eq:hc-15}
\end{equation}
By Lemma~\ref{lem:spherical-hypercontractivity} and
$\|\kappa\|_2^2\lesssim d^{-m}$ from~\eqref{eq:finite-bounds},
for every fixed $s\ge2$ and every realization of the replicas,
\begin{equation}
 \|p_1\|_{L^s(Z,Z')} \le C_s\sqrt{Y_H}, \quad 
 \|p_j\|_{L^s(Z,Z')}
  \le\sum_{i=1}^n\|r_i\|_{L^{js}(Z,Z')}^j
 \le C_{j,s}nd^{-mj},\quad j\ge2.
 \label{eq:newton-power-sums}
\end{equation}
Thus Newton's expansion reduces the problem to fixed posterior
moments of $Y_H$.

\begin{proposition}[Higher fixed orders]
\label{prop:finite-higher}
For every fixed $k\ge4$, $\binom n{k+1}\bar g_k=o(1)$.
\end{proposition}

\begin{proof}
Adding the term $t|b_qb_\ell|O_{q\ell}$ in
Lemma~\ref{lem:finite-mgf} and integrating the resulting posterior
tails gives, for every fixed $s\ge1$,
\begin{equation}
 \|Y_H\|_{L^s(\P_{\post})}
 \le C_s\big(nd^{-2m}+(1+V_0)d^{-(m+q_m)}\big).
 \label{eq:newton-overlap-main}
\end{equation}
Lemma~\ref{lem:finite-quadratic-overlap} proves this estimate.
For one term of~\eqref{eq:newton-main}, let
$b=\sum_{j=2}^k m_j$. By H\"older's inequality and
\eqref{eq:newton-power-sums},
\begin{equation}
 \E_{\post}\Big[\E_{Z,Z'}\Big[\prod_{j=1}^k|p_j|^{m_j}\Big]\Big]
 \le C_k n^b d^{-m(k-m_1)}
       \|Y_H\|_{L^{k/2}(\P_{\post})}^{m_1/2},
 \quad b+m_1/2\le k/2.
 \label{eq:newton-partition}
\end{equation}
When $m_1=0$, the factor involving $Y_H$ is omitted. There are
only finitely many terms for each fixed $k$.

\noindent\textbf{Case I.} Suppose $n\ge d^{m-q_m}$.
Since $n=o(d^{m+q_m/3})$ and $N_{\max}$ is polynomial in $d$,
\[
 (1+V_0)d^{-(m+q_m)}
  \lesssim nd^{-2m}\big(1+d^{-q_m}\log d
                         +nd^{-(m+q_m)}\big)
 \lesssim nd^{-2m}.
\]
Substituting into~\eqref{eq:newton-partition} bounds each term
by $C_kn^{b+m_1/2}d^{-mk}\le C_kn^{k/2}d^{-mk}$.
Equations~\eqref{eq:hc-14} and~\eqref{eq:newton-main} therefore give
\begin{equation}
 \binom n{k+1}\bar g_k
 \le C_kn^{1+k/2}d^{-mk}
 \le C_k\frac{d^{q_m-m}}{\omega_d^{1+k/2}}=o(1),
 \quad k\ge4.
 \label{eq:hc-20}
\end{equation}
Here we used $\binom n{k+1}/\binom nk=(n-k)/(k+1)$ and
$(m+q_m/3)(1+k/2)-mk\le q_m-m\le0$.

\noindent\textbf{Case II.} Suppose $n<d^{m-q_m}$.
Then $V_0=o(1)$, and~\eqref{eq:newton-overlap-main} is
$O_s(d^{-(m+q_m)})$. After multiplication by the cavity factor,
each term in~\eqref{eq:newton-partition} is bounded by
\[
 C_k n^{1+b}d^{-m(k-m_1)}d^{-(m+q_m)m_1/2}
  \le C_k d^{(m-q_m)(1+b+m_1/2)-mk} 
  \le C_k d^{(m-q_m)(1+k/2)-mk}=o(1).
\]
The last exponent is at most $-m-3q_m$ for $k\ge4$.
\end{proof}

Proposition~\ref{prop:finite-higher} controls every fixed range of
orders for the auxiliary model. We next record a tail bound for the
original kernel $\kappa_d$. In this tail estimate, $H_d$, $\bar g_k$,
and $\bar G_k$ refer to the original graph law $P_{n,d}(K)$.
By Jensen's inequality,
\begin{equation}
 \bar g_k\le\E[h_{k,d}(X_1,\ldots,X_k)^2]
 =\E[H_d(T_d)^k]\le\E[|H_d(T_d)|^k].
 \label{eq:smooth-prior-start}
\end{equation}
The equality is the contraction identity~\eqref{eq:4-4}.
Taylor's theorem and Lemma~\ref{lem:gegen-moment} give
\begin{equation}
 \|H_d(T_d)\|_s\le Cd^{-m}(s/d)^{q_m/2},
 \quad 2\le s\le c_{\ref{lem:gegen-moment}}(m+5)d,
 \quad \|H_d\|_\infty\le Cd^{-m}.
 \label{eq:growing-moment}
\end{equation}
The constant in Lemma~\ref{lem:gegen-moment} is chosen uniformly over degrees at most $m+5$. The estimate is proved in Appendix~\ref{app:proof-cavity-tail}.
For fixed $k$, these bounds give
\[
 \bar G_k\lesssim_k n^{k+1}d^{-k(m+q_m/2)}
 =\frac{d^{m+q_m/3-kq_m/6}}{\omega_d^{k+1}}.
\]
Thus the prior bound suffices when $k>6m/q_m+2$.
The same estimate also controls the sum over all such orders.
Specifically, we have the following proposition.

\begin{proposition}[Higher order tail]
\label{prop:cavity-tail}
For any fixed integer $k_\star>6m/q_m+2$,
\begin{equation}
 \sum_{k=k_\star+1}^{n-1}\bar G_k=o(1).
 \label{eq:growing-tail}
\end{equation}
\end{proposition}

\begin{proof}
For $k_\star<k\le c_{\ref{lem:gegen-moment}}(m+5)d$, equations
\eqref{eq:smooth-prior-start} and~\eqref{eq:growing-moment},
together with $k!\ge(k/e)^k$, give
\begin{equation}
 \bar G_k\le n\big(Cnd^{-m-q_m/2}k^{q_m/2-1}\big)^k
 \le n\big(Cd^{-q_m/6}/\omega_d\big)^k.
 \label{eq:8-9}
\end{equation}
Its geometric sum is $o(1)$ since
$m+q_m/3-(k_\star+1)q_m/6<0$.
For $k>c_{\ref{lem:gegen-moment}}(m+5)d$, the supremum bound instead gives
\begin{equation}
 \bar G_k\le n\Big(\frac{Cen d^{-m}}k\Big)^k
 \le n\big(Cd^{q_m/3-1}/\omega_d\big)^k.
 \label{eq:8-10}
\end{equation}
This sum is also $o(1)$ because $q_m/3<1$.
\end{proof}

Propositions~\ref{prop:finite-order-two},~\ref{prop:finite-order-three},
and~\ref{prop:finite-higher} control the fixed orders for the auxiliary
model, while Proposition~\ref{prop:cavity-tail} controls the tail for
the original kernel. It remains to establish the fixed order bounds
under the original graph law.

\subsection{Return to the original graph law}
\label{sec:smooth-transfer}

From this point on, $P$ again denotes the original graph law
$P_{n,d}(K)$, and $\kappa_d$ denotes its centered kernel.
We approximate $\kappa_d$ by a finite Taylor polynomial and
apply the estimates from
Sections~\ref{sec:smooth-order-two}--\ref{sec:smooth-higher-orders}
under this original posterior. There are two errors to control:
replacing the kernel in the likelihood bound, and replacing it
in $h_{k,d}$.

Recall that $K\in C^{m+6}([-1,1])$. Define
\begin{equation}
 \widetilde\Pi(t)=\sum_{j=0}^{m+5}
       \frac{\kappa_d^{(j)}(0)}{j!}t^j,
 \quad \Pi(t)=\widetilde\Pi(t)-\E[\widetilde\Pi(T_d)],
 \quad R=\kappa_d-\Pi.
 \label{eq:taylor-polynomial}
\end{equation}
Both $\Pi$ and $R$ are canonical. We use $\Pi$ in the
observables and prior exponential moments; all posterior
expectations below use $P$. In particular, we do not require
$p_d+s_d\Pi$ to take values in $[0,1]$.

\begin{lemma}[Taylor coefficients]
\label{lem:taylor-finite}
Write $\Pi=\sum_{q\in\mathcal I_\Pi}b_{q,d}^\Pi P_{q,d}$, where
$\mathcal I_\Pi=\{q\ge1:b_{q,d}^\Pi\ne0\}$.
This set is fixed for all sufficiently large $d$ and lies in $\{1,\ldots,m+5\}$.
Moreover,
\begin{align}
 \sum_q(b_{q,d}^\Pi)^2&=O(1),
 \quad \sum_q\frac{(b_{q,d}^\Pi)^2}{N_{q,d}}=O(d^{-m}),\notag\\
 \max_q\frac{|b_{q,d}^\Pi|}{N_{q,d}}
 &=O(d^{-(m+q_m)/2}),
 \quad \sum_q\frac{(b_{q,d}^\Pi)^2}{N_{q,d}^2}
       =O(d^{-(m+q_m)}).
 \label{eq:taylor-finite-bounds}
\end{align}
\end{lemma}

\begin{proof}
Write $t^j=\sum_qc_{j,q,d}P_{q,d}(t)$.
Orthogonality and parity give $c_{j,q,d}=0$ unless
$q\le j$ and $j-q$ is even. Otherwise,
\begin{equation}
 |c_{j,q,d}|
 =\sqrt{N_{q,d}}\big|\E[T_d^j\psi_{q,d}(T_d)]\big|
 \le\sqrt{N_{q,d}}\|T_d^j\|_2
 \lesssim d^{-(j-q)/2}.
 \label{eq:smooth-monomial-coefficients}
\end{equation}
Only $m\le j\le m+5$ contribute to the nonconstant part of
$\Pi$. For each $q$, the first nonzero Taylor term of the same
parity has a nonzero leading projection by
Lemma~\ref{lem:gegen-hermite}; later Taylor terms are smaller.
Thus the active set is fixed for all sufficiently large $d$.
For each contributing pair $(j,q)$, we have $j+q\ge m+q_m$ and
\[
 \frac{c_{j,q,d}^2}{N_{q,d}}\lesssim d^{-j},
 \quad \frac{|c_{j,q,d}|}{N_{q,d}}
       \lesssim d^{-(j+q)/2},
 \quad \frac{c_{j,q,d}^2}{N_{q,d}^2}
       \lesssim d^{-(j+q)}.
\]
Summing over the fixed set of Taylor terms proves
\eqref{eq:taylor-finite-bounds}.
\end{proof}

The Taylor polynomial also satisfies~\eqref{eq:finite-field}.
Indeed, if $\ell\le r$ arise from Taylor orders $s,t\ge m$, then
\[
 |b_{\ell,d}^\Pi b_{r,d}^\Pi|\frac n{N_{r,d}}
 \lesssim nd^{-(s-\ell+t-r)/2-r}
 \le nd^{-m+(\ell-r)/2}
 \le d^{q_m/3}/\omega_d=o(d).
\]
The remaining logarithmic term in~\eqref{eq:finite-field} is
$o(d)$ because the coefficients are bounded. We next record
the error supplied by the six additional derivatives.

\begin{lemma}[Taylor residual]
\label{lem:residual-bounds}
For every fixed $1\le s<\infty$,
\begin{equation}
 |R(t)| \lesssim |t|^{m+6}+d^{-(m+6)/2}, \quad 
 \sup_x\|R(x,\cdot)\|_s \lesssim_s d^{-(m+6)/2},
 \quad \|T_R\|_{\op}\le\|R\|_2\lesssim d^{-(m+6)/2}.
 \label{eq:hc-21}
\end{equation}
\end{lemma}

\begin{proof}
Taylor's theorem gives
$|\kappa_d(t)-\widetilde\Pi(t)|\lesssim|t|^{m+6}$.
Since $R$ is its centered version, spherical moments give
\[
 \|R(T_d)\|_s
 \lesssim\|T_d\|_{s(m+6)}^{m+6}
 \lesssim_s d^{-(m+6)/2}.
\]
The pointwise estimate follows in the same way.
Rotation invariance gives the uniform slice norm, and
Cauchy--Schwarz gives $\|T_R\|_{\op}\le\|R\|_2$.
\end{proof}

We first compare the likelihood weights. Recall $\Gamma_d$
from~\eqref{eq:4-5}, and define
$S_\Pi=\sum_{i<j}\Pi(T_{ij})\Pi(T_{ij}')$.
Since Taylor approximation is most effective near zero, let
$\mathcal E_\tau(X)=\{\max_{i<j}|T_{ij}|\le\tau\}$, where
$\tau=\sqrt{C_0\log(n+d)/d}$. The next estimate keeps the
indicator of this event inside the prior expectation.

\begin{proposition}[Likelihood comparison]
\label{prop:weighted-transfer}
Suppose $n\le Cd^{m+q_m/3}$. For every fixed $D>0$, one can
choose $C_0$ so that $\P[\mathcal E_\tau^c]\le(n+d)^{-D}$
and, for every fixed $a>1$, every event $\mathcal G$ and every
real measurable function $F$ of the independent prior pair,
\begin{equation}
 \E\big[\1_{\mathcal G\cap\mathcal E_\tau(X)\cap
                \mathcal E_\tau(X')}\Gamma_de^F\big]
 \le e^{o(1)}
 \big(\E[\1_{\mathcal G}e^{a(S_\Pi+F)}]\big)^{1/a},
 \label{eq:transfer-interface}
\end{equation}
whenever the expectation on the right is finite.
\end{proposition}

\begin{proof}
For $f\in\{\Pi,R\}$, let
$f_\tau=f\1_{\{|t|\le\tau\}}-\E[f(T_d)\1_{\{|T_d|\le\tau\}}]$
and define
\[
 U=\sum_{i<j}\big(
 \Pi_\tau(T_{ij})R_\tau(T_{ij}')
 +R_\tau(T_{ij})\Pi_\tau(T_{ij}')
 +R_\tau(T_{ij})R_\tau(T_{ij}')\big).
\]
Lemma~\ref{lem:central-event} and $1+x\le e^x$ give
$\Gamma_d\le\exp(S_\Pi+U+o(1))$ on the two central events.
The four parameter estimate in~\eqref{eq:tr-3} gives
$\E[e^{s|U|}]=1+o(1)$ for every fixed $s>0$; the tail
integration is given in~\eqref{eq:tr-6} in
Appendix~\ref{app:proof-weighted-transfer}.
For $a'=a/(a-1)$, H\"older's inequality therefore gives
\begin{equation}
 \E\big[\1_{\mathcal G\cap\mathcal E_\tau(X)\cap
                 \mathcal E_\tau(X')}\Gamma_de^F\big]  
 \le e^{o(1)}
 \big(\E[\1_{\mathcal G}e^{a(S_\Pi+F)}]\big)^{1/a}
 \big(\E[e^{a'|U|}]\big)^{1/a'}.
 \label{eq:tr-7}
\end{equation}
This proves~\eqref{eq:transfer-interface}.
\end{proof}

We can now use the estimates for each fixed order. Let
$h_{k,d}^\Pi$ be formed from $\Pi$ as
in~\eqref{eq:star-kernel}, and define
\[
 \bar g_k^\Pi=\E_P\big[
 \E_P[h_{k,d}^\Pi(X_1,\ldots,X_k)\mid A]^2\big].
\]
Take $\mathcal I=\mathcal I_\Pi$ and $b_q=b_{q,d}^\Pi$ in
the overlap and Gram event definitions. By
Lemma~\ref{lem:finite-mgf} and the coefficient bounds just
verified,
\begin{equation}
 \log\E\big[\1_{\mathcal G_D(X)\cap\mathcal G_D(X')}
           e^{a(S_\Pi+F)}\big]
 \lesssim_{a,t,D}1+V_0,
 \quad V_0\le d^{2q_m/3+o(1)}.
 \label{eq:smooth-prior-mgf}
\end{equation}
Here $F$ is any additional overlap term allowed in that lemma.
Proposition~\ref{prop:weighted-transfer} supplies this same
bound with the original likelihood. Thus the arguments for
order two, order three, and higher fixed orders apply with
$\Pi$ in the observables and $P$ in the posterior.

Recall that, for $I\subseteq[n]$, we write $X_I=(X_i)_{i\in I}$ and
$X_I'=(X_i')_{i\in I}$. To remove $R$ from the observables,
the replica identity gives
\[
 \binom nk|\bar g_k-\bar g_k^\Pi|
 \le\E_{\post}\Bigg[\Bigg|
 \sum_{|I|=k}\big(
 h_{k,d}(X_I)h_{k,d}(X_I')
 -h_{k,d}^\Pi(X_I)h_{k,d}^\Pi(X_I')\big)
 \Bigg|\Bigg].
\]
Every summand contains at least one residual factor. The gain
$d^{-3}$ in Lemma~\ref{lem:residual-bounds} makes the prior
tail small enough to survive the likelihood comparison.
The resulting estimate is as follows.

\begin{proposition}[Fixed orders under the original law]
\label{prop:star-transfer}
For every fixed $k\ge2$,
\begin{equation}
 \binom n{k+1}\bar g_k^\Pi=o(1),
 \quad \binom nk|\bar g_k-\bar g_k^\Pi|=o(n^{-1}),
 \quad \binom n{k+1}\bar g_k=o(1).
 \label{eq:star-transfer-summary}
\end{equation}
\end{proposition}

The proof is given in Appendix~\ref{app:proof-star-transfer}.
It expands the residual terms and checks the slice norms
needed for the canonical statistic bound. We are now ready to
combine the fixed orders with the tail estimate from
Proposition~\ref{prop:cavity-tail}.

\begin{proposition}[Indistinguishability for finitely flat kernels]
\label{prop:finite-flat-converse}
If $K\in\maK_{\mathrm{smooth}}(m)$ and
$d\gg n^{3/(3m+q_m)}$, then $\KL(P\|Q)\to0$ and
$\TV(P,Q)\to0$.
\end{proposition}

\begin{proof}
Fix an integer $k_\star>6m/q_m+2$.
Lemma~\ref{lem:cavity} and
Propositions~\ref{prop:star-transfer} and~\ref{prop:cavity-tail}
give
\[
 \KL(P\|Q)
 \le\sum_{k=2}^{k_\star}\binom n{k+1}\bar g_k
    +\sum_{k=k_\star+1}^{n-1}\binom n{k+1}\bar g_k=o(1).
\]
Pinsker's inequality completes the proof of the converse in
Theorem~\ref{thm:finite-flat}.
\end{proof}

\section{Discussions}\label{sec:discussion}
\noindent\textbf{The conjecture in~\cite{mao2026random}.}
Our results show that signed triangles attain the optimal detection
threshold for many kernels beyond those previously studied. However,
their signal can be weakened by cancellation between the even and odd
parts of a mixed cusp. Remark~\ref{rem:cusp-cancellation} gives a fixed
kernel for which $n^3\tau_{3,d}^2\to0$ while
$n^4\tau_{4,d}^2\to\infty$. By Theorem~\ref{thm:cusp},
$\TV(P,Q)\to1$, contradicting the indistinguishability prediction
of~\cite[Conjecture~1]{mao2026random}. The quadrilateral signal avoids
this cancellation, and our matching impossibility result establishes
the optimality of signed quadrilateral counts for this example.

\

\noindent\textbf{Connecting the hard and smooth thresholds.}
The thresholds $n^3$ for dense hard RGGs in~\cite{bubeck2016testing}
and $n^{3/4}$ for smooth kernels with $K'(0)\ne0$
in~\cite{mao2026random} are connected by the pure odd cusp family.
For each fixed $0<\beta\le1$, Theorem~\ref{thm:cusp} gives
$d_\star=n^{3/(3\beta+1)}$. The exponent approaches $3$ as
$\beta\downarrow0$, agreeing with the jump threshold, and equals
$3/4$ at $\beta=1$, when the leading local term is linear.
Thus, within a family of fixed dense kernels, the local power
$\beta$ explains the range of threshold exponents between these
two benchmarks.

\

\noindent\textbf{Smooth kernels with $K'(0)=0$.}
When the linear term vanishes, the first nonzero higher derivative
determines the leading local signal.
For $K\in\maK_{\mathrm{smooth}}(m)$,
Theorem~\ref{thm:finite-flat} gives
$d_\star=n^{3/(3m+q_m)}$, attained by signed triangles.
The parity of the leading Taylor term determines $q_m$:
$q_m=1$ for odd $m$ and $q_m=2$ for even $m$.
These formulas continue the odd and even cusp expressions at the
level of threshold exponents. Together, the two theorems describe
how local behavior at zero determines detection across
discontinuous, nondifferentiable, and finitely flat smooth kernels.

\

\noindent\textbf{Sparse graphs and kernels depending on the dimension.}
Our results concern fixed dense kernels. When $p\to0$, many bounds
depend on $p$, and this dependence must be tracked throughout the
proof. If $K=K_d$ also varies with $d$, its local regularity at zero
alone may no longer determine the threshold. One must also control
how $K_d$ changes with $d$, particularly on the scale $d^{-1/2}$
of the latent inner products. A general characterization of the
optimal detection threshold in these settings remains open.

\section*{Acknowledgments}
Part of this work was carried out when D. Huang was visiting Yale University.
We thank Yihong Wu for helpful discussions and for suggesting
that we expand $K$ to order $m$ and compute the resulting
contributions for general kernels.

\section*{Statement on AI use}

We began by asking whether signed triangle counts are always optimal
for RGG detection. Our initial calculations suggested a counterexample
in which signed quadrilateral counts perform better.
This example
belongs to the cusp class studied in this paper, where we establish
the optimality of signed quadrilateral counts. 
We also observed that the quadrilateral signal cannot cancel because all its spectral contributions are nonnegative.
It led us to ask more
generally how the kernel determines the optimal test statistic.
We used AI tools in the following two stages of this investigation.

\begin{itemize}
\item During discussions with Yihong Wu, he suggested that we expand
a general kernel $K$ to order $m$ and compute the resulting
contributions. We pursued this approach for smooth kernels with
$K'(0)=0$ and a first nonzero derivative of finite order.
We used GPT-5.6 Sol to assist in developing the proofs for this
class. These results show that signed triangle counts remain
optimal for the finitely flat smooth kernels considered here.

\item We then returned to the question of when signed triangles
can be suboptimal and considered another direction beyond
$K'(0)\neq0$: kernels that are not differentiable at zero.
The discontinuous kernels in~\cite{bubeck2016testing}
motivated us to ask whether their detection threshold could be
connected with the smooth kernel threshold
in~\cite{mao2026random}. We used local power expansions near zero,
with terms such as $|x|^\alpha$ ($0<\alpha<1$), to describe different degrees
of regularity and formulate the cusp class.
We used GPT-6 Astra to assist in developing the proofs for this
class, including the cases where the triangle signal cancels
and signed quadrilateral counts become optimal.
\end{itemize}

We also used GPT-5.6 Sol and GPT-6 Astra to assist with figures,
LaTeX formatting, and language editing. We take responsibility
for the mathematical claims and the final presentation of this paper.

\bibliographystyle{alpha}
\bibliography{main}

\appendix

\section{Spectral asymptotics}\label{app:harmonic}

\subsection{Spectral estimates for cusp kernels}
\label{app:cusp-spectrum}

We first record the local coefficient estimates used in
Section~\ref{sec:cusp-spectrum}. Let $K$ satisfy one of the continuous
cases in Definition~\ref{def:cusp}. Write
$e_d(t)=(\kappa_d(t)+\kappa_d(-t))/2$ and
$o_d(t)=(\kappa_d(t)-\kappa_d(-t))/2$.
Only even degrees occur in $e_d$, and only odd degrees in $o_d$.
Thus $v_d=\|e_d(T_d)\|_2^2+\|o_d(T_d)\|_2^2$.

\begin{lemma}[Local harmonic coefficients]\label{gen:lem:gap-transfer}
For this $K$, we have $p_d\to p_0$,
$s_d\to\sqrt{p_0(1-p_0)}$, and~\eqref{eq:cusp-low-coefficients} holds in the stated cases.
For every fixed $p\ge1$,
\begin{align*}
 \|e_d(T_d)\|_2^2&\asymp d^{-\alpha},\quad
 \|e_d(T_d)\|_p\lesssim d^{-\alpha/2}
 \quad\text{(pure even or mixed)},\\
 \|o_d(T_d)\|_2^2&\asymp d^{-\beta},\quad
 \|o_d(T_d)\|_p\lesssim d^{-\beta/2}
 \quad\text{(pure odd or mixed)}.
\end{align*}
A component of $K$ that vanishes near zero has exponentially small
centered fixed moments and harmonic coefficients.
\end{lemma}

\begin{proof}
Let $U_d=\sqrt dT_d$. Its even moments satisfy
\[
 \E[U_d^{2j}]
 =\frac{d^j(2j-1)!!}{d(d+2)\cdots(d+2j-2)}
 \le (2j-1)!!,\quad j\ge1.
\]
Its odd moments vanish by symmetry. Hence
\[
 \E[e^{tU_d}]
 =\sum_{j\ge0}\frac{t^{2j}\E[U_d^{2j}]}{(2j)!}
 \overset{\mathrm{(a)}}{\le}e^{t^2/2},\quad
 \P[U_d>u]\overset{\mathrm{(b)}}{\le}
 \inf_{t>0}e^{-tu+t^2/2}=e^{-u^2/2}.
\]
Here $\mathrm{(a)}$ uses the even moment bound, and
$\mathrm{(b)}$ applies Markov's inequality to $e^{tU_d}$.
Applying the same bound to $-U_d$ gives
\begin{equation}
 \E[e^{tU_d}]\le e^{t^2/2},\quad
 \P[|U_d|>u]\le2e^{-u^2/2},\quad t\in\mathbb R,\quad u\ge0.
 \label{eq:kernel-bridge-sphere-tail}
\end{equation}
In particular, for fixed $r_0>0$, $p\ge1$, and $q$,
\[
 \|\indc{|T_d|\ge r_0}\|_p\le 2^{1/p}e^{-dr_0^2/(2p)},\quad
 \big|\E[\kappa_d(T_d)\psi_{q,d}(T_d)\indc{|T_d|\ge r_0}]\big|
 \lesssim e^{-dr_0^2/4}.
\]
The second bound uses Cauchy--Schwarz,
$\|\psi_{q,d}(T_d)\|_2=1$, and boundedness of $\kappa_d$.
Thus changing $K$ outside $(-r_0,r_0)$ has exponentially small effects
on the moments and coefficients under consideration.

Recall that $G\sim N(0,1)$ and $M_t=\E[|G|^t]$.
Extend $K_e$ and $K_o$ by zero outside $(0,r_0)$. We first show that
\begin{align}
 \gamma_{1,d}
 &=s_d^{-1}\E[K_o(|G|/\sqrt d)|G|]+O(d^{-\beta/2-1}),
 \label{gen:eq:first-gap-transfer}\\
 \gamma_{2,d}
 &=\frac{1+O(d^{-1})}{\sqrt2s_d}\E[K_e(|G|/\sqrt d)(G^2-1)]
   +O(d^{-\alpha/2-1}),
 \label{gen:eq:second-gap-transfer}
\end{align}
where the first relation applies in the pure odd and mixed cases,
and the second in the pure even and mixed cases.
Indeed, the density of $U_d$ is
\[
 f_d(u)=\frac{\Gamma(d/2)}{\sqrt{\pi d}\Gamma((d-1)/2)}
           (1-u^2/d)_+^{(d-3)/2}.
\]
For the standard normal density $\varphi$ and every fixed $q\ge0$,
\begin{equation}
 \int_{\mathbb R}(1+|u|^q)|f_d(u)-\varphi(u)|du=O_q(d^{-1}).
 \label{gen:eq:weighted-gaussian}
\end{equation}
Indeed, expansion of the gamma ratio and logarithm gives
$|f_d(u)-\varphi(u)|\le Cd^{-1}(1+u^4)e^{-cu^2}$
for $|u|\le d^{1/8}$. By~\eqref{eq:kernel-bridge-sphere-tail},
\[
 \int_{|u|>d^{1/8}}(1+|u|^q)(f_d(u)+\varphi(u))\dd u
 \lesssim_q e^{-cd^{1/4}},
\]
which proves~\eqref{gen:eq:weighted-gaussian}. Now use
\[
 \psi_{1,d}(u/\sqrt d)=u,\quad
 \psi_{2,d}(u/\sqrt d)
 =\frac{\sqrt{(d-1)(d+2)/2}}{d-1}(u^2-1).
\]
The centering constant cancels from the second coefficient since
$\E[U_d^2-1]=0$. The bounds
$|K_o(d^{-1/2}|u|)|\lesssim d^{-\beta/2}|u|^\beta$ and
$|K_e(d^{-1/2}|u|)|\lesssim d^{-\alpha/2}|u|^\alpha$,
together with~\eqref{gen:eq:weighted-gaussian}, prove
\eqref{gen:eq:first-gap-transfer}--\eqref{gen:eq:second-gap-transfer}.
The local expansions and dominated convergence give
\begin{align*}
 d^{\beta/2}\E[K_o(|G|/\sqrt d)|G|]
 &\longrightarrow a_oM_{\beta+1},\\
 d^{\alpha/2}\E[K_e(|G|/\sqrt d)(G^2-1)]
 &\longrightarrow a_e(M_{\alpha+2}-M_\alpha)
 \overset{\mathrm{(a)}}{=}a_e\alpha M_\alpha,
\end{align*}
where $\mathrm{(a)}$ uses $M_{\alpha+2}=(\alpha+1)M_\alpha$,
obtained by integration by parts in the Gaussian density.
Substitution in~\eqref{gen:eq:first-gap-transfer}--\eqref{gen:eq:second-gap-transfer}
proves~\eqref{eq:cusp-low-coefficients}.

Since the odd part integrates to zero,
$|p_d-p_0|\lesssim d^{-\alpha/2}$ in the pure even and mixed cases.
In the pure odd case, $|p_d-p_0|\lesssim e^{-cd}$.
The local power bounds and~\eqref{eq:kernel-bridge-sphere-tail} give
\[
 \|e_d(T_d)\|_p\lesssim d^{-\alpha/2},\quad
 \|o_d(T_d)\|_p\lesssim d^{-\beta/2}.
\]
For the even part this includes its centering correction.
The matching lower bounds follow from Parseval and
\eqref{eq:cusp-low-coefficients}:
\[
 \|e_d(T_d)\|_2^2\ge\gamma_{2,d}^2\asymp d^{-\alpha},\quad
 \|o_d(T_d)\|_2^2\ge\gamma_{1,d}^2\asymp d^{-\beta}.
\]
Each bound applies when the corresponding local component is nonzero.
When it vanishes,~\eqref{eq:kernel-bridge-sphere-tail} gives exponentially
small moments and coefficients, including the centering correction.
\end{proof}

\begin{proof}[Proof of Proposition~\ref{gen:prop:spectral-scales}]
For the continuous cases, we separate the two parities in
\eqref{eq:trace-formula}. The harmonic dimensions increase with
degree, so for an integer $Q\ge1$, a set
$I\subseteq\{Q,Q+1,\ldots\}$, and $\ell\ge2$,
\[
 \sum_{q\in I}|\gamma_{q,d}|^\ell N_{q,d}^{1-\ell/2}
 \overset{\mathrm{(a)}}{\le}N_{Q,d}^{1-\ell/2}
       \sum_{q\in I}|\gamma_{q,d}|^\ell
 \overset{\mathrm{(b)}}{\le}N_{Q,d}^{1-\ell/2}
       \Bigg(\sum_{q\in I}\gamma_{q,d}^2\Bigg)^{\ell/2}.
\]
Here $\mathrm{(a)}$ uses $N_{q,d}\ge N_{Q,d}$ and $\ell\ge2$,
and $\mathrm{(b)}$ uses the monotonicity of sequence norms.
Apply this bound to odd degrees $q\ge3$ and even degrees $q\ge4$:
\begin{align}
 \tau_{3,d}
 &=\frac{\gamma_{1,d}^3}{\sqrt d}
   +\frac{\gamma_{2,d}^3}{\sqrt{N_{2,d}}}
   +O(\|o_d(T_d)\|_2^3d^{-3/2}+\|e_d(T_d)\|_2^3d^{-2}),
 \label{gen:eq:cubic-decomposition}\\
 \tau_{4,d}
 &=\frac{\gamma_{1,d}^4}{d}
   +\frac{\gamma_{2,d}^4}{N_{2,d}}
   +O(\|o_d(T_d)\|_2^4d^{-3}+\|e_d(T_d)\|_2^4d^{-4}).
 \label{gen:eq:quartic-decomposition}
\end{align}
By Lemma~\ref{gen:lem:gap-transfer}, the fourth trace is comparable
to the sum of its first two terms. The operator norm satisfies
\[
 \max\Bigg\{\frac{\gamma_{1,d}^2}{d},
              \frac{\gamma_{2,d}^2}{N_{2,d}}\Bigg\}
 \le\|T_{\kappa_d}\|_{\op}^2
 \le\max\Bigg\{\frac{\|o_d(T_d)\|_2^2}{d},\frac{\|e_d(T_d)\|_2^2}{N_{2,d}}\Bigg\}.
\]
The same lemma and $N_{2,d}\asymp d^2$ give the asserted fourth
trace and operator norm scales in the pure even, pure odd, and
mixed cases. In the pure cases, the other parity contributes
only exponentially small errors. In the mixed case with
$\beta\ne\alpha+1/3$, the first two terms of
\eqref{gen:eq:cubic-decomposition} have different orders, and the
remainder is smaller than the dominant term. This proves the
pure cubic scales and~\eqref{gen:eq:nonbalanced-cubic}.

If $\beta=\alpha+1/3$, the first two cubic terms have the same
order. Since $N_{2,d}=d^2(1+O(d^{-1}))/2$,
\eqref{eq:cusp-low-coefficients} and
\eqref{gen:eq:cubic-decomposition} give
\[
 s_d^3\tau_{3,d}
 =d^{-(3\alpha+2)/2}
  \Bigg(a_o^3M_{\beta+1}^3+\frac{(a_e\alpha M_\alpha)^3}{2}+o(1)\Bigg)
  +O(d^{-3\alpha/2-2}).
\]
Here the last term bounds
$\|o_d(T_d)\|_2^3d^{-3/2}+\|e_d(T_d)\|_2^3d^{-2}$.
Keeping the common factor $s_d^{-3}$ proves
\eqref{eq:cusp-critical-coefficient} even when the two leading
constants cancel.

Finally, consider a jump, and let $\varepsilon$ be the sign in
Definition~\ref{def:cusp}. Recall that $U_d=\sqrt dT_d$.
By symmetry,
\begin{align*}
 \varepsilon s_d\gamma_{1,d}
 &=\frac12\E\big[\varepsilon\{K(|T_d|)-K(-|T_d|)\}|U_d|\big]\overset{\mathrm{(a)}}{\ge}\frac{c_1}{2}
   \E[|U_d|\indc{|T_d|<r_0}]
   -C\E[|U_d|\indc{|T_d|\ge r_0}]
 \overset{\mathrm{(b)}}{\gtrsim}1.
\end{align*}
Here $\mathrm{(a)}$ uses the jump condition and boundedness of $K$.
For $\mathrm{(b)}$, use~\eqref{eq:kernel-bridge-sphere-tail}
and $\E[|U_d|]\to M_1>0$.
The boundedness of $\kappa_d$ then gives
$\gamma_{1,d}^2\asymp v_d\asymp 1$.
Applying the same dimension bound to degrees $q\ge2$ yields
\[
 \tau_{3,d}=\frac{\gamma_{1,d}^3}{\sqrt d}+O(d^{-1}),\quad
 \tau_{4,d}=\frac{\gamma_{1,d}^4}{d}+O(d^{-2}),\quad
 \frac{\gamma_{1,d}^2}{d}
 \le\|T_{\kappa_d}\|_{\op}^2\le\frac{v_d}{d}.
\]
These are the jump assertions and complete the proof.
\end{proof}

\subsection{Finitely flat spectral asymptotics}\label{app:proof-spectrum}

We prove Proposition~\ref{prop:spectrum} by finding the leading
coefficients and bounding the remaining degrees.

\begin{proof}[Proof of Proposition~\ref{prop:spectrum}]
Since $P_{1,d}(t)=t$, Lemma~\ref{lem:gegen-moment} gives
$\|T_d\|_p=O(d^{-1/2})$ for every fixed $p\ge2$.
Taylor's theorem and centering therefore yield
\begin{equation}
\Big\|
\kappa_d(T_d)
-\frac{K^{(m)}(0)}{m!s_d}
 \big(T_d^m-\E[T_d^m]\big)
\Big\|_2
=
O\big(d^{-(m+1)/2}\big).
\label{eq:smooth-centered-remainder}
\end{equation}

Under the coupling used in Lemma~\ref{lem:gegen-hermite},
$\sqrt dT_d\to G$ almost surely, where $G\sim\mathcal N(0,1)$
is a standard normal random variable.
The moment bounds from Lemma~\ref{lem:gegen-moment} imply
$(\sqrt dT_d)^m\to G^m$ in $L^2$.
Let $\operatorname{He}_q$ denote the monic Hermite
polynomial of degree $q$ orthogonal with respect to the standard Gaussian measure.
Since $p_d\to K(0)$ and
$s_d\to\sqrt{K(0)(1-K(0))}$,
combining~\eqref{eq:smooth-centered-remainder} with
Lemma~\ref{lem:gegen-hermite} and the Cauchy--Schwarz inequality gives,
for $q=m,m-2,\ldots,q_m$,
\begin{align*}
 d^{m/2}\gamma_{q,d}
 &\longrightarrow
 \frac{K^{(m)}(0)\E[G^m\operatorname{He}_q(G)]}
 {m!\sqrt{q!K(0)(1-K(0))}}=\frac{K^{(m)}(0)\E[G^{m-q}]}
 {(m-q)!\sqrt{q!K(0)(1-K(0))}}\ne0.
\end{align*}
The centering term vanishes because
$\E[\operatorname{He}_q(G)]=0$.
The equality follows by integrating the Gaussian density by parts
$q$ times, and the limit is nonzero because $m-q$ is even.
This proves $|\gamma_{q,d}|\asymp d^{-m/2}$ for the stated degrees.
Moreover, the Taylor estimate and Parseval's identity give
\begin{equation}
\sum_{q\ge1}\gamma_{q,d}^2
=
\E[\kappa_d(T_d)^2]
=
O(d^{-m}).
\label{eq:smooth-coefficient-energy}
\end{equation}

\noindent\textbf{Odd $m$.} Fix an integer $\ell\ge3$.
Here $q_m=1$.
By~\eqref{eq:trace-formula} and
\eqref{eq:smooth-coefficient-energy},
\[
 \big|
\tau_{\ell,d}-\gamma_{1,d}^{\ell}d^{1-\ell/2}
\big|
\le
N_{2,d}^{1-\ell/2}
\big(\sum_{q\ge2}\gamma_{q,d}^2\big)^{\ell/2}=
O\big(d^{-m\ell/2+2-\ell}\big).
\]
Here $N_{q,d}\ge N_{2,d}\asymp d^2$ for $q\ge2$, and
$\sum_q|\gamma_{q,d}|^\ell\le(\sum_q\gamma_{q,d}^2)^{\ell/2}$.
Since
$|\gamma_{1,d}|^\ell d^{1-\ell/2}
\asymp d^{-m\ell/2+1-\ell/2}$,
the relative error is
$O(d^{1-\ell/2})=O(d^{-1/2})$, as required.
 
\noindent\textbf{Even $m$.} Here $q_m=2$.
The centered Taylor term of degree $m$ is even and hence orthogonal
to $\psi_{1,d}(t)=\sqrt dt$.
Thus~\eqref{eq:smooth-centered-remainder} and
the Cauchy--Schwarz inequality imply
$\gamma_{1,d}=O(d^{-(m+1)/2})$.
Separating degree one from the degrees at least three in
\eqref{eq:trace-formula}, we obtain
\begin{align*}
 \big|\tau_{\ell,d}-\gamma_{2,d}^{\ell}N_{2,d}^{1-\ell/2}\big|
 &\overset{\mathrm{(a)}}{\le}|\gamma_{1,d}|^\ell d^{1-\ell/2}
 +N_{3,d}^{1-\ell/2}
 \Big(\sum_{q\ge3}\gamma_{q,d}^2\Big)^{\ell/2}\overset{\mathrm{(b)}}{\lesssim}
 d^{-m\ell/2+1-\ell}+d^{-m\ell/2+3-3\ell/2}.
\end{align*}
Here $\mathrm{(a)}$ uses $N_{q,d}\ge N_{3,d}$ and the monotonicity
of sequence norms. For $\mathrm{(b)}$, use
$\gamma_{1,d}=O(d^{-(m+1)/2})$,
$N_{3,d}\asymp d^3$, and~\eqref{eq:smooth-coefficient-energy}.
Since
$|\gamma_{2,d}|^\ell N_{2,d}^{1-\ell/2}
\asymp d^{-m\ell/2+2-\ell}$,
the relative error is
$O(d^{-1}+d^{1-\ell/2})=O(d^{-1/2})$. This proves~\eqref{eq:prop-finite-positive-flat-1}.

Finally, taking $\ell=3,4$ and using
$|\gamma_{q_m,d}|\asymp d^{-m/2}$ and
$N_{q_m,d}\asymp d^{q_m}$ gives the asserted estimates for
$|\tau_{3,d}|$ and $\tau_{4,d}$.
This completes the proof.
\end{proof}
The following lemmas give the moment bounds and Gaussian limit used
in the proof of Proposition~\ref{prop:spectrum}.

\begin{lemma}[Gegenbauer moments at fixed degree]\label{lem:gegen-moment}
For every fixed integer $r\ge1$, there are constants
$c_{\ref{lem:gegen-moment}}(r),C_{\ref{lem:gegen-moment}}(r)>0$
such that, for $2\le p\le c_{\ref{lem:gegen-moment}}(r)d$,
\begin{equation}
 \|P_{r,d}(T_d)\|_p\le C_{\ref{lem:gegen-moment}}(r)\Big(\frac pd\Big)^{r/2}.       \label{eq:tr-8a}
\end{equation}
\end{lemma}

\begin{proof}
The normalized Gegenbauer recurrence is
\begin{equation}
 P_{r+1,d}(t)
 =\frac{d+2r-2}{d+r-2}tP_{r,d}(t)
  -\frac r{d+r-2}P_{r-1,d}(t).
\label{eq:gegenbauer-recurrence}
\end{equation}
Induction on $r$ gives
\[
 P_{r,d}(t)=\sum_{j=0}^{\lfloor r/2\rfloor}
 c_{r,j}(d)t^{r-2j},\quad |c_{r,j}(d)|\le Cd^{-j}.
\]
By~\eqref{eq:kernel-bridge-sphere-tail},
$\|T_d\|_v\le C\sqrt{v/d}$ for $v\ge2$. Since $r$ is fixed,
\begin{align*}
 \|P_{r,d}(T_d)\|_p
 &\overset{\mathrm{(a)}}{\le}C\sum_{j=0}^{\lfloor r/2\rfloor}
       d^{-j}\Big(\frac pd\Big)^{(r-2j)/2}=C\Big(\frac pd\Big)^{r/2}
       \sum_{j=0}^{\lfloor r/2\rfloor}p^{-j}
 \overset{\mathrm{(b)}}{\lesssim}\Big(\frac pd\Big)^{r/2}.
\end{align*}
Here $\mathrm{(a)}$ uses the triangle inequality and
$\|T_d^a\|_p=\|T_d\|_{ap}^a$ for $a\ge1$; the constant term is
bounded directly. Also, $\mathrm{(b)}$ uses $p\ge2$.
This proves~\eqref{eq:tr-8a}.
\end{proof}

We choose $c_{\ref{lem:gegen-moment}}(r)\le1$ nonincreasing in $r$,
by replacing it with the minimum over degrees at most $r$.
To compare Gegenbauer and Hermite polynomials, let
$G_1,G_2,\ldots$ be independent standard Gaussian variables and define
\[
 U_d=\frac{\sqrt dG_1}{(\sum_{j=1}^dG_j^2)^{1/2}}
 \overset{\mathrm d}=\sqrt dT_d,\quad G=G_1.
\]
Recall that $\psi_{q,d}(t)=\sqrt{N_{q,d}}P_{q,d}(t)$.

\begin{lemma}[Gegenbauer--Hermite limit]\label{lem:gegen-hermite}
Under this coupling, for every fixed $q\ge0$ and $1\le r<\infty$,
\[
 \psi_{q,d}(U_d/\sqrt d)
 \longrightarrow \frac{\operatorname{He}_q(G)}{\sqrt{q!}}
 \quad\text{in }L^r.
\]
\end{lemma}

\begin{proof}
Start from $P_{0,d}=1$, $P_{1,d}(t)=t$, and use
\eqref{eq:gegenbauer-recurrence}.
Since $N_{q,d}=d^q/q!(1+O_q(d^{-1}))$, after substituting
$t=u/\sqrt d$ and multiplying by $\sqrt{N_{q+1,d}}$, the two coefficients
converge to $1/\sqrt{q+1}$ and $\sqrt{q/(q+1)}$.  The limiting recurrence is
\[
 \frac{\operatorname{He}_{q+1}(u)}{\sqrt{(q+1)!}}
 =\frac{u}{\sqrt{q+1}}\frac{\operatorname{He}_q(u)}{\sqrt{q!}}
 -\sqrt{\frac q{q+1}}
  \frac{\operatorname{He}_{q-1}(u)}{\sqrt{(q-1)!}}.
\]
Induction gives coefficientwise convergence, and the strong law gives
$U_d\to G_1$ almost surely. By~\eqref{eq:kernel-bridge-sphere-tail},
$\sup_d\E[|U_d|^a]<\infty$ for every fixed $a>0$.
Thus the polynomial convergence also holds in $L^r$.
\end{proof}

\section{Cavity identities and concentration tools}\label{app:information}

We prove the Gram matrix and concentration bounds used in
Section~\ref{sec:polynomial}. The posterior bounds for cusp kernels
are proved in Appendix~\ref{app:cusp}.

\subsection{Uniform Gram operators for harmonic features}
\label{app:events}

The same harmonic Gram event enters both the signed block inequality and the
posterior bounds for finite harmonic expansions.  We construct it once here under the prior.

\begin{lemma}[A uniform bound for Gram matrices]\label{lem:gram-event}
Let $z_1,\ldots,z_n$ be independent unit vectors in a Euclidean space
of dimension $N$ with $\E[z_i\otimes z_i]=I/N$, and let
$\mathbf U=(\langle z_i,z_j\rangle)_{i,j\le n}$.  For every fixed $D>0$, with
probability at least $1-(n+N)^{-D}$,
\begin{equation}
 \|\mathbf U\|_{\op}
 \le C_{\ref{lem:gram-event}}(D)\{1+n/N+\log(n+N)\}.                            \label{eq:mc-6a}
\end{equation}
\end{lemma}

\begin{proof}
Let $Z:\mathbb R^n\to\mathbb R^N$ have columns $z_i$.  The nonzero
eigenvalues of $\mathbf U=Z^*Z$ are those of
$ZZ^*=\sum_{i=1}^n z_i\otimes z_i$.
Let $Y_i=z_i\otimes z_i-I/N$. Since $\|z_i\|=1$,
\[
 \E[Y_i]=0,\quad \|Y_i\|_{\op}\le1,\quad
 \E[Y_i^2]=\Big(\frac1N-\frac1{N^2}\Big)I.
\]
Thus $\big\|\sum_i\E[Y_i^2]\big\|_{\op}\le n/N$.
The matrix Bernstein inequality for symmetric matrices
\cite[Theorem~1.6.2]{tropp2015introduction} therefore gives, for
$u\ge1$,
\[
 \P\big[\big\|\sum_iY_i\big\|_{\op}
 >C\big(\sqrt{nu/N}+u\big)\big]\le2Ne^{-u}.
\]
Choose $u=(D+3)\log(n+N)$. Outside an event of probability at most
$2Ne^{-u}\le(n+N)^{-D}$,
\[
 \|\mathbf U\|_{\op}
 \overset{\mathrm{(a)}}{\le}\frac nN+C\Big(\sqrt{\frac{nu}{N}}+u\Big)
 \overset{\mathrm{(b)}}{\lesssim}\frac nN+u.
\]
Here $\mathrm{(a)}$ adds the mean $nI/N$, and
$\mathrm{(b)}$ uses $2\sqrt{ab}\le a+b$.
This proves~\eqref{eq:mc-6a}.
\end{proof}

Recall the finite degree set $\mathcal I$ and the event
$\mathcal G_D$ in~\eqref{eq:smooth-gram-event}.
Choose an orthonormal basis $(Y_{q,a})_{a\le N_{q,d}}$ of the
harmonics of degree $q$, and let
$z_q(x)=N_{q,d}^{-1/2}(Y_{q,a}(x))_{a\le N_{q,d}}$.
The addition formula gives $U_q(x,y)=\langle z_q(x),z_q(y)\rangle$,
$\|z_q(x)\|=1$, and $\E[z_q(X)z_q(X)^\top]=I/N_{q,d}$.
Thus Lemma~\ref{lem:gram-event} applies to each $\mathbf U_q$.
Since $\mathcal I$ is fixed and $N_{q,d}\asymp d^q$,
there is a fixed $D'>D$ such that
\[
 \sum_{q\in\mathcal I}(n+N_{q,d})^{-D'}
 \le(n+N_{\max})^{-D}.
\]
Apply the lemma with $D'$ and take a union bound. After increasing
the constant in~\eqref{eq:smooth-gram-event} to
$C_{\ref{lem:gram-event}}(D,\mathcal I)$, we obtain the claimed
probability $\P[\mathcal G_D^c]\le(n+N_{\max})^{-D}$.

\subsection{Concentration inequalities for posterior overlaps}
\label{app:tools}

We first record the polynomial moment and $U$-statistic bounds.
We then bound the exponential moments of the harmonic overlaps
in~\eqref{eq:smooth-quadratic-form}.

\begin{lemma}[Hypercontractivity on a product of spheres]
\label{lem:spherical-hypercontractivity}
Fix integers $s,L\ge1$.  If $F$ is a polynomial on
$(\Sph^{d-1})^s$ of total degree in spherical harmonics at most $L$, then, for
every $p\ge2$,
\begin{equation}
 \|F\|_{L^p((\sigma_d)^s)}
 \le (p-1)^{L/2}\|F\|_{L^2((\sigma_d)^s)}.               \label{eq:app-hc}
\end{equation}
In particular, degree at most $L$ in each variable gives the bound
$(p-1)^{sL/2}\|F\|_2$.
\end{lemma}

\begin{proof}
Let $\mathsf P_\rho$ be the product Poisson semigroup, which multiplies
a harmonic component of total degree $j$ by $\rho^j$.
At $\rho=(p-1)^{-1/2}$, \cite[Corollary~11]{montanaro2012hypercontractive}
gives $\|\mathsf P_\rho G\|_p\le\|G\|_2$. This is the product form of
the spherical hypercontractive inequality in~\cite{beckner1992sobolev}.
Writing $F=\sum_{j=0}^L F_j$ as an orthogonal harmonic sum, we obtain
\[
 \|F\|_p
 \overset{\mathrm{(a)}}{\le}\Big\|\sum_{j=0}^L\rho^{-j}F_j\Big\|_2
 \overset{\mathrm{(b)}}{\le}\rho^{-L}\|F\|_2
 =(p-1)^{L/2}\|F\|_2.
\]
Here $\mathrm{(a)}$ applies the semigroup bound to
$\sum_j\rho^{-j}F_j$, and $\mathrm{(b)}$ uses orthogonality
and $0<\rho\le1$. For coordinate degrees at most $L$, the total degree is at most $sL$.
\end{proof}

The next bound estimates a product of Gaussian factors.
Let $(B_1,\ldots,B_N)$ be centered jointly Gaussian vectors in a
Euclidean space $E$ of finite dimension, with joint covariance
$C=\operatorname{Cov}(B_1,\ldots,B_N)$, and let $\gamma\sim N(0,I_E)$.
The covariance $C$ may be singular.

\begin{lemma}[Gaussian covariance bound]\label{lem:gaussian-holder}
If $\operatorname{Cov}(B_i)=I_E$ for all $i$ and
$C\preceq pI_{E^N}$ for some $p\ge1$, then, for all nonnegative
measurable $f_i$,
\begin{equation}
 \E\Bigg[\prod_{i=1}^N f_i(B_i)\Bigg]
 \le \prod_{i=1}^N\big(\E[f_i(\gamma)^p]\big)^{1/p}.
 \label{eq:t-1}
\end{equation}
\end{lemma}

\begin{proof}
Let $M=N\dim(E)$.  Let $G\sim N(0,I_M)$ and realize
$B=C^{1/2}G$.  If $\pi_i:E^N\to E$ is the $i$th coordinate projection,
let
$L_i=\pi_iC^{1/2}:\mathbb R^M\longrightarrow E$.
The diagonal covariance assumption says $L_iL_i^*=I_E$, while
$\sum_{i=1}^N L_i^*L_i=C$.
Diagonalize $I_M-C/p=\sum_{a=1}^M\lambda_a u_a\otimes u_a$,
where $\lambda_a\ge0$ and the $u_a$ are orthonormal. Then
\begin{equation}
 \frac1p\sum_{i=1}^N L_i^*L_i
 +\sum_{a=1}^M\lambda_a u_a\otimes u_a=I_M.                    \label{eq:t-1a}
\end{equation}
First suppose that the $f_i$ are bounded and compactly supported.
Apply~\cite[Proposition~2.8]{bennett2008brascamp} to the maps
$L_i$ and $x\mapsto\langle u_a,x\rangle$, with weights
$1/p$ and $\lambda_a$. Equation~\eqref{eq:t-1a} gives
\begin{align}
 \int_{\mathbb R^M}e^{-\|x\|^2/2}
       \prod_{i=1}^N f_i(L_ix)\dd x
 \overset{\mathrm{(a)}}{\le}&~
 \prod_{i=1}^N\Big(\int_E f_i(y)^p e^{-\|y\|^2/2}\dd y\Big)^{1/p}
 \prod_{a=1}^M\Big(\int_{\mathbb R}e^{-t^2/2}\dd t\Big)^{\lambda_a}\notag\\
 \overset{\mathrm{(b)}}{=}&~(2\pi)^{M/2}
 \prod_{i=1}^N\big(\E[f_i(\gamma)^p]\big)^{1/p}.
 \label{eq:t-1b}
\end{align}
Here $\mathrm{(a)}$ is the geometric Brascamp--Lieb inequality;
\eqref{eq:t-1a} combines its Gaussian factors into
$e^{-\|x\|^2/2}$. For $\mathrm{(b)}$, take traces in~\eqref{eq:t-1a}
to get $\sum_a\lambda_a=M-M/p$.
Divide by $(2\pi)^{M/2}$ and use $B=C^{1/2}G$ to obtain~\eqref{eq:t-1}.
Monotone truncation handles general nonnegative $f_i$.
No inverse of $C$ was used, so the proof includes singular $C$.
\end{proof}

Let $k\ge2$ be fixed, let $(\Omega,\mu)$ be a probability space, and let
$h:\Omega^k\to\mathbb R$ be bounded, symmetric, and canonical:
\[
 \int h(x_1,\ldots,x_k)\mu(\dd x_j)=0
 \quad\text{for every }j,
 \quad\mu^{k-1}\text{-a.e. in the remaining variables}.             
\]
Let $c_{i_1,\ldots,i_k}$ be deterministic symmetric coefficients,
extended by zero to tuples with repeated indices.
For $I\subseteq[k]$, write $I^c=[k]\setminus I$,
$x_I=(x_j)_{j\in I}$, and $\mu^I=\bigotimes_{j\in I}\mu$.
Define
\[
 M=\max_{i_1<\cdots<i_k}|c_{i_1,\ldots,i_k}|,
 \quad
 a_r=\max_{\substack{I\subseteq[k]\\|I|=r}}\esssup_{x_{I^c}}
       \|h(x_I,x_{I^c})\|_{L^2(\mu^I)},\quad 0\le r\le k.
\]
In particular, $a_0=\|h\|_\infty$ and
$a_k=\|h\|_{L^2(\mu^{\otimes k})}$.
The next bound will be used conditionally on one latent sample.
It depends only on the largest coefficient and the $L^2$ norms
obtained by integrating some of the variables of $h$.

\begin{lemma}[A bound for canonical $U$-statistics]
\label{lem:adamczak-cor}
Let $W_1,\ldots,W_n$ be independent with law $\mu$, and define
$V=\sum_{i_1<\cdots<i_k}c_{i_1,\ldots,i_k}h(W_{i_1},\ldots,W_{i_k})$.
If $t\ge\max_{0\le r\le k}Mn^{r/2}a_r$, then
\begin{equation}
 \P[|V|\ge t]
 \le C_{\ref{lem:adamczak-cor}}(k)\exp\Bigg(
 -c_{\ref{lem:adamczak-cor}}(k)\min_{0\le r\le k}
 \Big(\frac{t}{Mn^{r/2}a_r}\Big)^{2/(2k-r)}\Bigg).
 \label{eq:t-7}
\end{equation}
The bound also holds conditionally on random coefficients independent
of the sample.
\end{lemma}

\begin{proof}
Write $\mathbf i=(i_1,\ldots,i_k)\in[n]^k$ and
$\mathbf i_I=(i_j)_{j\in I}$; sums over $\mathbf i_I$ range over $[n]^I$.
Let $\Part(I)$ be the partitions of $I\subseteq[k]$ into nonempty
blocks, with $\Part(\varnothing)=\{\varnothing\}$.
For $\mathcal J=\{J_1,\ldots,J_\ell\}\in\Part(I)$, define
\begin{equation}
 \mathcal N_{I,\mathcal J}
 =\esssup_{x_{I^c}}\max_{\mathbf i_{I^c}}\sup_f
 \Bigg|\sum_{\mathbf i_I}\int c_{\mathbf i}h(x_I,x_{I^c})
 \prod_{a=1}^{\ell}f^{(a)}_{\mathbf i_{J_a}}(x_{J_a})
 \mu^I(\dd x_I)\Bigg|,
 \label{eq:t-3}
\end{equation}
where the supremum is over families satisfying
$\sum_{\mathbf i_{J_a}}\|f^{(a)}_{\mathbf i_{J_a}}\|_2^2\le1$
for each $a$. For $I=\varnothing$, this means
$\mathcal N_{\varnothing,\varnothing}
 =\max_{\mathbf i}\|c_{\mathbf i}h\|_\infty$.
Apply~\cite[Theorem~1]{delapena1995decoupling} to the ordered sum
$k!V$, then~\cite[Theorem~7]{adamczak2006moment} to its decoupled form.
Absorbing the factors depending only on $k$ gives
\begin{equation}
 \P[|V|\ge t]
 \le C_{\ref{lem:adamczak-cor}}(k)\exp\Bigg(
 -c_{\ref{lem:adamczak-cor}}(k)
 \min_{I\subseteq[k]}\min_{\mathcal J\in\Part(I)}
 \Big(\frac{t}{\mathcal N_{I,\mathcal J}}\Big)^{
 2/(|\mathcal J|+2|I^c|)}\Bigg).
 \label{eq:t-4}
\end{equation}
Fix $I,\mathcal J$, the frozen indices, and the frozen variables.
Cauchy--Schwarz in the remaining indices and variables gives
\begin{align}
 \mathcal N_{I,\mathcal J}
 &\overset{\mathrm{(a)}}{\le}\esssup_{x_{I^c}}\max_{\mathbf i_{I^c}}
 \Bigg(\sum_{\mathbf i_I}\int
 |c_{\mathbf i}h(x_I,x_{I^c})|^2\mu^I(\dd x_I)\Bigg)^{1/2}\overset{\mathrm{(b)}}{\le}Mn^{r/2}a_r,\quad r=|I|.
 \label{eq:t-6}
\end{align}
Here $\mathrm{(a)}$ is Cauchy--Schwarz. Its second factor is at most
one because the blocks of $\mathcal J$ partition $I$:
\[
 \sum_{\mathbf i_I}\int
 \prod_{a=1}^{\ell}|f^{(a)}_{\mathbf i_{J_a}}(x_{J_a})|^2
 \mu^I(\dd x_I)
 =\prod_{a=1}^{\ell}\sum_{\mathbf i_{J_a}}
 \|f^{(a)}_{\mathbf i_{J_a}}\|_2^2\le1.
\]
For $\mathrm{(b)}$, use $|c_{\mathbf i}|\le M$ and at most $n^r$
choices of $\mathbf i_I$. Since $|\mathcal J|\le r$ and
$t/(Mn^{r/2}a_r)\ge1$,
\[
 \Big(\frac{t}{\mathcal N_{I,\mathcal J}}\Big)^{2/(|\mathcal J|+2|I^c|)}
 \ge\Big(\frac{t}{Mn^{r/2}a_r}\Big)^{2/(2k-r)}.
\]
Substitution in~\eqref{eq:t-4} proves~\eqref{eq:t-7}.
\end{proof}

At order two, we use a bound that retains four different norms.
Let $h:\Omega^2\to\mathbb R$ be bounded, symmetric, and canonical,
and let $W_1,\ldots,W_n$ be independent with law $\mu$.
Define $U_n(h)=\sum_{i<j}h(W_i,W_j)$ and
\begin{align*}
 A_h&=\|h\|_\infty,\quad
 B_h^2=n\esssup_x\int h(x,y)^2\mu(\dd y),\\
 C_h^2&=n^2\int h(x,y)^2\mu(\dd x)\mu(\dd y),\quad
 D_h=n\|T_h\|_{L^2(\mu)\to L^2(\mu)},
\end{align*}
where $T_hf(x)=\int h(x,y)f(y)\mu(\dd y)$.

\begin{lemma}[Gin\'e--Lata\l a--Zinn bound at order two]
\label{lem:glz}
For this kernel $h$, there are universal constants $c,C>0$ such that,
for all $u>0$,
\begin{equation}
 \P[|U_n(h)|\ge u]
 \le C\exp\Big(-c\min\Big\{\frac{u^2}{C_h^2},\frac{u}{D_h},
   \frac{u^{2/3}}{B_h^{2/3}},\frac{u^{1/2}}{A_h^{1/2}}\Big\}
 \Big).                                                \label{eq:t-8}
\end{equation}
\end{lemma}

\begin{proof}
Apply~\cite[Theorem~1]{delapena1995decoupling} to
$2U_n(h)=\sum_{i\ne j}h(W_i,W_j)$, then use
\cite[Theorem~3.3]{gine2000exponential}.
For the decoupled kernel matrix
$(\indc{i\ne j}h)_{i,j\le n}$, its four parameters are bounded by
$A_h,B_h,C_h,D_h$. In particular, the operator parameter is
$\|\mathbf 1\mathbf 1^\top-I_n\|_{\op}\|T_h\|_{\op}
=(n-1)\|T_h\|_{\op}\le D_h$.
This proves~\eqref{eq:t-8} after adjusting the universal constants.
\end{proof}

We next estimate the overlaps $O_{q\ell}$
from~\eqref{eq:smooth-quadratic-form}. Recall that
$U_q(x,y)=\langle z_q(x),z_q(y)\rangle$, with
$\|z_q(x)\|=1$ and $\E[z_q(X)z_q(X)^\top]=I/N_{q,d}$.
All harmonic degrees below are fixed. We first bound the exponential
moment of a smooth function on a product of spheres.
Gradients, Hessians, and distances below use the standard product metric.

\begin{lemma}[A smooth exponential moment bound]\label{lem:smooth-t2-mgf}
Let $\mu=\sigma_d^{\otimes Q}$ for integers $d\ge3$ and $Q\ge1$,
and let $f$ be a smooth function with $\E_\mu[f]=0$ and
$\|\operatorname{Hess}f\|_{\op}\le L$ at every point.
For $|t|L<d-2$,
\begin{equation}
 \log\E_\mu[e^{tf}]
 \le\frac{t^2\E_\mu[|\nabla f|^2]}{2(d-2-|t|L)}.
 \label{eq:mc-1}
\end{equation}
\end{lemma}

\begin{proof}
The product sphere has Ricci curvature $(d-2)I$. Hence
\cite[Theorems~1 and~2]{otto2000generalization} gives
$W_2^2(\nu,\mu)\le2\KL(\nu\Vert\mu)/(d-2)$ for every probability measure $\nu$,
where $W_2$ is the quadratic Wasserstein distance for the product metric.
For an optimal coupling of $X\sim\mu$ and $Y\sim\nu$, Taylor's
inequality along a minimizing geodesic gives
\begin{align*}
 \E_\nu[f]
 &\overset{\mathrm{(a)}}{\le}\sqrt{\E_\mu[|\nabla f|^2]}W_2(\nu,\mu)
       +\frac L2W_2^2(\nu,\mu)\overset{\mathrm{(b)}}{\le}
 \sqrt{\frac{2\E_\mu[|\nabla f|^2]\KL(\nu\Vert\mu)}{d-2}}
       +\frac L{d-2}\KL(\nu\Vert\mu).
\end{align*}
Here $\mathrm{(a)}$ uses $\E_\mu[f]=0$, the Hessian bound, and
Cauchy--Schwarz; $\mathrm{(b)}$ uses the transport inequality.
For $t>0$, the entropy variational formula therefore yields
\begin{align*}
 \log\E_\mu[e^{tf}]
 &\overset{\mathrm{(c)}}{=}
 \sup_{\nu\ll\mu}\big(t\E_\nu[f]-\KL(\nu\Vert\mu)\big)\\
 &\le\sup_{a\ge0}\Bigg(
 t\sqrt{\frac{2\E_\mu[|\nabla f|^2]}{d-2}}a
 -\Big(1-\frac{tL}{d-2}\Big)a^2\Bigg)=\frac{t^2\E_\mu[|\nabla f|^2]}{2(d-2-tL)}.
\end{align*}
For $\mathrm{(c)}$, Jensen's inequality gives the upper bound,
with equality at $\dd\nu/\dd\mu=e^{tf}/\E_\mu[e^{tf}]$.
The next line sets $a=\sqrt{\KL(\nu\Vert\mu)}$ and maximizes over
all $a\ge0$. Apply the same argument to $-f$ when $t<0$.
\end{proof}

We apply Lemma~\ref{lem:smooth-t2-mgf} to signed sums of harmonic kernels.
For independent $X_1,\ldots,X_Q\sim\sigma_d$ and signs
$\epsilon_a\in\{-1,1\}$, define
$H_{q,Q,\epsilon}=\sum_{a<b}\epsilon_a\epsilon_bU_q(X_a,X_b)$.

\begin{lemma}[Signed replica MGF]\label{lem:signed-replica-mgf}
For each fixed $q\ge1$, there are constants
$c_{\ref{lem:signed-replica-mgf}}(q),C_{\ref{lem:signed-replica-mgf}}(q)>0$
such that, for $d\ge3$, $Q\ge2$, all choices of signs, and
$|t|Q<c_{\ref{lem:signed-replica-mgf}}(q)d$,
\begin{equation}
 \log\E[e^{tH_{q,Q,\epsilon}}]
 \le C_{\ref{lem:signed-replica-mgf}}(q)\frac{t^2Q^2}{N_{q,d}}.                         \label{eq:mc-4}
\end{equation}
\end{lemma}

\begin{proof}
Let $\lambda_q=q(q+d-2)$. Each edge term is an eigenfunction of
the negative Laplacian on the product of spheres, with eigenvalue $2\lambda_q$. Distinct edge terms are orthogonal:
for edges sharing one endpoint, condition on that endpoint and use
centering. Thus
\[
 \E[|\nabla H_{q,Q,\epsilon}|^2]
 \overset{\mathrm{(a)}}{=}2\lambda_q\E[H_{q,Q,\epsilon}^2]
 =\frac{2\lambda_q}{N_{q,d}}\binom Q2
 \lesssim\frac{dQ^2}{N_{q,d}},
\]
where $\mathrm{(a)}$ is integration by parts on the product sphere.
Induction in~\eqref{eq:gegenbauer-recurrence} bounds the coefficients
of $P_{q,d}$ uniformly for fixed $q$ and $d\ge3$. Hence
$\sup_{|u|\le1}(|P_{q,d}'(u)|+|P_{q,d}''(u)|)\le C(q)$.
The spherical first and second derivatives of $\langle x,y\rangle$
are bounded, so each Hessian block contributed by an edge has norm
at most $C(q)$. Each block row contains $Q-1$ such edges, giving
$\|\operatorname{Hess}H_{q,Q,\epsilon}\|_{\op}\le C(q)Q$.
Choose $c_{\ref{lem:signed-replica-mgf}}(q)$ so that
$|t|C(q)Q\le(d-2)/2$. Then
\[
 \log\E[e^{tH_{q,Q,\epsilon}}]
 \overset{\mathrm{(b)}}{\le}
 \frac{t^2\E[|\nabla H_{q,Q,\epsilon}|^2]}{d-2}
 \lesssim\frac{t^2Q^2}{N_{q,d}},
\]
where $\mathrm{(b)}$ uses~\eqref{eq:mc-1} in
Lemma~\ref{lem:smooth-t2-mgf}.
\end{proof}

\begin{lemma}[Exponential moments conditional on one sample]
\label{lem:signed-block}
Fix $q,\ell\ge1$ and condition on $X'$. For every real $t$ satisfying
$|t|\|\mathbf U_\ell(X')\|_{\op}<c_{\ref{lem:signed-block}}(q)d$,
\[
 \log\E_X[e^{tO_{q\ell}}]
 \le C_{\ref{lem:signed-block}}(q)
 \frac{t^2n\|\mathbf U_\ell(X')\|_{\op}}{N_{q,d}}.
\]
\end{lemma}

\begin{proof}
The Gram matrix $\mathbf U_\ell(X')$ is positive semidefinite with unit
diagonal, so its operator norm is at least one. Let $Q$ be the smallest
even integer not less than this norm. Then
$\|\mathbf U_\ell(X')\|_{\op}\le Q\le3\|\mathbf U_\ell(X')\|_{\op}$.
We may therefore apply~\eqref{eq:mc-4} in
Lemma~\ref{lem:signed-replica-mgf} at $|t|$ and $Q$.

Let $z_i=z_q(X_i)$ and $v_i=z_\ell(X_i')$. Since both vectors have norm one,
\[
 O_{q\ell}=\frac12\Big\|\sum_i z_iv_i^\top\Big\|_{\HS}^2-\frac n2.
\]
Let $(B_i)_{i\le n}$ be centered Gaussian vectors in $\mathbb R^{N_{q,d}}$
with $\operatorname{Cov}(B_i,B_j)=U_\ell(X_i',X_j')I$, and let
$\gamma\sim N(0,I_{N_{q,d}})$. Write $z$ for an independent copy of $z_q(X_1)$.
Expectations indexed by $z_1,\ldots,z_Q$ below use $Q$ independent copies of $z$.
For $t\ge0$, the Gaussian exponential identity gives
\[
 \E_X[e^{tO_{q\ell}}]
 =e^{-tn/2}\E_B\Big[\prod_i\phi_+(B_i)\Big],\quad
 \phi_+(b)=\E_z[e^{\sqrt t\langle b,z\rangle}].
\]
The covariance operator is at most $QI$.
Lemma~\ref{lem:gaussian-holder} and Gaussian integration yield
\begin{align}
 \E_X[e^{tO_{q\ell}}]
 &\overset{\mathrm{(a)}}{\le}
 e^{-tn/2}\big(\E_\gamma[\phi_+(\gamma)^Q]\big)^{n/Q}\notag\\
 &\overset{\mathrm{(b)}}{=}e^{-tn/2}
 \Big(\E_{z_1,\ldots,z_Q}
   \big[e^{(t/2)\|\sum_a z_a\|^2}\big]\Big)^{n/Q}=\Big(\E_{z_1,\ldots,z_Q}
   \big[e^{t\sum_{a<b}\langle z_a,z_b\rangle}\big]\Big)^{n/Q}.
 \label{eq:mc-5b}
\end{align}
Here $\mathrm{(a)}$ is Lemma~\ref{lem:gaussian-holder},
and $\mathrm{(b)}$ uses
$\E_\gamma[e^{\langle\gamma,w\rangle}]=e^{\|w\|^2/2}$.
The last equality uses $\sum_a\|z_a\|^2=Q$.
By~\eqref{eq:mc-4} in Lemma~\ref{lem:signed-replica-mgf}, the logarithm
is at most $Ct^2nQ/N_{q,d}$.

For the negative sign, the Gaussian characteristic function gives
\[
 \E_X[e^{-tO_{q\ell}}]
 =e^{tn/2}\E_B\Big[\prod_i\phi_-(B_i)\Big],\quad
 \phi_-(b)=\E_z[e^{i\sqrt t\langle b,z\rangle}].
\]
Taking absolute values gives
\[
 \E_X[e^{-tO_{q\ell}}]
 \le e^{tn/2}\E_B\Big[\prod_i|\phi_-(B_i)|\Big]
 \overset{\mathrm{(c)}}{\le}
 e^{tn/2}\big(\E_\gamma[|\phi_-(\gamma)|^Q]\big)^{n/Q},
\]
where $\mathrm{(c)}$ applies Lemma~\ref{lem:gaussian-holder}
to the nonnegative functions $|\phi_-|$.
Choose $Q/2$ signs equal to $1$ and the other $Q/2$ equal to $-1$.
Expanding $|\phi_-|^Q$ and integrating the Gaussian variable gives
\begin{align}
 e^{tQ/2}\E_\gamma[|\phi_-(\gamma)|^Q]
 &=\E_{z_1,\ldots,z_Q}
 \Big[\exp\Big(\frac{tQ}2-\frac t2
         \big\|\sum_a\epsilon_a z_a\big\|^2\Big)\Big]\notag\\
 &=\E_{z_1,\ldots,z_Q}
 \Big[\exp\Big(-t\sum_{a<b}\epsilon_a\epsilon_b
                    \langle z_a,z_b\rangle\Big)\Big].
 \label{eq:mc-6}
\end{align}
Again~\eqref{eq:mc-4} bounds the logarithm by $Ct^2Q^2/N_{q,d}$.
Raising to the power $n/Q$ and using
$Q\le3\|\mathbf U_\ell(X')\|_{\op}$ proves both signs.
\end{proof}

\section{Proofs for cusp kernels}\label{app:cusp}

Recall the posterior costs $g_{r,k}$ and $G_k$ from
\eqref{eq:4-2} and \eqref{eq:cavity-costs}.
We prove the common moment and prediction estimates, followed by
the bounds for order two, order three, and any fixed range of
higher orders. These estimates apply to the three continuous
cases of Definition~\ref{def:cusp}.
Section~\ref{sec:order-high-csup} combines them with the prior
bound and treats the jump case by~\eqref{gen:eq:universal-Gram-KL}.

\subsection{Moment and conditional expectation bounds}\label{app:cusp-local}

Recall Lemma~\ref{lem:localization}, which reduces the continuous
cusp proof to the localized model. The same comparison preserves
the spectral and moment estimates used here. Orthonormality and the
Hilbert--Schmidt bound give
\begin{align*}
 |\gamma_{q,d}-\widetilde\gamma_{q,d}|
 &\le\|\kappa_d-\widetilde\kappa_d\|_2,\quad |v_d-\widetilde v_d|
 +\|T_{\kappa_d}-T_{\widetilde\kappa_d}\|_{\mathrm{op}}
 \le C\|\kappa_d-\widetilde\kappa_d\|_2,\\
 \|H_d-\widetilde H_d\|_\infty
 &\le(\|\kappa_d\|_2+\|\widetilde\kappa_d\|_2)
       \|\kappa_d-\widetilde\kappa_d\|_2.
\end{align*}
For the last line, expand the two convolutions and use
Cauchy--Schwarz; rotational invariance makes every row have the
same $L^2$ norm. The errors are exponentially small by
\eqref{eq:cusp-localized-KL} in Lemma~\ref{lem:localization}, so the earlier estimates apply
to the localized model. As in Section~\ref{sec:cusp-converse},
we omit tildes throughout the continuous cusp proof and retain
the notation of the main text. Unless a conditioning or another law is
specified, expectations and $L^p$ norms below use independent uniform
roots; for a zonal profile $f$, $\|f\|_p=\|f(T_d)\|_p$.

Recall $v_d=\E[\kappa_d(T_d)^2]$ from
\eqref{eq:harmonic-expansion}, and let $e_d,o_d$ be the even and
odd parts defined in Appendix~\ref{app:cusp-spectrum}.
By~\eqref{eq:cusp-low-coefficients}, a present even component has
$|\gamma_{2,d}|\asymp d^{-\alpha/2}$ and a present odd component
has $|\gamma_{1,d}|\asymp d^{-\beta/2}$.
A missing component and its corresponding coefficient are exactly
zero after localization. The local power bounds and
$\delta_d=d^{-11/24}$ therefore give, for a fixed $B$,
\begin{align}
 \|\kappa_d\|_\infty^2
 &\lesssim v_d^{11/12},\notag\\
 |e_d(y/\sqrt d)|&\le C|\gamma_{2,d}|(1+|y|)^B,\quad
 |o_d(y/\sqrt d)|\le C|\gamma_{1,d}|(1+|y|)^B.
 \label{eq:cusp-localized-envelope}
\end{align}
Here Lemma~\ref{gen:lem:gap-transfer} gives
$v_d\asymp d^{-\alpha}$, $d^{-\beta}$, or
$d^{-\min\{\alpha,\beta\}}$ in the pure even, pure odd, or mixed case,
respectively. Centering changes the supremum bound by at most a factor of two.
The even part is centered under the exact spherical marginal.
Every fixed amplification $p_d+s_db_*\kappa_d$ therefore stays
in $[c/2,1-c/2]$ for all sufficiently large $d$.

We will use the following consequences of
Lemma~\ref{gen:lem:gap-transfer} and
Proposition~\ref{gen:prop:spectral-scales}:
\begin{equation}
 v_d\asymp\gamma_{2,d}^2+\gamma_{1,d}^2,\quad
 \tau_{4,d}\asymp\gamma_{2,d}^4d^{-2}+\gamma_{1,d}^4d^{-1},\quad
 v_d/d\lesssim\sqrt{\tau_{4,d}}.
 \label{eq:cusp-app-scales}
\end{equation}
Outside the critical mixed case, recall the exponent $s$ from
\eqref{eq:cusp-dimension}. Then
\begin{equation}
 \tau_{3,d}^2\asymp d^{-3/s},\quad
 \sqrt{\tau_{4,d}}\lesssim d^{-1/s-1/6},\quad
 d^{1/3-1/s}\lesssim v_d\lesssim d^{2/3-1/s},\quad 1/s\le5/3.
 \label{eq:cusp-app-ordinary-scales}
\end{equation}
For a critical mixed link, the upper bounds in
\eqref{eq:cusp-app-ordinary-scales} still hold with
$1/s=\alpha+2/3$, except that
$\tau_{3,d}^2\lesssim d^{-3/s}$ may be a strict upper bound.
Thus the estimates at the triangle scale also apply to this case.
On $\beta=\alpha+1/3$, the harmonic sum~\eqref{eq:trace-formula}
instead gives
\begin{align}
 v_d&\asymp d^{-\alpha},\quad
 \sqrt{\tau_{4,d}}\asymp d^{-\alpha-5/6},\notag\\
 \|T_{\kappa_d}\|_{\mathrm{op}}^2&\asymp d^{-\alpha-4/3},\quad
 |\tau_{5,d}|\lesssim d^{-5\alpha/2-7/3},\quad
 \tau_{8,d}\asymp d^{-4\alpha-13/3}.
 \label{eq:cusp-app-critical-scales}
\end{align}
For the last two bounds, separate degrees one and two in
\eqref{eq:trace-formula}. Parseval and $N_{q,d}\ge cd^3$ for $q\ge3$
give, for $\ell=5,8$,
\[
 \sum_{q\ge3}|\gamma_{q,d}|^\ell N_{q,d}^{1-\ell/2}
 \le Cd^{-3(\ell/2-1)}v_d^{\ell/2}.
\]
Together with $|\gamma_{1,d}|\asymp d^{-\alpha/2-1/6}$ and
$|\gamma_{2,d}|\asymp d^{-\alpha/2}$, this proves the asserted scales.
We retain the actual signed $\tau_{3,d}$ on this relation.
Recall also from Lemma~\ref{lem:contracted} that
$\|H_d(T_d)\|_p\le Cp\sqrt{\tau_{4,d}}$ for $p\ge2$ and
$\|H_d(T_d)\|_4\lesssim\sqrt{\tau_{4,d}}$.

Recall the interaction kernel $h_{m,d}$ from~\eqref{eq:star-kernel}.
We bound its moments by expanding the joint density of the
projections onto its roots. Fix an integer $m\ge2$ and $1\le p<\infty$.
For four deterministic roots, let
$S=\sum_{i<j}|\langle x_i,x_j\rangle|$.

\begin{lemma}[Root moments]\label{lem:root-moments}
Under~\eqref{eq:cusp-localized-envelope}, for sufficiently large $d$,
\begin{equation}
 \|h_{m,d}\|_p\le C_{\ref{lem:root-moments}}(m,p)
 \sum_{\substack{0\le j\le m\\j\ \mathrm{even}}}
 |\gamma_{2,d}|^{m-j}|\gamma_{1,d}|^j d^{-m/2+j/4}.
 \label{eq:cusp-app-root-fixed}
\end{equation}
For any four deterministic roots,
\begin{equation}
 |h_{4,d}|\le C\big[\gamma_{1,d}^4S^2+\gamma_{2,d}^2\gamma_{1,d}^2(S^3+S/d)
                    +\gamma_{2,d}^4(S^4+S^2/d+d^{-2})\big].
 \label{eq:cusp-app-root-four}
\end{equation}
\end{lemma}

\begin{proof}
Recall the density $f_d$ of $\sqrt dT_d$ from Lemma~\ref{gen:lem:gap-transfer}.
For nonzero components, define
$g_e(y)=e_d(y/\sqrt d)/|\gamma_{2,d}|$ and
$g_o(y)=o_d(y/\sqrt d)/|\gamma_{1,d}|$.
Set the corresponding profile to zero when a component is missing. By~\eqref{eq:cusp-localized-envelope},
both normalized profiles have uniformly bounded moments of each fixed order.
Centering and parity give
\[
 \int y^r g_e(y)f_d(y)dy=0\quad(r=0\text{ or }r\text{ odd}),\quad
 \int y^r g_o(y)f_d(y)dy=0\quad(r\text{ even}).
\]
For the Gram matrix $R=I+F$ of $m$ roots, the joint density of their scaled projections onto $Z$ is
\begin{equation}
 f_{d,R}(y)=\frac{\Gamma(d/2)}{\Gamma((d-m)/2)(\pi d)^{m/2}\sqrt{\det R}}
 \Big(1-\frac{y^\top R^{-1}y}{d}\Big)_+^{(d-m-2)/2}.
 \label{eq:cusp-app-projection-density}
\end{equation}
For fixed $L,J$ and $\|F\|_{\mathrm{op}}\le1/2$, expand this density against the exact product marginal. For a multi-index $\eta=(\eta_{ij})_{i<j}$, write $|\eta|=\sum_{i<j}\eta_{ij}$ and $F^\eta=\prod_{i<j}F_{ij}^{\eta_{ij}}$. There are polynomials $A_{\eta,h}$ such that
\begin{equation}
 f_{d,I+F}(y)=\prod_i f_d(y_i)
 \sum_{\ell=0}^L\sum_{h=0}^Jd^{-h}\sum_{|\eta|=\ell}F^\eta A_{\eta,h}(y)
 +\mathcal R(y,F),\quad \deg A_{\eta,h}\le2\ell+4h.
 \label{eq:cusp-app-density-expansion}
\end{equation}
For every fixed $b\ge0$, the remainder satisfies
\[
 \int(1+|y|)^b|\mathcal R(y,F)|dy
 \le C\big(\|F\|_{\HS}^{L+1}+d^{-J-1}\big).
\]
Indeed, at Gaussian order an $F$ derivative adds degree at most two,
and a term of order $d^{-h}$ adds degree at most $4h$.
Taylor's formula applies on
$\max_i|y_i|\le d^\epsilon$, with fixed sufficiently small
$\epsilon>0$. The weighted integral outside this set is
$O(\exp(-cd^{2\epsilon}))$ by the spherical tail bound.

A product with an odd number of odd profiles integrates to zero.
For $j$ even, a surviving monomial must have degree at least two
in each even coordinate and at least one in each odd coordinate. Hence
\[
 2\ell+4h\ge2(m-j)+j,\quad
 d^{-h}\|F^\eta\|_p\overset{\mathrm{(a)}}{\le}C d^{-h-\ell/2}
                         \overset{\mathrm{(b)}}{\le}C d^{-m/2+j/4}.
\]
Here $\mathrm{(a)}$ uses H\"older and the fixed spherical moments of
the entries of $F$; $\mathrm{(b)}$ uses the degree condition.
Choose $L,J\ge2m$. Then the integrated remainder is smaller than
the displayed bound. Also,
$\P[\|F\|_{\mathrm{op}}>1/2]\le Ce^{-cd}$ for fixed $m$;
the deterministic H\"older bound and~\eqref{eq:cusp-localized-envelope}
control this event. Summing over the even values of $j$ proves
\eqref{eq:cusp-app-root-fixed}.

For~\eqref{eq:cusp-app-root-four}, expand first in $F$, retaining exact parity at each finite $d$. Let $M_j(F)$ be the normalized integral with $j$ odd profiles, and let $D$ denote differentiation with respect to the matrix argument. The derivatives at $F=0$ satisfy
\[
\begin{array}{c|c|c}
j&\text{constant and lower derivatives}&\text{remainder}\\\hline
4&M_4(0)=DM_4(0)=0&O(S^2)\\
2&M_2(0)=0,\quad DM_2(0),D^2M_2(0)=O(d^{-1})&O(S^3)\\
0&M_0(0)=O(d^{-2}),\quad DM_0(0)=0,&O(S^4)\\
 &D^2M_0(0),D^3M_0(0)=O(d^{-1})&
\end{array}
\]
Coordinate sign changes give the zeros in the table. The degree
condition gives the stated powers of $d^{-1}$ for the other entries.
The next derivative has a uniformly bounded weighted $L^1$ norm
on $\|F\|_{\mathrm{op}}\le1/2$, by the density expansion. Taylor's
formula therefore gives
\[
 |M_4(F)|\le CS^2,\quad
 |M_2(F)|\le C(S^3+S/d),\quad
 |M_0(F)|\le C(S^4+S^2/d+d^{-2}).
\]
Here $S^2/d\le C S/d$ and $S^3/d\le C S^2/d$ on this set.
If $\|F\|_{\mathrm{op}}>1/2$, then $S\ge c$, and H\"older gives
$|M_j(F)|\le C$. Restoring the profile coefficients proves
\eqref{eq:cusp-app-root-four}.
\end{proof}

Two consequences of Lemma~\ref{lem:root-moments} are used repeatedly. With a constant independent of $p\ge1$,
\begin{equation}
 \|h_{3,d}\|_4^2\le C(\gamma_{2,d}^6d^{-3}+\gamma_{2,d}^2\gamma_{1,d}^4d^{-2}),\quad
 \|h_{4,d}\|_p\le C\big(\gamma_{1,d}^4p/d+\gamma_{2,d}^2\gamma_{1,d}^2(p/d)^{3/2}+\gamma_{2,d}^4p^2/d^2\big).
 \label{eq:cusp-app-root-useful}
\end{equation}
The second bound follows from~\eqref{eq:cusp-app-root-four} and
$\|S\|_p\le C\sqrt{p/d}$, using $p\ge1$ to absorb the terms containing $d^{-1}$.
In particular, $h_{3,d}=0$ for a localized pure odd link.
Moreover, four independent copies of the integrated point give
$\|h_{m,d}\|_4^2=(\E[h_{4,d}^m])^{1/2}\le\|h_{4,d}\|_m^{m/2}$.
Substituting~\eqref{eq:cusp-app-root-useful} and using
\eqref{eq:cusp-app-scales} yields
\begin{equation}
 \|h_{m,d}\|_4^2\le(Cm^2)^{m/2}\tau_{4,d}^{m/2};
 \quad\text{on }\beta=\alpha+1/3,\quad
 \|h_{m,d}\|_4^2\le(Cm)^{m/2}\tau_{4,d}^{m/2}\quad(m\le c d^{1/3}).
 \label{eq:cusp-app-root-growing}
\end{equation}
In the critical bound, use $\gamma_{1,d}^2\asymp\gamma_{2,d}^2d^{-1/3}$ in~\eqref{eq:cusp-app-root-useful}; the two smaller terms are at most $Cp\tau_{4,d}$ when $p\le c d^{1/3}$.

For likelihood denominators we also need a bound uniform in the root configuration. The bound applies even when the roots are close to one another.

\begin{lemma}[A uniform envelope]\label{lem:growing-roots}
There are fixed constants $C,B,c>0$ such that, for even
$4\le m\le c\sqrt d$ and arbitrary roots,
\begin{equation}
 |h_{m,d}(x_1,\ldots,x_m)|\le(Cm)^{Bm}v_d^{m/2}
 \Big(\sum_{i<j}\langle x_i,x_j\rangle^2+d^{-1}\Big).
 \label{eq:cusp-app-uniform-root}
\end{equation}
\end{lemma}

\begin{proof}
Let $R=I+F$ be the Gram matrix of the roots, so $F_{ii}=0$.
For positive definite $R$, write $h_{m,d}(R)$ for the integral
against the density $f_{d,R}$ in~\eqref{eq:cusp-app-projection-density}.
Let $\varphi_m$ be the standard Gaussian density and
$L_{d,m}=f_{d,I}/\varphi_m$. Extend each profile by zero outside
$[-\sqrt d,\sqrt d]$. Direct expansion gives
\begin{align*}
 L_{d,m}(y)
 &=c_{d,m}\exp(|y|^2/2)(1-|y|^2/d)_+^{(d-m-2)/2},\quad
 |\log c_{d,m}|\le Cm^2/d,\\
 |L_{d,m}(y)-1|
 &\le C(m^2+|y|^4)/d,\\
 \Big|\frac{d-m-2}{d-|y|^2}L_{d,m}(y)-1\Big|
 &\le C(m^2+|y|^4)/d.
\end{align*}
In the last line the product is extended by zero outside $|y|^2<d$.
For $|y|^2\le d/2$, expand the logarithms. For $|y|^2>d/2$,
both multipliers are bounded by $C$, as follows by maximizing their
radial factors.

Exact spherical centering and the case $m=1$ give
\[
 \Big|\int\kappa_d(y/\sqrt d)\varphi_1(y)dy\Big|
 \le C\sqrt{v_d}/d,\quad
 \int|y\kappa_d(y/\sqrt d)|\varphi_1(y)dy\le C\sqrt{v_d}.
\]
For the Gaussian derivative, the summand indexed by $i<j$ is
\[
 F_{ij}\prod_{a\in\{i,j\}}\int y_a\kappa_d(y_a/\sqrt d)\varphi_1(y_a)dy_a
       \prod_{a\notin\{i,j\}}\int\kappa_d(y_a/\sqrt d)\varphi_1(y_a)dy_a.
\]
Since $m-2\ge2$, its absolute value is at most
$C^m|F_{ij}|v_d^{m/2}d^{-(m-2)}$.
The two multiplier bounds control the spherical correction by
$(Cm)^{Bm}v_d^{m/2}/d$. Thus
\begin{equation}
 |h_{m,d}(I)|+|Dh_{m,d}(I)[F]|
 \le(Cm)^{Bm}v_d^{m/2}/d\quad(\|F\|_{\mathrm{op}}\le1/2).
 \label{eq:cusp-app-uniform-linear}
\end{equation}
For the second derivative, let $R_t=I+tF$,
$q_t=y^\top R_t^{-1}y$, and $\ell_t=\log f_{d,R_t}$. Then
\begin{align*}
 q_t'&=-y^\top R_t^{-1}FR_t^{-1}y,\quad
 q_t''=2y^\top R_t^{-1}FR_t^{-1}FR_t^{-1}y,\\
 \ell_t'&=-\tfrac12\operatorname{Tr}(R_t^{-1}F)
       -\frac{d-m-2}{2}\frac{q_t'}{d-q_t},\\
 \ell_t''&=\tfrac12\operatorname{Tr}((R_t^{-1}F)^2)
       -\frac{d-m-2}{2}
       \Big(\frac{q_t''}{d-q_t}+\frac{(q_t')^2}{(d-q_t)^2}\Big).
\end{align*}
Under $f_{d,R_t}$, $q_t/d\sim\mathrm{Beta}(m/2,(d-m)/2)$.
Since $|q_t'|\le C\|F\|_{\HS}q_t$ and
$|q_t''|\le C\|F\|_{\HS}^2q_t$, the beta moments give
$\|\ell_t'\|_4\le Cm\|F\|_{\HS}$ and
$\|\ell_t''\|_2\le Cm^2\|F\|_{\HS}^2$.
Each coordinate has marginal density $f_d$, because $(R_t)_{ii}=1$.
Hence
\[
 |D^2h_{m,d}(R_t)[F,F]|
 \overset{\mathrm{(a)}}{\le}\Big\|\prod_i\kappa_d(y_i/\sqrt d)\Big\|_2
       \| (\ell_t')^2+\ell_t''\|_2
 \overset{\mathrm{(b)}}{\le}(Cm)^{Bm}v_d^{m/2}\|F\|_{\HS}^2.
\]
Here $\mathrm{(a)}$ is Cauchy--Schwarz; $\mathrm{(b)}$ uses the derivative
bounds and H\"older with the polynomial envelope in
\eqref{eq:cusp-localized-envelope}. The density vanishes to order
greater than two at its boundary, so differentiation under the integral
is valid. Taylor's integral formula and~\eqref{eq:cusp-app-uniform-linear}
now prove~\eqref{eq:cusp-app-uniform-root} when $\|F\|_{\mathrm{op}}\le1/2$.
Otherwise, $\sum_{i<j}R_{ij}^2\ge1/8$, and H\"older gives the same bound.
\end{proof}
\noindent\textbf{Posterior prediction.}\phantomsection\label{app:cusp-queries}
We next bound the squared conditional means in~\eqref{eq:4-2}.
We need the full graph $A^{(r)}$ and, for two roots, the same graph
with edge $12$ withheld. Write $\mathcal A_{r,K_k}=A^{(r)}$ and
$\mathcal A_{r,\varnothing}=A^{(r)}\setminus\{12\}$ for these
observations. Here $K_k$ is the complete graph on the queried roots $[k]$.
We take $E=K_k$, or $E=\varnothing$ when $k=2$.
Fix $C_0>0$, let $2\le k\le\min(r,C_0\log d)$ with $r\le n$,
and let $f=f(X_1,\ldots,X_k)$ be bounded and canonical in each root.
For an edge $e=\{i,j\}$, write $T_e=T_{ij}$.
For a fixed factor $b_*>1$, define
\[
 A(t)=\exp(nb_*^2H_d(t)),\quad
 f_F=f\prod_{e\in F}b_*\kappa_d(T_e),\quad F\subseteq K_k.
\]
Here $T_A$ is the integral operator with profile $A$,
$T_A^{\otimes k}$ applies it separately to the $k$ roots, and the
inner product is in $L^2(\sigma_d^{\otimes k})$.
The amplification by $b_*$ is used only to bound a likelihood
ratio. It is valid for large $d$ by~\eqref{eq:cusp-localized-envelope}.

\begin{lemma}[Conditional expectations]
\label{lem:cusp-query}
There are constants $b_*>1$ and $C,c,\delta>0$, depending only on
$K,C_0$, such that, if $n\sqrt{\tau_{4,d}}\le c$ and $d$ is sufficiently large,
\begin{align}
 \E_P\big[\E_P[f\mid\mathcal A_{r,E}]^2\big]
 &\le C^k\Big(k^2n\tau_{4,d}\|f\|_4^2+
       \sum_{F\subseteq E}\langle f_F,T_A^{\otimes k}f_F\rangle\Big)
       +\|f\|_\infty^2\exp(-cd^\delta),
 \label{eq:cusp-query-tensor}\\
 \E_P\big[\E_P[f\mid\mathcal A_{r,E}]^2\big]
 &\le C^k\|f\|_4^2
       \Big(k^2n\tau_{4,d}+(Ck)^k(v_d^{11/24}+Cn\sqrt{\tau_{4,d}/d})^k\Big)+\|f\|_\infty^2\exp(-cd^\delta).
 \label{eq:cusp-query-rough}
\end{align}
The operator norm of $T_A$ is at most $C$ on constants,
$Cn\sqrt{\tau_{4,d}/d}$ on centered functions, and
$Cn\sqrt{\tau_{4,d}}/d$ on centered even functions.
\end{lemma}

\begin{proof}
We first remove likelihood denominators. For this comparison only,
let $Q$ be a product of Bernoulli laws whose parameters are bounded away from zero and one,
and let $\mathcal T_\theta$ independently retain each coordinate
with probability $\theta$ and otherwise resample it from its marginal.
For sufficiently small fixed $\theta$,
$\|\mathcal T_\theta N\|_4\le\|N\|_2$. To see the bound on one
coordinate, write $N=a+bY$, where $\E Y=0$, $\E Y^2=1$, and
$|Y|\le M$. Expanding the fourth power and using
$2|a||b|^3\le a^2b^2+b^4$ gives
\[
 \E[(a+\theta bY)^4]
 \le a^4+(6\theta^2+2M\theta^3)a^2b^2
          +(2M\theta^3+M^2\theta^4)b^4
 \le(a^2+b^2)^2.
\]
Applying this bound one coordinate at a time and using Minkowski's
inequality proves the product bound.
For a positive $L$ with $\E_Q L=1$,
\cite[Corollary~1.11]{mossel2011reverse}, applied with exponents
$1/2$ and $-2$, gives
$\|\mathcal T_\theta L\|_{-2}\ge\|L\|_{1/2}$ when
$\theta\le1/6$. Here $\|U\|_q=(\E_Q[U^q])^{1/q}$ for positive
$U$ and $q\ne0$. Moreover, H\"older's inequality gives
\[
 1\le(\E_Q L^{1/2})^{2/3}(\E_Q L^2)^{1/3},\quad
 \|L\|_{1/2}\ge(\E_Q L^2)^{-1}.
\]
Thus $\E_Q[(\mathcal T_\theta L)^{-2}]\le(\E_Q L^2)^2$.
Choose $b_*=1/\theta$ satisfying both inequalities.
For the observed coordinates, let $L$ be the likelihood relative to $Q$
and let $N$ be the numerator of the posterior mean of $f$, so that
this mean is $N/L$. The product expansion
shows that the original likelihood and numerator are
$L=\mathcal T_\theta\widehat L$ and
$N=\mathcal T_\theta\widehat N$, respectively, when hats denote
amplification by $b_*$. Consequently
\begin{equation}
 \E_Q[N^2/L]
 \overset{\mathrm{(a)}}{\le}\big(\E_Q[N^4]\big)^{1/2}\big(\E_Q[L^{-2}]\big)^{1/2}
 \overset{\mathrm{(b)}}{\le}\E_Q[\widehat N^2]\E_Q[\widehat L^2].
 \label{eq:cusp-app-uq-denominator}
\end{equation}
Here $\mathrm{(a)}$ is Cauchy--Schwarz and $\mathrm{(b)}$ uses the two
bounds for $\mathcal T_\theta$. The same calculation applies conditionally on an observed subgraph,
since the conditional likelihood still has mean one under $Q$.

We now take $Q=\Bern(p_d)^{\otimes s}$ on $B\in\{0,1\}^s$ and
write $Y_i(B)=(B_i-p_d)/s_d$.
Let $Z_1,\ldots,Z_s$ and $X$ be independent with law $\sigma_d$,
where $s\le n$, and define
\[
 L_Z(B)=\E_X\Big[\prod_{i=1}^s
       (1+b_*\kappa_d(\langle X,Z_i\rangle)Y_i(B))\Big].
\]
For integer $p\ge2$, let $(X_i,X_i')_{i\le p}$ be independent
pairs with law $\sigma_d\otimes\sigma_d$, and define
\[
 J_p=\sum_{\substack{I\subseteq[p]\\|I|\ge2}}
       b_*^{2|I|}h_{2|I|,d}((X_i,X_i')_{i\in I}),\quad
 S_p=\sum_{i<j}\langle U_i,U_j\rangle^2,
\]
where $U_1,\ldots,U_{2p}$ list these roots. The term $J_p$ contains
those factors involving at least two root pairs. The replica identity gives
\[
 \E_Z\big[(\E_Q[L_Z^2])^p\big]=\E\Big[
       \big(1+\sum_{i=1}^pb_*^2H_d(\langle X_i,X_i'\rangle)+J_p\big)^s\Big].
\]
For a fixed sufficiently large integer $D_0$,
\eqref{eq:cusp-app-uniform-root} in Lemma~\ref{lem:growing-roots} gives
\[
 |J_p|\le Cp^{D_0}v_d^2(S_p+d^{-1})
 \quad\text{if }p^{D_0+1}v_d=o(1)\text{ and }2p\le c\sqrt d.
\]
Indeed, the coefficient of $S_p+d^{-1}$ is at most
$\sum_{j=2}^p(Cp^{2B+1}v_d)^j\le Cp^{4B+2}v_d^2$.
Increasing $D_0$ absorbs the fixed factor $b_*$.
For $0\le t\le d/4$, the spherical even moments give
\[
 \E\exp(tT_d^2)
 \le\sum_{j\ge0}\frac{t^j(2j-1)!!}{j!d^j}
 =(1-2t/d)^{-1/2}\le\exp(2t/d).
\]
H\"older over the $\binom{2p}{2}$ overlaps gives
\[
 \E\exp(Cnp^{D_0}v_d^2S_p)
 \le\exp(Cnp^{D_0+2}v_d^2/d),
 \quad Cnp^{D_0+2}v_d^2\le d/4.
\]
Also, $\|H_d\|_q\le Cq\sqrt{\tau_{4,d}}$ implies
$\E\exp(Cn|H_d|)\le\exp(Cn\sqrt{\tau_{4,d}})$ when
$n\sqrt{\tau_{4,d}}$ is small. Since the $p$ designated pairs are
independent, Cauchy--Schwarz gives
\begin{equation}
 \E_Z\big[(\E_Q[L_Z^2])^p\big]\le
 \exp\big(Cpn\sqrt{\tau_{4,d}}+Cnp^{D_0+2}v_d^2/d\big).
 \label{eq:cusp-app-uq-design-moments}
\end{equation}
Choose a fixed $\delta>0$ sufficiently small that
$2d^\delta\le c\sqrt d$ and
$d^{\delta(D_0+2)}v_d=o(1)$. Since $nv_d/d\le Cn\sqrt{\tau_{4,d}}$, the second
term in the exponent is $o(1)$, uniformly for
$2\le p\le2d^\delta$. Taking $p=\lfloor d^\delta\rfloor$ and
then a sufficiently large fixed $M$ proves
$\P_Z[\E_Q[L_Z^2]>M]\le\exp(-cd^\delta)$.

Condition also on the $s=r-k$ other latent points. Conditional Jensen
shows that this increases the squared conditional mean.
Given these points, edges between them are independent of the queried roots.
The remaining likelihood includes the incident edges
and the edges in $E$. Each internal edge contributes
at most $1+Cv_d^{11/12}$ to its amplified second moment, and the
$k$ incident channels factor. Thus the multiplier in
\eqref{eq:cusp-app-uq-denominator} is at most
\[
 (1+Cv_d^{11/12})^{|E|}\big(\E_Q[L_Z^2]\big)^k
 \overset{\mathrm{(a)}}{\le}\exp(Ck^2v_d^{11/12})M^k
 \overset{\mathrm{(b)}}{\le}C^k
 \quad\text{on }\{\E_Q[L_Z^2]\le M\}.
\]
Here $\mathrm{(a)}$ uses $1+x\le e^x$ and the displayed event;
$\mathrm{(b)}$ uses $k=O(\log d)$ and the polynomial decay of $v_d$.
On the complement, bound the actual squared conditional mean by
$\|f\|_\infty^2$.

Let $X=(X_1,\ldots,X_k)$ and $X'=(X_1',\ldots,X_k')$ be the
independent prior root vectors. After averaging the numerator second
moment over the design, its exact value is
\[
 \E\Big[f(X)f(X')
  \prod_{ij\in E}(1+b_*^2\kappa_d(T_{ij})\kappa_d(T_{ij}'))
  \big(1+b_*^2\sum_{i=1}^kH_d(\langle X_i,X_i'\rangle)+J_k\big)^s\Big].
\]
The first bound in~\eqref{eq:cusp-app-root-growing} gives
\[
 \|J_k\|_4\le\sum_{j=2}^k\binom kj(Cj)^j\tau_{4,d}^{j/2}
 \le Ck^2\tau_{4,d}.
\]
The calculation proving \eqref{eq:cusp-app-uq-design-moments}, with
$p=k$ and fixed exponential powers, also gives
\[
 \E\exp\Big(Cn\big(\sum_i|b_*^2H_d(\langle X_i,X_i'\rangle)|+|J_k|\big)\Big)
 \le C^k.
\]
The internal product has absolute value at most
$\exp(Ck^2v_d^{11/12})$. Using the difference of powers and then
H\"older with exponents $2,4,4$, the change caused by setting
$J_k=0$ is at most
\[
 C^kn\|f(X)f(X')\|_2\|J_k\|_4
 \le C^kk^2n\tau_{4,d}\|f\|_4^2.
\]
The exponential bound holds for every fixed power after decreasing
the allowed constant in $n\sqrt{\tau_{4,d}}\le c$.

Expand the internal factors over $F\subseteq E$ and expand the remaining
power multinomially. For each $\boldsymbol a=(a_1,\ldots,a_k)$,
the integrated coefficient is
\[
 Q_{F,\boldsymbol a}
 =\big\langle f_F,\bigotimes_{i=1}^k
                T_{b_*^{2a_i}H_d^{a_i}}f_F\big\rangle\ge0.
\]
The kernel $b_*^2H_d$ is positive semidefinite. Its pointwise powers are also positive semidefinite, since
$\langle u,v\rangle^j=\langle u^{\otimes j},v^{\otimes j}\rangle$.
The convention $H_d^0=1$ gives the constant projection. Thus
replacing $(s)_{|\boldsymbol a|}$ by $n^{|\boldsymbol a|}$
increases the sum after integration, and
\[
 \sum_{|\boldsymbol a|\le s}
   \frac{(s)_{|\boldsymbol a|}}{\prod_i a_i!}Q_{F,\boldsymbol a}
 \overset{\mathrm{(a)}}{\le}\sum_{\boldsymbol a\ge0}
   \prod_i\frac{n^{a_i}}{a_i!}Q_{F,\boldsymbol a}
 =\langle f_F,T_A^{\otimes k}f_F\rangle.
\]
Here $\mathrm{(a)}$ uses $Q_{F,\boldsymbol a}\ge0$ and
$(s)_j\le n^j$. The equality uses the series for $A$.
This proves~\eqref{eq:cusp-query-tensor}, including signed $f_F$.

The constant eigenvalue of $T_A$ is bounded by
$\E\exp(n|b_*^2H_d|)\le C$. Let $\lambda_q(T_A)$ denote its eigenvalue
on spherical harmonics of degree $q$.
For $q\ge1$, the multiplicity $N_{q,d}$ gives
\[
 |\lambda_q(T_A)|
 \overset{\mathrm{(a)}}{\le}N_{q,d}^{-1/2}
       \sum_{j\ge1}\frac{n^j}{j!}b_*^{2j}\|H_d^j\|_2
 \overset{\mathrm{(b)}}{\le}Cn\sqrt{\tau_{4,d}} N_{q,d}^{-1/2}.
\]
Here $\mathrm{(a)}$ is the Hilbert--Schmidt bound on each harmonic
eigenspace, and $\mathrm{(b)}$ follows by summing the series, since
$b_*^{2j}\|H_d^j\|_2\le(Cj\sqrt{\tau_{4,d}})^j$. Since $N_{q,d}\ge d$ for
$q\ge1$ and $N_{q,d}\ge cd^2$ for even $q\ge2$, this proves
the asserted operator bounds.

If $F$ has $e$ edges, at least $(k-2e)_+$ coordinates of $f_F$
remain canonical, and $\|f_F\|_2^2\le (Cv_d^{11/12})^e\|f\|_2^2$. Hence
\begin{align*}
 \sum_{F\subseteq E}\langle f_F,T_A^{\otimes k}f_F\rangle
 &\overset{\mathrm{(a)}}{\le}C^k\|f\|_2^2
    \sum_{e=0}^{\binom k2}\binom{\binom k2}{e}
       (Cv_d^{11/12})^e\big(Cn\sqrt{\tau_{4,d}/d}\big)^{(k-2e)_+}\\
 &\overset{\mathrm{(b)}}{\le}C^k(Ck)^k\big(v_d^{11/24}+n\sqrt{\tau_{4,d}/d}\big)^k\|f\|_2^2.
\end{align*}
Here $\mathrm{(a)}$ applies the operator bounds to the canonical coordinates.
For $\mathrm{(b)}$ and $2e\le k$, use $\binom{\binom k2}{e}\le k^{2e}$ and
compare each term with $(v_d^{11/24}+Cn\sqrt{\tau_{4,d}/d})^k$. For $2e>k$, the
binomial tail is at most
$(Ckv_d^{11/12})^{k/2}\exp(Ck^2v_d^{11/12})$; here $k^2v_d^{11/12}=o(1)$.
This proves \eqref{eq:cusp-query-rough}.

\end{proof}

\subsection{Order two}\label{app:cusp-order-two}\label{app:cusp-sums}
Under the assumptions of Lemma~\ref{lem:cusp-query}, let $f$ be a
bounded canonical zonal profile of two roots. Its tensor bound gives
\begin{align}
 \E_P\big[\E_P[f\mid A^{(r)}]^2\big]
 &\le C\big(\E[f(T_d)\kappa_d(T_d)]^2+n\tau_{4,d}\|f\|_4^2\big)
       +\|f\|_\infty^2\exp(-cd^\delta)\notag\\
 &\quad{}+\frac{Cn^2\tau_{4,d}}{d}\big(\|f\|_2^2+
       \|f\kappa_d-\E[f\kappa_d]\|_2^2\big),
 \label{eq:cusp-query-two}\\
 \E_P\big[\E_P[f\mid\mathcal A_{r,\varnothing}]^2\big]
 &\le C\Big(n\tau_{4,d}\|f\|_4^2+
             \frac{n^2\tau_{4,d}}{d}\|f\|_2^2\Big)
       +\|f\|_\infty^2\exp(-cd^\delta).
 \label{eq:cusp-query-two-minus}
\end{align}

To verify these bounds, note that rotational invariance gives
$\E[f\kappa_d\mid X_i]=\E[f\kappa_d]$ for $i=1,2$.
Thus $f$ and $f\kappa_d-\E[f\kappa_d]$ are canonical in both roots.
The operator bounds in Lemma~\ref{lem:cusp-query} give
\begin{align*}
 \langle f,T_A^{\otimes2}f\rangle
 &\le Cn^2\tau_{4,d}\|f\|_2^2/d,\\
 \langle f\kappa_d,T_A^{\otimes2}(f\kappa_d)\rangle
 &\le C\E[f\kappa_d]^2
       +Cn^2\tau_{4,d}\|f\kappa_d-\E[f\kappa_d]\|_2^2/d.
\end{align*}
Only the first term occurs when edge $12$ is withheld.
When it is retained, both occur, with a fixed factor $b_*^2$ on
the second. Substitution in~\eqref{eq:cusp-query-tensor} in Lemma~\ref{lem:cusp-query} proves
\eqref{eq:cusp-query-two} and~\eqref{eq:cusp-query-two-minus}.

Proposition~\ref{prop:cusp-triangle-two} applies
\eqref{eq:cusp-query-two} with $f=H_d$ at the triangle scale.
On the critical relation $\beta=\alpha+1/3$, the last term in
\eqref{eq:cusp-triangle-two} is
$n^5\tau_{4,d}^2/d\asymp n^5\tau_{4,d}^{5/2}d^{\alpha-1/6}$.
We improve this term by conditioning on the graph among the other
vertices, then applying the bound for two roots to $H_{4,d}$.
Assume $\Lambda_{n,d}=n^3\tau_{3,d}^2+n^4\tau_{4,d}^2\le c_{\ref{prop:cusp-KL}}(K)$ throughout this
part; then $n\lesssim d^{\alpha+5/6}$ and
$nv_d/d\lesssim n\sqrt{\tau_{4,d}} d^{-1/6}$ by~\eqref{eq:cusp-app-critical-scales}.

Recall that $H_{4,d}$ is the profile of $T_{\kappa_d}^4$ and
$g_{r,2}=\E_P[\E_P[H_d(T_{12})\mid A^{(r)}]^2]$.
We will also use
\begin{equation}
 \E\big[\exp\big(Cn(|H_d(T_{13})|+|H_d(T_{24})|+|h_{4,d}|)\big)\big]\le C
 \label{eq:cusp-app-critical-root-weight}
\end{equation}
for four independent roots. To verify it, use
$\|H_d\|_p\le Cp\sqrt{\tau_{4,d}}$ and
$|h_{4,d}|\le Cv_d^2(\sum_{i<j}T_{ij}^2+d^{-1})$ from
Lemma~\ref{lem:growing-roots}.
H\"older separates the two $H_d$ factors and the six squared
inner products. The exponential estimates used in
\eqref{eq:cusp-app-uq-design-moments} then apply because
$nv_d^2/d\lesssim n\sqrt{\tau_{4,d}}$.
For each fixed coefficient in the exponential, decrease the allowed
bound on $n\sqrt{\tau_{4,d}}$ if needed.

For the next estimate, keep the graph on the other $m$ vertices
while integrating their latent points. Let $0\le m\le n-2$, let
$Z_1,\ldots,Z_m$ be independent prior directions, and let $C$ be
their observed graph.
Let $X,Y$ be two additional independent prior roots. Their incident
edge vectors are $B,B'$, and their mutual edge is withheld.
Here $Q$ is the matched product Bernoulli law on these $2m$ edges,
and $Y_i(B)=(B_i-p_d)/s_d$. Define
\begin{align*}
 \ell_{x,Z}^{b_*}(B)&=\prod_{i=1}^m
    (1+b_*\kappa_d(\langle x,Z_i\rangle)Y_i(B)),\quad L_Z^{b_*}(B)=\E_X[\ell_{X,Z}^{b_*}(B)],\\
 N_C^{b_*}(B,B')&=\E[H_d(\langle X,Y\rangle)
       \ell_{X,Z}^{b_*}(B)\ell_{Y,Z}^{b_*}(B')\mid C].
\end{align*}
For these likelihood quantities, a superscript $b$ means that
$b_*$ is replaced by $b$ in the defining formula.
The amplification changes only the incident likelihood factors;
the conditional law $Z\mid C$ remains the original posterior.
The following lemma bounds prediction by the second moment of
this numerator.

\begin{lemma}[Conditioning on the remaining graph]
\label{lem:cusp-retained-design}
For the localized critical link with $\Lambda_{n,d}\le c_{\ref{prop:cusp-KL}}(K)$, there are fixed constants $b_*,C,c,\delta>0$ such that, for
sufficiently large $d$,
\begin{equation}
 \E_P\big[\E_P[H_d(\langle X,Y\rangle)\mid C,B,B^{\prime}]^2\big]
 \le C\E_C\E_Q[(N_C^{b_*})^2]
       +\|H_d\|_\infty^2\exp(-cd^\delta).
 \label{eq:cusp-app-ctwo-retained-design}
\end{equation}
\end{lemma}

\begin{proof}
In expectations involving $L_Z^{b_*}$ alone, $Q$ acts on its $m$
edge coordinates. Since the two roots are conditionally
independent before their edges are observed, the amplified likelihood
relative to $Q$ given $C$ is
\[
 L_C^{b_*}(B,B')=\E[L_Z^{b_*}(B)L_Z^{b_*}(B')\mid C].
\]
Conditional Jensen gives
\begin{align}
 \E_Q[(L_C^{b_*})^2]
 &\overset{\mathrm{(a)}}{\le}\E\big[(\E_Q[(L_Z^{b_*})^2])^2\mid C\big],\notag\\
 \E_C\big[(\E_Q[(L_C^{b_*})^2])^p\big]
 &\overset{\mathrm{(b)}}{\le}\E_Z\big[(\E_Q[(L_Z^{b_*})^2])^{2p}\big]
 \le\exp(Cpn\sqrt{\tau_{4,d}}+o(1)),
 \label{eq:cusp-app-ctwo-conditional-design-moment}
\end{align}
where $\mathrm{(a)}$ and $\mathrm{(b)}$ are conditional Jensen, followed
by averaging over $C$. The last inequality is
\eqref{eq:cusp-app-uq-design-moments} from the proof of
Lemma~\ref{lem:cusp-query}. Take $p=\lfloor d^\delta\rfloor$,
decreasing $\delta$ so that the bound also holds at $2p$.

Bayes' formula and the inequality for a fixed amplification
\eqref{eq:cusp-app-uq-denominator}, applied conditionally on $C$, give
\[
 \E_Q[(N_C^1)^2/L_C^1]
 \le\E_Q[(N_C^{b_*})^2]\E_Q[(L_C^{b_*})^2].
\]
For a sufficiently large fixed $M$, Markov's inequality gives
\[
 \P_C[\E_Q[(L_C^{b_*})^2]>M]
 \le M^{-p}\exp(Cpn\sqrt{\tau_{4,d}}+o(1))
 \le\exp(-cd^\delta).
\]
On $\{\E_Q[(L_C^{b_*})^2]\le M\}$ use the conditional Bayes bound;
on its complement
use $|\E[H_d(\langle X,Y\rangle)\mid C,B,B^{\prime}]|\le\|H_d\|_\infty$ under the actual law.
This proves \eqref{eq:cusp-app-ctwo-retained-design}.
\end{proof}

\begin{proof}[Proof of Proposition~\ref{prop:cusp-critical-two}]
Let $\mathcal A=\mathcal A_{r,\varnothing}$ be the observed graph
with edge $12$ withheld. Under $Q$, $Y_{12}$ has mean zero and
variance one. Bayes' formula gives
\begin{align}
 g_{r,2}
 &=\E_{\mathcal A}\E_Q\Bigg[
 \frac{\big(\E[H_d(T_{12})\mid\mathcal A]
 +Y_{12}\E[H_d(T_{12})\kappa_d(T_{12})\mid\mathcal A]\big)^2}
 {1+Y_{12}\E[\kappa_d(T_{12})\mid\mathcal A]}\Bigg]\notag\\
 &\overset{\mathrm{(a)}}{\le}C\Big(\tau_{3,d}^2+
   \E_P\big[\E_P[H_d(T_{12})\mid\mathcal A]^2\big]
   +\E_P\big[\E_P[H_d\kappa_d-\tau_{3,d}\mid\mathcal A]^2\big]\Big).
 \label{eq:cusp-app-ctwo-mutual}
\end{align}
Here $\mathrm{(a)}$ uses the dense lower bound on the denominator
and the zero mean and unit variance of $Y_{12}$ under $Q$.
We also use $\E[H_d(T_d)\kappa_d(T_d)]=\tau_{3,d}$.
The centered profile $H_d\kappa_d-\tau_{3,d}$ is canonical in both
roots. By H\"older, the local envelope, and
\eqref{eq:H-moments} in Lemma~\ref{lem:contracted},
\[
 \|H_d\kappa_d-\tau_{3,d}\|_4^2
 \le4\|H_d\|_8^2\|\kappa_d\|_8^2
 \le Cv_d\tau_{4,d}.
\]
Applying~\eqref{eq:cusp-query-two-minus}, a consequence of
Lemma~\ref{lem:cusp-query}, gives
\begin{align}
 \E_P\big[\E_P[H_d\kappa_d-\tau_{3,d}\mid\mathcal A]^2\big]
 &\le C(n\tau_{4,d}+n^2\tau_{4,d}/d)v_d\tau_{4,d}\notag\\
 &\le Cn\tau_{4,d}^2(v_d+nv_d/d)
 \overset{\mathrm{(b)}}{\le}Cn\tau_{4,d}^2.
 \label{eq:cusp-app-ctwo-decoration}
\end{align}
Here $\mathrm{(b)}$ uses $v_d=o(1)$ and
$nv_d/d\lesssim n\sqrt{\tau_{4,d}}d^{-1/6}$.
The exponentially small error is absorbed into the last bound.

We now bound $\E_P\big[\E_P[H_d(T_{12})\mid\mathcal A_{r,\varnothing}]^2\big]$ while retaining the posterior of the other latent points.
Let $m=r-2$, let $C$ be the observed graph on the $m$ remaining
vertices, and apply Lemma~\ref{lem:cusp-retained-design}.
In the product expansion, the coefficient indexed by
incident subsets $S,T\subseteq[m]$ is
$b_*^{|S|+|T|}\E[f_{S,T}(Z)\mid C]$, where
\begin{equation}
 f_{S,T}(Z)=\E_{X,Y}\Bigg[
 H_d(\langle X,Y\rangle)\prod_{i\in S}\kappa_d(\langle X,Z_i\rangle)
          \prod_{j\in T}\kappa_d(\langle Y,Z_j\rangle)\Bigg].
 \label{eq:cusp-app-ctwo-coefficient}
\end{equation}
Canonicality makes the coefficient zero if either subset is empty.
Fix $0<\theta<1$ and decrease $c_{\ref{prop:cusp-KL}}(K)$ so that
\eqref{eq:cusp-app-ctwo-conditional-design-moment} also holds with
the fixed amplification $b_*/\theta$.
Take $L=C\log d$ with $C$ sufficiently large.
Parseval and the same likelihood bound with amplification
$b_*/\theta$ give
\begin{align*}
 &~\sum_{|S|+|T|>L}b_*^{2(|S|+|T|)}
       \E_C\big[\E[f_{S,T}(Z)\mid C]^2\big] \le~\theta^{2L}\E_C\E_Q[(N_C^{b_*/\theta})^2]
 \overset{\mathrm{(a)}}{\le}C\theta^{2L}\|H_d\|_\infty^2
 \le d^{-30}.
\end{align*}
Here $\mathrm{(a)}$ uses
$|N_C^{b_*/\theta}|\le\|H_d\|_\infty L_C^{b_*/\theta}$ and
\eqref{eq:cusp-app-ctwo-conditional-design-moment} in
Lemma~\ref{lem:cusp-retained-design}. The larger fixed amplification
is feasible by~\eqref{eq:cusp-localized-envelope}.

For $S\cap T\ne\varnothing$, conditional Jensen permits replacing
the squared posterior coefficient by its prior second moment.
Extend the nonnegative sum to all intersecting pairs of subsets. Its prior value is exactly
\begin{align*}
 &~\E\Big[H_d(T_{12})H_d(T_{34})
   (1+b_*^2H_d(T_{13})+b_*^2H_d(T_{24})+b_*^4h_{4,d})^m\Big]\\
 &~{}-\E\Big[H_d(T_{12})H_d(T_{34})
   (1+b_*^2H_d(T_{13})+b_*^2H_d(T_{24}))^m\Big].
\end{align*}
The difference of powers and H\"older bound its absolute value by
\[
 Cm\|H_d(T_{12})H_d(T_{34})\|_2\|h_{4,d}\|_4
 \overset{\mathrm{(a)}}{\le}Cm\tau_{4,d}^2
 \le Cn\tau_{4,d}^2.
\]
The exponential factor is bounded by~\eqref{eq:cusp-app-critical-root-weight}.
Here $\mathrm{(a)}$ uses independence of $T_{12},T_{34}$,
$\|H_d\|_2^2=\tau_{4,d}$, and the critical bound in
\eqref{eq:cusp-app-root-growing}.
In particular, $S=T=\{i\}$ has coefficient
$H_{4,d}(1)=\tau_{4,d}$. For $S=\{i\}$ and $T=\{j\}$ with
$i\ne j$, convolution in \eqref{eq:cusp-app-ctwo-coefficient} gives
$H_{4,d}(T_{ij})$. When $m\ge2$, their posterior contribution is at most
$Cm^2\E_P[\E_P[H_{4,d}(T_{12})\mid A^{(m)}]^2]$.
Conditional Jensen and $m\le n$ bound this by
$Cn^2\E_P[\E_P[H_{4,d}(T_{12})\mid A^{(r)}]^2]$.
For $m<2$ there are no such terms.

It remains to sum disjoint subsets with $s=|S|$, $t=|T|$ and
$k=s+t\ge3$. By relabeling the roots, take $S=[s]$ and
$T=\{s+1,\ldots,s+t\}$, and write $f_{s,t}=f_{S,T}$.
It is canonical in each of its $k$ roots. Four replicas and
integration over their disjoint design vertices give
\[
 \E[f_{s,t}^4]=\langle h_{4,d}^s,T_{H_d}^{\otimes4}h_{4,d}^t\rangle.
\]
Since $\|T_{H_d}\|_{\mathrm{op}}=\|T_{\kappa_d}\|_{\mathrm{op}}^2$,
the critical bound~\eqref{eq:cusp-app-root-growing} implies
\begin{equation}
 \|f_{s,t}\|_4^2
 \overset{\mathrm{(a)}}{\le}\|T_{\kappa_d}\|_{\mathrm{op}}^4
       \|h_{4,d}\|_{2s}^{s/2}\|h_{4,d}\|_{2t}^{t/2}
 \overset{\mathrm{(b)}}{\le}\|T_{\kappa_d}\|_{\mathrm{op}}^4
       (Cs)^{s/2}(Ct)^{t/2}\tau_{4,d}^{k/2}.
 \label{eq:cusp-app-ctwo-bistar-moment}
\end{equation}
Here $\mathrm{(a)}$ uses the operator norm of $T_{H_d}^{\otimes4}$
and $\mathrm{(b)}$ uses~\eqref{eq:cusp-app-root-useful}, since
$2k\le2L\le c_0d^{1/3}$ for large $d$.
There are at most $m^k/(s!t!)$ such ordered subset pairs. The inequalities $j!\ge(j/e)^j$ and
$s\log s+(k-s)\log(k-s)\ge k\log(k/2)$ give
\[
 \sum_{s=1}^{k-1}
 \frac{(Cs)^{s/2}(C(k-s))^{(k-s)/2}}{s!(k-s)!}
 \le C^kk^{1-k/2}.
\]
Applying \eqref{eq:cusp-query-rough} in Lemma~\ref{lem:cusp-query}
and multiplying by $n^3$, the $k^2n\tau_{4,d}$ contribution is at most
\begin{align}
 Cn^4\tau_{4,d}\|T_{\kappa_d}\|_{\mathrm{op}}^4
       \sum_{k=3}^Lk^{3-k/2}(Cn\sqrt{\tau_{4,d}})^k
 &\overset{\mathrm{(a)}}{\le}Cn^7\tau_{4,d}^{5/2}\|T_{\kappa_d}\|_{\mathrm{op}}^4\overset{\mathrm{(b)}}{\le}Cn^7\tau_{4,d}^{7/2}d^{-1}.
 \label{eq:cusp-app-ctwo-sum-joint}
\end{align}
Here $\mathrm{(a)}$ sums a geometric series when $n\sqrt{\tau_{4,d}}$
is small; $\mathrm{(b)}$ uses~\eqref{eq:cusp-app-critical-scales}.
For the remaining contribution, write
$k^{1+k/2}=k^{5/2}(\sqrt k)^{k-3}$. The bound
\[
 C\sqrt L n\sqrt{\tau_{4,d}}
    \big(v_d^{11/24}+n\sqrt{\tau_{4,d}/d}\big)=o(1)
\]
allows a geometric majorant, which gives
\begin{align}
 &~n^3\|T_{\kappa_d}\|_{\mathrm{op}}^4
       \sum_{k=3}^L C^k k^{1+k/2}
       \big[n\sqrt{\tau_{4,d}}
          (v_d^{11/24}+n\sqrt{\tau_{4,d}/d})\big]^k
 \notag\\
 \overset{\mathrm{(a)}}{\le}&~Cn^6\tau_{4,d}^{3/2}\|T_{\kappa_d}\|_{\mathrm{op}}^4
       \big(v_d^{11/24}+n\sqrt{\tau_{4,d}/d}\big)^3
 \overset{\mathrm{(b)}}{\le}Cn^6\tau_{4,d}^3d^{-3\alpha/8-1/6}.
 \label{eq:cusp-app-ctwo-sum-internal}
\end{align}
Here $\mathrm{(a)}$ sums the geometric majorant; $\mathrm{(b)}$ uses
$v_d^{11/24}+n\sqrt{\tau_{4,d}/d}\le Cd^{-11\alpha/24}$ and
$n^3\|T_{\kappa_d}\|_{\mathrm{op}}^4\asymp n^3\tau_{4,d}^{3/2}d^{\alpha-1/6}$.
There are at most $\exp(C(\log d)^2)$ retained choices, while
the query errors are bounded by $\exp(-cd^\delta)$. Their total,
including the cutoff error, is at most $d^{-20}$ after multiplication
by $n^3$. Combining \eqref{eq:cusp-app-ctwo-mutual}--\eqref{eq:cusp-app-ctwo-sum-internal} gives
\begin{align}
 g_{r,2}
 &\le C\Big(\tau_{3,d}^2+n\tau_{4,d}^2
      +n^2\E_P\big[\E_P[H_{4,d}(T_{12})\mid A^{(r)}]^2\big]\Big)\notag\\
 &\quad{}+C\big(n^4\tau_{4,d}^{7/2}d^{-1}
             +n^3\tau_{4,d}^3d^{-3\alpha/8-1/6}\big)+d^{-20}.
 \label{eq:cusp-app-ctwo-first}
\end{align}
The second line is at most $Cn\tau_{4,d}^2$, since
$n\sqrt{\tau_{4,d}}\le1$ and $\tau_{4,d}\asymp d^{-2\alpha-5/3}$.

For the second prediction, convolution gives
$\|H_{4,d}\|_2^2=\tau_{8,d}$ and $\E[H_{4,d}\kappa_d]=\tau_{5,d}$.
Apply Lemma~\ref{lem:contracted} to $H_d$, whose degree one
coefficient is $\gamma_{1,d}^2/\sqrt d$. Its envelope gives
\begin{align}
 H_{4,d}(t)&=\frac{\gamma_{1,d}^4}{d}t
       +O(\tau_{4,d}(t^2+d^{-1})),\notag\\
 \|H_{4,d}\|_4^2
 &\lesssim\gamma_{1,d}^8/d^3+\tau_{4,d}^2/d^2
 \overset{\mathrm{(a)}}{\lesssim}\tau_{8,d}.
 \label{eq:cusp-critical-H4-moment}
\end{align}
Here $\mathrm{(a)}$ uses~\eqref{eq:cusp-app-critical-scales}; the first
term is the degree one contribution to $\tau_{8,d}$, and
$\tau_{4,d}^2/d^2\lesssim\tau_{8,d}$.
Also, H\"older and~\eqref{eq:cusp-localized-envelope} give
$\|H_{4,d}\kappa_d-\tau_{5,d}\|_2^2\le
\|H_{4,d}\|_4^2\|\kappa_d\|_4^2\lesssim v_d\tau_{8,d}$.
Substitution in~\eqref{eq:cusp-query-two}, a consequence of
Lemma~\ref{lem:cusp-query}, and
$\tau_{8,d}/d\asymp\|T_{\kappa_d}\|_{\mathrm{op}}^8$ give
\begin{equation}
 \E_P\big[\E_P[H_{4,d}(T_{12})\mid A^{(r)}]^2\big]\le C\big(\tau_{5,d}^2+n\tau_{4,d}\tau_{8,d}+n^2\tau_{4,d}\|T_{\kappa_d}\|_{\mathrm{op}}^8\big).
 \label{eq:cusp-app-ctwo-second}
\end{equation}

Finally,
\begin{align*}
 &~n^5\tau_{5,d}^2+n^6\tau_{4,d}\tau_{8,d}
       +n^7\tau_{4,d}\|T_{\kappa_d}\|_{\mathrm{op}}^8\\
 \overset{\mathrm{(a)}}{\le}&~Cn^4\tau_{4,d}^2
 \big[(n\sqrt{\tau_{4,d}})d^{-1/2}
 +(n\sqrt{\tau_{4,d}})^2d^{-1}
 +(n\sqrt{\tau_{4,d}})^3d^{\alpha-7/6}\big]
 \overset{\mathrm{(b)}}{\le}Cn^4\tau_{4,d}^2.
\end{align*}
Here $\mathrm{(a)}$ uses~\eqref{eq:cusp-app-critical-scales}, and
$\mathrm{(b)}$ uses $\alpha\le1$ and $n\sqrt{\tau_{4,d}}\le c$.
Equations~\eqref{eq:cusp-app-ctwo-first} and
\eqref{eq:cusp-app-ctwo-second}
give the stated prediction bounds. Multiplying
\eqref{eq:cusp-app-ctwo-first} by $n^3$ bounds $G_2$. For $r=2$, the identity
$\E[H_d(T_{12})\mid A_{12}]=\tau_{3,d}Y_{12}$ proves the same
bound directly. No lower bound on $|\tau_{3,d}|$ is used.
\end{proof}
\subsection{Order three}\label{app:cusp-order-three}

We prove the critical part of Proposition~\ref{prop:cusp-triangle-three}.
Throughout this subsection, $\beta=\alpha+1/3$ and
$\Lambda_{n,d}\le c_{\ref{prop:cusp-KL}}(K)$. Recall from~\eqref{eq:cusp-app-critical-scales}
that $v_d\asymp d^{-\alpha}$ and
$\sqrt{\tau_{4,d}}\asymp d^{-\alpha-5/6}$.
Thus $n\sqrt{\tau_{4,d}}\le c_{\ref{prop:cusp-KL}}(K)^{1/4}$, and
Lemma~\ref{lem:cusp-query} applies for sufficiently small $c_{\ref{prop:cusp-KL}}(K)$.
We retain the internal edges in its expansion, as in the proof sketch
in Section~\ref{sec:order-3-csup}.

Let $A(t)=\exp(nb_*^2H_d(t))$, with $b_*$ from
Lemma~\ref{lem:cusp-query}, and let $T_A$ be its integral operator.
For a fixed $m\ge3$, let $X,X'$ be independent prior $m$-tuples and write
$T_{ij}'=\langle X_i',X_j'\rangle$.
We keep $m$ general because the
same calculation will be used in Proposition~\ref{prop:cusp-critical-higher}.
For $F\subseteq K_m$, define
$f_F(X)=h_{m,d}(X)\prod_{\{i,j\}\in E(F)}b_*\kappa_d(T_{ij})$.
Expanding the internal factors gives
\begin{align}
&~\E\Big[h_{m,d}(X)h_{m,d}(X')
  \prod_{i=1}^m A(\langle X_i,X_i'\rangle)
  \prod_{i<j}\big(1+b_*^2\kappa_d(T_{ij})\kappa_d(T_{ij}')\big)\Big]
  \notag\\
=&~\sum_{F\subseteq K_m}
 \E\Big[f_F(X)f_F(X')\prod_{i=1}^m A(\langle X_i,X_i'\rangle)\Big]
 =\sum_{F\subseteq K_m}\langle f_F,T_A^{\otimes m}f_F\rangle.
\label{eq:cusp-critical-internal-expansion}
\end{align}
By~\eqref{eq:cusp-query-tensor} in Lemma~\ref{lem:cusp-query},
it remains to bound these terms together with
$n\tau_{4,d}\|h_{m,d}\|_4^2$ and the exponentially small error.

For $I\subseteq[m]$, let $X_I=(X_i)_{i\in I}$.
To bound the quadratic forms, we separate the mean of $f_F$ from
its centered dependence on subsets of roots.
These quantities have a graph interpretation. Let
$\operatorname{cone}(F)$ be the graph obtained by adding one vertex
adjacent to all $m$ roots, and define
$t(J,\kappa_d)=\E[\prod_{e\in E(J)}\kappa_d(T_e)]$. Then
$\E[f_F]=b_*^{|E(F)|}t(\operatorname{cone}(F),\kappa_d)$.
A squared conditional mean joins two copies of this graph along
the roots in $I$. See Figure~\ref{fig:cusp-cones}.
We first give a bound for these graph integrals. A biconnected
block is a maximal subgraph on at least three vertices that stays
connected after deleting any one vertex.

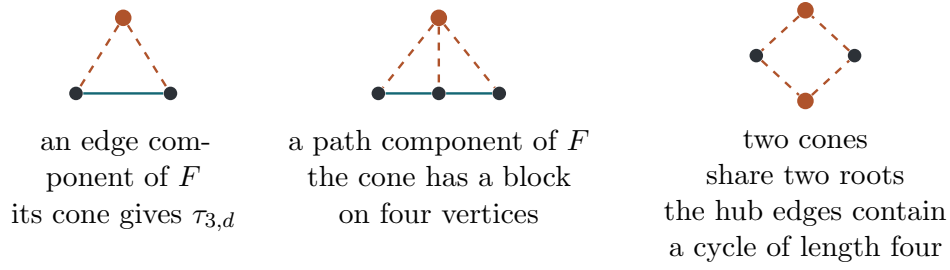
\begin{figure}[htbp]
\centering
\begingroup
\definecolor{proofteal}{RGB}{20,103,116}
\definecolor{proofcopper}{RGB}{175,83,43}
\definecolor{proofink}{RGB}{44,49,55}
\begin{tikzpicture}[vertex/.style={circle,fill=proofink,inner sep=1.8pt},hub/.style={circle,fill=proofcopper,inner sep=2.2pt},rootedge/.style={draw=proofteal,line width=.9pt},hubedge/.style={draw=proofcopper,line width=.9pt,dashed}]
\begin{scope}
\node[vertex] (a) at (0,0) {}; \node[vertex] (b) at (1.25,0) {}; \node[hub] (h) at (.625,1) {};
\draw[rootedge] (a)--(b); \draw[hubedge] (h)--(a) (h)--(b);
\node[anchor=north,align=center,text width=3.8cm] at (.625,-.35) {an edge component of $F$\\its cone gives $\tau_{3,d}$};
\end{scope}
\begin{scope}[xshift=4cm]
\node[vertex] (a) at (0,0) {}; \node[vertex] (b) at (.8,0) {}; \node[vertex] (c) at (1.6,0) {}; \node[hub] (h) at (.8,1) {};
\draw[rootedge] (a)--(b)--(c); \draw[hubedge] (h)--(a) (h)--(b) (h)--(c);
\node[anchor=north,align=center,text width=4.1cm] at (.8,-.35) {a path component of $F$\\the cone has a block on four vertices};
\end{scope}
\begin{scope}[xshift=9cm]
\node[vertex] (a) at (0,.5) {}; \node[vertex] (b) at (1.3,.5) {}; \node[hub] (h) at (.65,1.1) {}; \node[hub] (j) at (.65,-.1) {};
\draw[hubedge] (h)--(a)--(j)--(b)--(h);
\node[anchor=north,align=center,text width=4.1cm] at (.65,-.35) {two cones share two roots\\the hub edges contain a cycle of length four};
\end{scope}
\end{tikzpicture}
\endgroup
\caption{Internal components and the graphs created by conditional second moments. Solid edges belong to $F$; dashed edges join a hub to a root. Integrating an edge component preserves the signed cubic trace. A larger component, or two cones glued on at least two roots, has a biconnected block on at least four vertices.}
\label{fig:cusp-cones}
\end{figure}

The graph integrals below use only $\kappa_d$ and the centered square
$\kappa_d^2-v_d$. On the critical relation,
\eqref{eq:cusp-localized-envelope} gives
\begin{align}
 |e_d(y/\sqrt d)|+d^{1/6}|o_d(y/\sqrt d)|
 &\le Cd^{-\alpha/2}(1+|y|)^B,\notag\\
 |(e_d^2+o_d^2-v_d)(y/\sqrt d)|
       +d^{1/6}|2e_do_d(y/\sqrt d)|
 &\le Cd^{-\alpha}(1+|y|)^{2B}.
 \label{eq:cusp-edge-envelope}
\end{align}
The second line follows by squaring the first and using
$v_d\asymp d^{-\alpha}$. Both profiles have mean zero.
To combine these bounds, fix a connected simple graph $F$ with a
biconnected block on at least four vertices, and partition its edges
into $E_1,E_2$.

\begin{lemma}[Graph integrals]\label{lem:fixed-graph-budget}
For this graph and partition, on $\beta=\alpha+1/3$ and for sufficiently large $d$,
\[
 \Big|\E\Big[\prod_{e\in E_1}\kappa_d(T_e)
                    \prod_{e\in E_2}(\kappa_d(T_e)^2-v_d)\Big]\Big|
 \le C_{\ref{lem:fixed-graph-budget}}(F)
 d^{-\alpha(|E_1|/2+|E_2|)-(|V(F)|-1)/2-1/6}.
\]
\end{lemma}

\begin{proof}[Proof of Lemma~\ref{lem:fixed-graph-budget}]
Let $k_e=\kappa_d$ for $e\in E_1$ and
$k_e=\kappa_d^2-v_d$ for $e\in E_2$.
Let $q=|V(F)|$, $z_{ij}=\sqrt d\langle X_i,X_j\rangle$, and
let $W=(z_{ij})_{1\le i,j\le q}$ be symmetric with zero diagonal.
Within this proof, $\mu_d$ denotes the law of $\sqrt dT_d$,
and $\P_{\mathrm{Gram}}$ denotes the law of $W$.
We expand the Gram density relative to the product of these marginal
laws, and then bound each monomial that survives centering.

\begin{itemize}
\item \textbf{Step I.}
By~\cite[equation~(2.1)]{rouault2007asymptotic}, the relative density
is $c_{d,q}\exp(R_d(W))$ on $I+W/\sqrt d\succ0$, where
\begin{equation}
 R_d(W)=\frac{d-q-1}{2}\log\det(I+W/\sqrt d)
 -\frac{d-3}{2}\sum_{i<j}\log(1-z_{ij}^2/d).
 \label{eq:cusp-app-0920-gram-log-ratio}
\end{equation}
Since $\operatorname{tr}(W^2)=2\sum_{i<j}z_{ij}^2$, the quadratic
terms of order one cancel. For $\|W\|_{\mathrm{op}}<\sqrt d$,
the remaining series is
\begin{align*}
 R_d(W)
 &=\frac{q-2}{2d}\sum_{i<j}z_{ij}^2
 +\sum_{\ell\ge3}\frac{(-1)^{\ell+1}}{2\ell}
 d^{1-\ell/2}\operatorname{tr}(W^\ell)\\
 &\quad{}-(q+1)\sum_{\ell\ge3}\frac{(-1)^{\ell+1}}{2\ell}
 d^{-\ell/2}\operatorname{tr}(W^\ell)
 +\sum_{r\ge2}\frac{d^{1-r}-3d^{-r}}{2r}
 \sum_{i<j}z_{ij}^{2r}.
\end{align*}
Fix $0<\varepsilon<1/6$ and set
$\mathcal D_d=\{\max_{i<j}|z_{ij}|\le d^\varepsilon\}$.
For sufficiently large $d$, this box lies in the positive definite
region, and
\[
 \sup_{\mathcal D_d}|R_d(W)|
 \le C(q)\big(d^{-1/2+3\varepsilon}
                  +d^{-1+2\varepsilon}\big)=o(1), \quad  \P_{\mathrm{Gram}}(\mathcal D_d^c)
 +\mu_d^{\otimes\binom q2}(\mathcal D_d^c)
 \le C(q)\exp(-cd^{2\varepsilon}).
\]
For the second bound, apply Lemma~\ref{lem:gegen-moment} at degree one
and Markov's inequality with moment order proportional to $d^{2\varepsilon}$,
then take a union bound over the $\binom q2$ entries. In particular,
\[
 c_{d,q}
 =\frac{\P_{\mathrm{Gram}}(\mathcal D_d)}
 {\int_{\mathcal D_d}\exp(R_d(W))d\mu_d^{\otimes\binom q2}}
 \overset{\mathrm{(a)}}{=}1+o(1),
\]
where $\mathrm{(a)}$ uses the two bounds on $\mathcal D_d$.

For each fixed $J\ge1$, Taylor expansion gives polynomials
$P_0=1,P_1,\ldots,P_J$ such that, on $\mathcal D_d$,
\begin{equation}
 \Big|\exp(R_d(W))-\sum_{j=0}^Jd^{-j/2}P_j(W)\Big|
 \le C(q,J)d^{-(J+1)/2}(1+\|W\|_{\HS})^{3(J+1)}.
 \label{eq:cusp-app-0920-density-remainder}
\end{equation}
Let $R(h,W)$ be the right side of~\eqref{eq:cusp-app-0920-gram-log-ratio}
with $d=h^{-2}$. Its singularity at $h=0$ is removable because
the quadratic terms cancel.
For $0\le h\le d^{-1/2}$ and $W\in\mathcal D_d$,
\[
 |\partial_h^j R(h,W)|\le C(q,j)(1+\|W\|_{\HS})^{j+2},\quad
 |\partial_h^{J+1}\exp(R(h,W))|
 \le C(q,J)(1+\|W\|_{\HS})^{3(J+1)}.
\]
The first estimate follows from the logarithmic series. The second uses the
product rule and $\sup|R(h,W)|=o(1)$.
Taylor's formula then gives~\eqref{eq:cusp-app-0920-density-remainder}.

All fixed moments of $\mu_d$ are uniformly bounded. We integrate
\eqref{eq:cusp-app-0920-density-remainder} against the edge profiles to obtain
\begin{align*}
 \E\Big[\prod_{e\in F}k_e(T_e)\Big]
 &\overset{\mathrm{(a)}}{=}c_{d,q}\sum_{j=0}^Jd^{-j/2}
 \int P_j(W)\prod_{e\in F}k_e(z_e/\sqrt d)
       d\mu_d^{\otimes\binom q2}(z)\\
 &\quad{}+O_{F,J}\Big(
 d^{-\alpha(|E_1|/2+|E_2|)}
 \big[d^{-(J+1)/2}+\exp(-cd^{2\varepsilon})\big]\Big).
\end{align*}
Here $\mathrm{(a)}$ uses~\eqref{eq:cusp-edge-envelope} and H\"older's inequality:
the polynomial remainder contributes $d^{-(J+1)/2}$, while the integrals
outside $\mathcal D_d$ are exponentially small. The retained polynomials
are integrated over the full product domain, so centering and parity
apply to each coordinate.

\item \textbf{Step II.}
The trace terms expand over closed walks:
\[
 \operatorname{tr}(W^\ell)
 =\sum_{i_1,\ldots,i_\ell=1}^q
 z_{i_1i_2}\cdots z_{i_\ell i_1}.
\]
Before collecting like monomials, the series for $R_d$ and its
exponential therefore give products of the following terms.
The second column records the exponent $s$ in the factor $d^{-s}$:
\[
 \begin{array}{c|c}
 \text{factor}&\text{exponent }s\\ \hline
 \text{closed walk of length }\ell\ge3&\ell/2-1\text{ or }\ell/2\\
 z_{ij}^{2r},\ r\ge2&r-1\text{ or }r\\
 z_{ij}^{2}&1.
 \end{array}
\]
For a monomial $\prod_e z_e^{m_e}$ that has nonzero integral,
\[
 m_e\ge1\quad(e\in F),\quad
 m_e\text{ is even}\quad(e\notin F),
\]
and each edge of $F$ contributes at most
\[
 \Big|\int z^{m_e}k_e(z/\sqrt d)d\mu_d(z)\Big|
 \le C(m_e)
 \begin{cases}
  d^{-\alpha/2},&e\in E_1,\\
  d^{-\alpha},&e\in E_2
 \end{cases}
 \times
 \begin{cases}
  1,&m_e\text{ even},\\
  d^{-1/6},&m_e\text{ odd}.
 \end{cases}
\]
The first condition follows from centering; the others follow from the
symmetry of $\mu_d$ and~\eqref{eq:cusp-edge-envelope}.
If $o$ edges have odd multiplicity, the integral consequently gains
$d^{-o/6}$. The support contains $F$, hence is connected on its
$q$ vertices.

\item \textbf{Step III.}
Decompose every nonsimple closed walk into simple cycles of length
at least three and backtracks of length two. If there are $r\ge2$
pieces, including $k_2$ backtracks, the difference between the
original walk cost and the sum of the costs $\ell_i/2-1$ is
\[
 \frac{\sum_{i=1}^r\ell_i}{2}-1
 -\sum_{i=1}^r\Big(\frac{\ell_i}{2}-1\Big)
 =r-1\ge\frac{k_2}{2}.
\]
We may therefore assign cost $1/2$ to every backtrack and cost
$\ell/2-1$ to every simple cycle of length $\ell$.
Walk terms with cost $\ell/2$ have one additional unit of cost. For each remaining factor supported on one edge, also assign cost $1/2$.
These assignments preserve multiplicities and do not increase the total cost.

Since the support is connected, order the pieces so that each new
piece meets the union of the earlier pieces. At any stage let $v_*$ be the number of vertices,
$o_*$ the number of edges with odd multiplicity, and set
\[
 s_*=(\text{assigned cost})+o_*/6-(v_*-1)/2.
\]
For the first piece and for each subsequent addition, the following
cases give $s_*\ge0$:
\begin{align*}
 s_*&=0,\quad\text{first even edge},\quad s_*=(\ell-3)/6,\quad\text{first simple cycle of length }\ell,\\
 \Delta s_*&=1/2-\Delta v_*/2\ge0,\quad\text{added even edge}.
\end{align*}
For an added simple cycle of length $\ell$ with $k<\ell$ old edges,
$\Delta v_*\le\ell-k-1$ and $\Delta o_*\ge\ell-2k$, so
\[
 \Delta s_*
 \ge\ell/2-1+(\ell-2k)/6-(\ell-k-1)/2 =(\ell+k-3)/6\ge0.
\]
If $k=\ell$, then $\Delta v_*=0$, $\Delta o_*\ge-\ell$, and
$\Delta s_*\ge\ell/3-1\ge0$.
If the final value is $s_*=0$, every increment is zero. The formulas
therefore allow only a first edge or triangle, followed by a pendant
even edge, a triangle meeting the old support at one vertex, or a
triangle all of whose edges are already present. Such a support has
only bridges and triangle blocks. It cannot contain the assumed block
of $F$, so $s_*>0$. Since $s_*$ is a multiple of $1/6$,
the original exponent $s$ satisfies
\[
 s+o/6\ge(\text{assigned cost})+o/6
 \overset{\mathrm{(a)}}{\ge}(q-1)/2+1/6.
\]
Here $\mathrm{(a)}$ uses $s_*\ge1/6$.
\end{itemize}

Each surviving monomial is therefore bounded by
$C(F,J)d^{-\alpha(|E_1|/2+|E_2|)}d^{-(q-1)/2-1/6}$.
Choose $J$ so that $(J+1)/2>(q-1)/2+1/6$ and sum the finitely many
retained terms. The integrated Taylor remainder and the exponential
tail error satisfy the same bound, proving the claim.
\end{proof}

We now bound the terms in the query estimate. Fix an integer $M\ge4$.
Proposition~\ref{prop:cusp-critical-higher} uses the estimates up to
order $M$ before applying the bound for growing orders.

\begin{lemma}[Fixed critical orders]\label{lem:cusp-critical-fixed-orders}
For the localized mixed link with $\beta=\alpha+1/3$ and
$\Lambda_{n,d}\le c_{\ref{prop:cusp-KL}}(K)$, for sufficiently large $d$,
\[
 G_3\lesssim n^4\tau_{4,d}^2,\quad
 \sum_{m=4}^{\min\{M,n-1\}}G_m\le C_{\ref{lem:cusp-critical-fixed-orders}}(M)\Lambda_{n,d}.
\]
\end{lemma}

\begin{proof}
Fix $3\le m\le\min\{M,n-1\}$ and let $e(F)=|E(F)|$.
By~\eqref{eq:cusp-query-tensor} in Lemma~\ref{lem:cusp-query} and
\eqref{eq:cusp-app-root-growing},
\begin{equation}
 G_m
 \le C(n\sqrt{\tau_{4,d}})^{m+2}
   +Cn^{m+1}\sum_{F\subseteq K_m}
       \langle f_F,T_A^{\otimes m}f_F\rangle
   +Cn^{m+1}\exp(-cd^\delta).
 \label{eq:cusp-app-fixed-query}
\end{equation}
Here we used $\sum_{r=m}^{n-1}\binom rm=\binom n{m+1}$.
The first term is at most $Cn^4\tau_{4,d}^2$ because
$n\sqrt{\tau_{4,d}}\le1$. The exponential error is also smaller
than this quantity, since
$n\lesssim d^{\alpha+5/6}$ and
$\tau_{4,d}^2\asymp d^{-4\alpha-10/3}$.

To separate the means from the dependence on the roots, define
\[
 P_If_F=\sum_{J\subseteq I}(-1)^{|I|-|J|}
                 \E[f_F\mid X_J],\quad I\subseteq[m].
\]
Here $\E[f_F\mid X_\varnothing]=\E[f_F]$.
These projections are orthogonal, sum to $f_F$, and are centered
in every coordinate in $I$.
Rotational invariance gives $T_A1=\E[A(T_d)]$.
Since $T_A$ preserves constants and their orthogonal complement,
\begin{equation}
 \langle f_F,T_A^{\otimes m}f_F\rangle
 =\sum_{I\subseteq[m]}\E[A(T_d)]^{m-|I|}
       \langle P_If_F,T_A^{\otimes |I|}P_If_F\rangle \overset{\mathrm{(a)}}{\le} C\sum_{I\subseteq[m]}
       (n\sqrt{\tau_{4,d}/d})^{|I|}\|P_If_F\|_2^2.
 \label{eq:cusp-app-0920-fixed-projections}
\end{equation}
For $I=\varnothing$, the inner product means $|\E[f_F]|^2$.
Here $\mathrm{(a)}$ uses the constant and centered operator bounds
in Lemma~\ref{lem:cusp-query}. We estimate the four types of terms
in~\eqref{eq:cusp-app-0920-fixed-projections} separately.

\begin{itemize}
\item \textbf{Term I.} For $I=\varnothing$, we bound $|\E[f_F]|^2$.
For $m=3$, integration gives the explicit cases
\[
 \E[f_F]=
 \begin{cases}
 0,&e(F)\le1,\\
 b_*^2\E[\kappa_d(T_d)H_d(T_d)^2],&e(F)=2,\\
 b_*^3\E\big[\prod_{1\le i<j\le4}\kappa_d(T_{ij})\big],&e(F)=3.
 \end{cases}
\]
In the first case, a root is incident to only one kernel factor.
For two internal edges, integrating the endpoints of their path
gives the two factors of $H_d$. For three internal edges, the graph is $K_4$.
We then have
\begin{equation}
 n^4\sum_{F\subseteq K_3}|\E[f_F]|^2
 \overset{\mathrm{(a)}}{\lesssim}n^4\big(v_d\|H_d\|_4^4+d^{-6\alpha-10/3}\big)
 \overset{\mathrm{(b)}}{\lesssim}n^4\tau_{4,d}^2.
 \label{eq:cusp-app-three-means}
\end{equation}
Here $\mathrm{(a)}$ uses Cauchy--Schwarz for the two-edge case
and Lemma~\ref{lem:fixed-graph-budget} for $K_4$;
$\mathrm{(b)}$ uses Lemma~\ref{lem:contracted}, $v_d\lesssim1$,
and $\tau_{4,d}^2\asymp d^{-4\alpha-10/3}$.

For $m\ge4$, an isolated vertex of $F$ again gives
$\E[f_F]=0$. Otherwise, let $F_1,\ldots,F_b$ be its connected
components, with $c_j=|V(F_j)|$ and $e_j=|E(F_j)|$.
Conditioning on the integrated center and using rotational
invariance gives the exact factorization
\[
 \E[f_F]=b_*^{e(F)}\prod_{j=1}^b
          t(\operatorname{cone}(F_j),\kappa_d).
\]
If $c_j=2$, this factor equals the signed $\tau_{3,d}$.
If $c_j\ge3$, the cone is biconnected, and
Lemma~\ref{lem:fixed-graph-budget} gives
\[
 n^{c_j+1}t(\operatorname{cone}(F_j),\kappa_d)^2
 \overset{\mathrm{(a)}}{\lesssim}n^{c_j+1}d^{-\alpha(c_j+e_j)-c_j-1/3} \asymp(n\sqrt{\tau_{4,d}})^{c_j+1}
       d^{-\alpha(e_j-1)-(c_j-3)/6}
 \overset{\mathrm{(b)}}{\lesssim}n^4\tau_{4,d}^2.
\]
Here $\mathrm{(a)}$ applies Lemma~\ref{lem:fixed-graph-budget} to
a cone with $c_j+1$ vertices and $c_j+e_j$ edges;
$\mathrm{(b)}$ uses $c_j\ge3$, $e_j\ge c_j-1$, and $n\sqrt{\tau_{4,d}}\le1$.
Consequently,
\[
 n^{m+1}|\E[f_F]|^2
 =b_*^{2e(F)}n^{1-b}
   \prod_{j=1}^b n^{c_j+1}
       t(\operatorname{cone}(F_j),\kappa_d)^2 \le Cn^{1-b}\Lambda_{n,d}^b
 \le C\Lambda_{n,d}.
\]
In particular, a perfect matching with $\ell=m/2$ edges gives
$b_*^{2\ell}n^{1-\ell}(n^3\tau_{3,d}^2)^\ell$.
Thus no absolute value is taken inside a triangle integral.

\item \textbf{Term II.} Suppose $|I|=1$.
By rotational invariance, $\E[f_F\mid X_i]=\E[f_F]$,
so $P_{\{i\}}f_F=0$ for every $i\in[m]$.

\item \textbf{Term III.} When $F=\varnothing$, $f_F=h_{m,d}$ is centered in every root, so only
$I=[m]$ survives. Also
$h_{m,d}(-X_1,\ldots,-X_m)=h_{m,d}(X_1,\ldots,X_m)$.
Its harmonic degrees are positive and have even sum.
For odd $m$, at least one degree is positive and even.
We use the centered operator bound in Lemma~\ref{lem:cusp-query}
in every coordinate and its even bound in one coordinate when $m$ is odd:
\begin{align*}
 n^{m+1}\langle h_{m,d},T_A^{\otimes m}h_{m,d}\rangle
 &\overset{\mathrm{(a)}}{\le} Cn^{m+1}(n\sqrt{\tau_{4,d}})^m
 d^{-\lceil m/2\rceil}\|h_{m,d}\|_2^2\\
 &\overset{\mathrm{(b)}}{\le} Cn(n\sqrt{\tau_{4,d}})^{2m}
 d^{-\lceil m/2\rceil}\\
 &\overset{\mathrm{(c)}}{\le} C(n\sqrt{\tau_{4,d}})^7
 d^{\alpha-7/6}\le Cn^4\tau_{4,d}^2.
\end{align*}
Here $\mathrm{(a)}$ uses these operator bounds and the parity of the
harmonic degrees; $\mathrm{(b)}$ uses $\|h_{m,d}\|_2\le\|h_{m,d}\|_4$
and~\eqref{eq:cusp-app-root-growing}; $\mathrm{(c)}$ uses $m\ge3$, $n\sqrt{\tau_{4,d}}\le1$,
and $\sqrt{\tau_{4,d}}\asymp d^{-\alpha-5/6}$.
The last inequality follows from $\alpha\le1$.

\item \textbf{Term IV.} Suppose $F\ne\varnothing$ and $k=|I|\ge2$.
Orthogonality gives
$\|P_If_F\|_2^2\le\|\E[f_F\mid X_I]\|_2^2$.
To write the latter as a graph integral, set $X_i'=X_i$ for
$i\in I$ and take all other $X_i,X_i'$, and two centers $Z,Z'$
independently from $\sigma_d$. Then
\begin{equation}
 \|\E[f_F\mid X_I]\|_2^2=b_*^{2e(F)}\E\Big[
 \prod_{j=1}^m\kappa_d(\langle Z,X_j\rangle)
                  \kappa_d(\langle Z',X_j'\rangle)
 \prod_{e\in E(F)}\kappa_d(T_e)\kappa_d(T_e')\Big].
 \label{eq:cusp-app-glued-moment}
\end{equation}
This graph has $2+2m-k$ vertices and contains $K_{2,k}$
between the centers and the shared roots, hence a block
on at least four vertices.
The repeated edges are exactly $D=E(F[I])$, where $F[I]$
is the subgraph induced by the roots in $I$.
Their factors expand as
\[
 \prod_{e\in D}\kappa_d(T_e)^2
 =\sum_{J\subseteq D}v_d^{|D|-|J|}
       \prod_{e\in J}\big(\kappa_d(T_e)^2-v_d\big).
\]
Each retained repeated edge contributes a centered square of size
$d^{-\alpha}$; each deleted edge contributes $v_d\asymp d^{-\alpha}$.
Thus the total amplitude is $d^{-\alpha(m+e(F))}$ in every term.
Deleting internal edges preserves all $2+2m-k$ vertices, the center
edges, and the $K_{2,k}$ subgraph. Lemma~\ref{lem:fixed-graph-budget}
therefore gives
\begin{equation}
 \|P_If_F\|_2^2
 \le Cd^{-\alpha(m+e(F))}
          d^{-(2+2m-k-1)/2-1/6}=Cd^{-\alpha(m+e(F))-m+k/2-2/3}.
 \label{eq:cusp-app-fixed-conditional-moment}
\end{equation}
Its contribution to~\eqref{eq:cusp-app-fixed-query} is therefore
\[
 n^{m+1}(n\sqrt{\tau_{4,d}/d})^k\|P_If_F\|_2^2
 \overset{\mathrm{(a)}}{\le}C(n\sqrt{\tau_{4,d}})^{m+1+k}
 d^{\alpha(1-e(F))-(m-1)/6}
 \overset{\mathrm{(b)}}{\le}Cn^4\tau_{4,d}^2.
\]
Here $\mathrm{(a)}$ uses~\eqref{eq:cusp-app-fixed-conditional-moment}
and $\sqrt{\tau_{4,d}}\asymp d^{-\alpha-5/6}$;
$\mathrm{(b)}$ uses $e(F)\ge1$, $m\ge3$, $k\ge2$, and
$n\sqrt{\tau_{4,d}}\le1$.
\end{itemize}

There are finitely many choices of $m,F,I$ for $m\le M$.
For $m=3$, all four cases are bounded by
$Cn^4\tau_{4,d}^2$; for $m\ge4$, the mean may additionally
contribute $Cn^3\tau_{3,d}^2$.
Summing these bounds in~\eqref{eq:cusp-app-fixed-query}
proves both assertions.
\end{proof}
Lemma~\ref{lem:cusp-critical-fixed-orders} proves the critical
part of Proposition~\ref{prop:cusp-triangle-three}.

\section{Proof for smooth kernels}
\label{app:finite-harmonic}

We prove the estimates used in
Sections~\ref{sec:smooth-order-two}--\ref{sec:smooth-higher-orders}.
Here $P$ is the finite harmonic model in~\eqref{eq:finite-harmonic-model},
with coefficients satisfying~\eqref{eq:finite-bounds}
and~\eqref{eq:finite-field}. We use
$n=d^{m+q_m/3}/\omega_d$, with $\omega_d\to\infty$, as in
Section~\ref{sec:polynomial}. Unsubscripted expectations are under
independent prior samples. Recall the Gram event
$\mathcal G_D$ from~\eqref{eq:smooth-gram-event}.

\subsection{The prior exponential bound}
\label{app:proof-finite-mgf}
\label{app:finite-setup}

The proof of Lemma~\ref{lem:finite-mgf} estimates each harmonic
pair separately. Conditioning on the sample carrying the larger
degree gives the smaller Gram bound.

\begin{proof}[Proof of Lemma~\ref{lem:finite-mgf}]
For $q\le\ell$, condition on $X'$. On $\mathcal G_D(X')$,
\begin{equation}
 \|\mathbf U_\ell(X')\|_{\op}
 \le C_{\ref{lem:gram-event}}(D,\mathcal I)
 \big(1+n/N_{\ell,d}+\log(n+N_{\ell,d})\big).
 \label{eq:mc-8}
\end{equation}
For fixed $c_0$, condition~\eqref{eq:finite-field} gives
$|c_0b_qb_\ell|\|\mathbf U_\ell(X')\|_{\op}=o(d)$.
Lemma~\ref{lem:signed-block} therefore gives
\begin{equation}
 \log\E_X[\exp(c_0b_qb_\ell O_{q\ell})\mid X']
 \le Cn\frac{b_q^2b_\ell^2}{N_{q,d}}
 \Big(1+\log(n+N_{\max})+\frac n{N_{\ell,d}}\Big).
 \label{eq:mc-9b}
\end{equation}
When $\ell<q$, exchange $X$ and $X'$.
H\"older's inequality with $|\mathcal I|^2$ equal exponents gives
\[
 \E[\1_{\mathcal G_D(X)\cap\mathcal G_D(X')}e^{S_\kappa}]
 \le\prod_{q,\ell}
 \Big(\E[\1_{\mathcal G_D(X)\cap\mathcal G_D(X')}
 e^{|\mathcal I|^2b_qb_\ell O_{q\ell}}]\Big)^{1/|\mathcal I|^2}.
\]
For $q\le\ell$, discard $\1_{\mathcal G_D(X)}$ and apply
\eqref{eq:mc-9b}; for $\ell<q$, condition on $X$.
The resulting sums satisfy
\begin{align*}
 \sum_{q\le\ell}\frac{b_q^2b_\ell^2}{N_{q,d}}
 &\le\Big(\sum_q\frac{b_q^2}{N_{q,d}}\Big)
       \Big(\sum_\ell b_\ell^2\Big)
       \overset{\mathrm{(a)}}{\lesssim}d^{-m},\\
 \sum_{q\le\ell}\frac{b_q^2b_\ell^2}{N_{q,d}N_{\ell,d}}
 &\le\Big(\sum_q\frac{b_q^2}{N_{q,d}}\Big)^2
       \overset{\mathrm{(a)}}{\lesssim}d^{-2m},
\end{align*}
where $\mathrm{(a)}$ uses~\eqref{eq:finite-bounds}.
The symmetric pairs give the same bounds. By the definition of $V_0$,
\begin{equation}
 \log\E[\1_{\mathcal G_D(X)\cap\mathcal G_D(X')}
 e^{S_\kappa}]\le CV_0.
 \label{eq:mc-10}
\end{equation}
For $a(S_\kappa+F)$, every pair coefficient has absolute value
at most $a(1+t_0)|b_qb_\ell|$. Condition~\eqref{eq:finite-field}
still holds, and Lemma~\ref{lem:signed-block} allows either sign.
The calculation giving~\eqref{eq:mc-10} therefore proves
\eqref{eq:finite-mgf}.
\end{proof}

\subsection{Order three}
\label{app:proof-finite-order-three}
\label{app:cubic-event}

We first bound the largest coefficient of the conditional
$U$-statistic in Section~\ref{sec:smooth-order-three}.

\begin{lemma}[Cubic envelope event]
\label{lem:cubic-envelope}
For every fixed $D>0$,
\begin{equation}
 \P\Big[\max_{i<j<k}|h_3(X_i,X_j,X_k)|
 >C_{\ref{lem:cubic-envelope}}(D)
   (\log(n+d))^{C_{\ref{lem:cubic-envelope}}(D)}
   d^{-3m/2-3q_m/4}\Big]\le(n+d)^{-D}.
 \label{eq:app-event-2}
\end{equation}
\end{lemma}

\begin{proof}
Recall that $H$ is the kernel of $T_\kappa^2$.
By~\eqref{eq:finite-bounds},
$\|H\|_2\le\|T_\kappa\|_{\op}\|\kappa\|_2
\lesssim d^{-m-q_m/2}$. For $p\ge2$,
\begin{equation}
 \|h_3\|_p\overset{\mathrm{(a)}}{\le}Cp^C\|h_3\|_2
 =Cp^C\big(\E[H(X_1,X_2)^3]\big)^{1/2}
 \overset{\mathrm{(b)}}{\lesssim}p^Cd^{-3m/2-3q_m/4}.
 \label{eq:app-event-1}
\end{equation}
Here $\mathrm{(a)}$ uses Lemma~\ref{lem:spherical-hypercontractivity}
for the fixed degree polynomial $h_3$; the equality follows from
Fubini; $\mathrm{(b)}$ uses
$\E[H^3]\le\|H\|_3^3\lesssim\|H\|_2^3$ from the same lemma.
Take $p=(D+4)\log(n+d)$. Markov's inequality and a union bound give
\[
 \P\Big[\max_{i<j<k}|h_3(X_i,X_j,X_k)|
 >eCp^Cd^{-3m/2-3q_m/4}\Big]
 \le\binom n3e^{-p}\le(n+d)^{-D},
\]
which proves~\eqref{eq:app-event-2}.
\end{proof}

\begin{proof}[Proof of Proposition~\ref{prop:finite-order-three}]
Recall $\kappa=\sum_{q\in\mathcal I}b_qU_q$ and that the kernel of
$T_\kappa^2$ is $H=\sum_{q\in\mathcal I}(b_q^2/N_{q,d})U_q$.
The coefficient bounds in~\eqref{eq:finite-bounds} imply
\begin{equation}
 \|\kappa\|_2\lesssim d^{-m/2},
 \quad\|T_\kappa\|_{\op}\lesssim d^{-(m+q_m)/2},
 \quad\|H\|_2\lesssim d^{-m-q_m/2}.
 \label{eq:smooth-cubic-spectral}
\end{equation}
For the last estimate, use
$\|T_\kappa^2\|_{\HS}\le
\|T_\kappa\|_{\op}\|T_\kappa\|_{\HS}$.

\noindent\textbf{Step I.} We verify the slice bounds used in~\eqref{eq:hc-9}.
Recall $a_r$ from Section~\ref{sec:smooth-order-three}.
By Lemma~\ref{lem:spherical-hypercontractivity},
$\sup_x\|\kappa(x,\cdot)\|_p\lesssim d^{-m/2}$
for every fixed $p$. H\"older's inequality gives
\begin{equation}
 a_0\le\sup_{x_1,x_2,x_3}
       \prod_{j=1}^3\|\kappa(x_j,\cdot)\|_3
 \lesssim d^{-3m/2}.
 \label{eq:hc-1}
\end{equation}
For one integrated variable, use~\eqref{eq:smooth-cubic-spectral}:
\begin{align}
 a_1&=\sup_{x_1,x_2}
 \big\|T_\kappa[\kappa(x_1,\cdot)\kappa(x_2,\cdot)]\big\|_2\overset{\mathrm{(a)}}{\le}\|T_\kappa\|_{\op}
 \sup_{x_1,x_2}\|\kappa(x_1,\cdot)\|_4
                  \|\kappa(x_2,\cdot)\|_4
 \lesssim d^{-3m/2-q_m/2}.
 \label{eq:hc-2}
\end{align}
Here $\mathrm{(a)}$ uses the operator norm and H\"older's inequality;
the last bound uses~\eqref{eq:smooth-cubic-spectral}
and Lemma~\ref{lem:spherical-hypercontractivity}.
Rotation invariance makes the slice with two variables integrated
constant in the third variable. Thus
\begin{align}
 a_2=a_3,\quad a_3^2&=\E[H(X_1,X_2)^3]
 \le\|H\|_3^3\overset{\mathrm{(a)}}{\lesssim}\|H\|_2^3
 \overset{\mathrm{(b)}}{\lesssim}d^{-3m-3q_m/2},\label{eq:hc-4}
\end{align}
where $\mathrm{(a)}$ uses Lemma~\ref{lem:spherical-hypercontractivity}
and $\mathrm{(b)}$ uses~\eqref{eq:smooth-cubic-spectral}.

\noindent\textbf{Step II.} Condition on $X'\in\mathcal A_D$, with
$\mathcal A_D$ defined in~\eqref{eq:smooth-cubic-event}.
The coefficients $h_3(X_i',X_j',X_k')$ are fixed, and
$h_3$ is canonical in each unprimed variable. The scales in
\eqref{eq:t-7} of Lemma~\ref{lem:adamczak-cor} satisfy
\[
 \max_{i<j<k}|h_3(X_i',X_j',X_k')|n^{r/2}a_r
 \le L_nd^{-3m/2-3q_m/4}n^{r/2}a_r\lesssim D_r.
\]
Here $D_r$ is defined in~\eqref{eq:smooth-cubic-scales}.
Equation~\eqref{eq:t-7} in that lemma now gives~\eqref{eq:hc-9}
for $t\ge C\max_{0\le r\le3}D_r$.

Let
\begin{equation}
 n=d^{m+q_m/3}/\omega_d,
 \quad\omega_d\to\infty,
 \quad t=n^{-1-\delta}.
 \label{eq:hc-9a}
\end{equation}
At $\delta=0$, substitution in the four terms gives
\begin{equation}
 (\zeta_0,\zeta_1,\zeta_2,\zeta_3)
 =\Big(\frac{2m}{3}+\frac{5q_m}{36},
       \frac{3m}{5}+\frac{3q_m}{10},
       \frac m2+\frac{5q_m}{12},
       \frac m3+\frac{4q_m}{9}\Big).
 \label{eq:hc-10}
\end{equation}
More explicitly, for $0\le r\le3$,
\[
 \Big(\frac{t}{D_r}\Big)^{2/(6-r)}
 \ge d^{\zeta_r-2\delta(m+q_m/3)/(6-r)+o(1)},
\]
since $L_n=d^{o(1)}$ and the powers of $\omega_d$ are positive.
As $m\ge q_m$ and $q_m\in\{1,2\}$,
\[
 \min_{0\le r\le3}\zeta_r>\frac{2q_m}{3}.
\]
Choose $\delta>0$ so small that
$\zeta:=\min_r\{\zeta_r-2\delta(m+q_m/3)/(6-r)\}>2q_m/3$.
In particular, $t/D_r\to\infty$ for every $r$.
Averaging the conditional estimate~\eqref{eq:hc-9} over $X'$ gives
\begin{equation}
 \P[|T_3|>n^{-1-\delta},\mathcal A_D(X')]
 \le\exp\big(-d^{\zeta+o(1)}\big),
 \quad\zeta>2q_m/3.
 \label{eq:hc-10a}
\end{equation}

\noindent\textbf{Step III.} We transfer~\eqref{eq:hc-10a} to the posterior.
Let $\mathcal B_t=\{|T_3|>t\}$ and
$\mathcal G=\mathcal G_D(X)\cap\mathcal G_D(X')
\cap\mathcal A_D(X)\cap\mathcal A_D(X')$.
Since $\Gamma_d\le e^{S_\kappa}$ by~\eqref{eq:finite-likelihood-overlap},
\begin{align}
 \E[\Gamma_d\1_{\mathcal B_t\cap\mathcal G}]
 &\overset{\mathrm{(a)}}{\le}\Big(\E[e^{2S_\kappa}
 \1_{\mathcal G_D(X)\cap\mathcal G_D(X')}]\Big)^{1/2}
 \P[\mathcal B_t\cap\mathcal G]^{1/2}\overset{\mathrm{(b)}}{\le}e^{C(1+V_0)}
 \P[\mathcal B_t\cap\mathcal G]^{1/2}.
 \label{eq:hc-12}
\end{align}
Here $\mathrm{(a)}$ is Cauchy--Schwarz and $\mathrm{(b)}$ uses
\eqref{eq:finite-mgf} in Lemma~\ref{lem:finite-mgf} with $a=2$.
Since $n\le d^{m+q_m/3}$, the definition of $V_0$ gives
$V_0\le d^{2q_m/3+o(1)}$. Thus
\begin{align*}
 \P_{\post}[\mathcal B_t\cap\mathcal G]
 &\overset{\mathrm{(a)}}{\le}
 \big(\E[\Gamma_d\1_{\mathcal B_t\cap\mathcal G}]\big)^{1/2}\overset{\mathrm{(b)}}{\le}
 \exp\big(C(1+V_0)-d^{\zeta+o(1)}\big)
 =\exp\big(-d^{\zeta+o(1)}\big),
\end{align*}
where $\mathrm{(a)}$ is~\eqref{eq:4-6} in Lemma~\ref{lem:planting}
and $\mathrm{(b)}$ combines~\eqref{eq:hc-10a} and~\eqref{eq:hc-12}.
The last equality uses $\zeta>2q_m/3$.
The prior marginal identity in Lemma~\ref{lem:planting} also gives
\[
 \P_{\post}[\mathcal G^c]
 \le2\P[\mathcal G_D^c]+2\P[\mathcal A_D^c]
 \lesssim n^{-D}.
\]
Finally, $|T_3|\lesssim n^3d^{-3m}$ by~\eqref{eq:hc-1}, so
\begin{align*}
 \E_{\post}[|T_3|]
 &\le t+Cn^3d^{-3m}
   \big(\P_{\post}[\mathcal B_t\cap\mathcal G]
       +\P_{\post}[\mathcal G^c]\big)\\
 &\le n^{-1-\delta}
 +Cn^3d^{-3m}\Big(\exp\big(-d^{\zeta+o(1)}\big)+n^{-D}\Big)
 =o(n^{-1}),
\end{align*}
where $D>5$ suffices. Substituting
\eqref{eq:smooth-cubic-replica} yields
\begin{equation}
 \binom n4\bar g_3
 =\frac{n-3}{4}\E_{\post}[T_3]
 \le\frac{n-3}{4}\E_{\post}[|T_3|]=o(1).
 \label{eq:hc-13}
\end{equation}
\end{proof}
\subsection{Quadratic overlaps for higher orders}
\phantomsection\label{app:proof-finite-higher}
We bound the posterior moments of $Y_H$ needed in
Section~\ref{sec:smooth-higher-orders}. Recall $Y_H$ from
\eqref{eq:hc-15} and $O_{q\ell}$ from~\eqref{eq:smooth-quadratic-form}.

\begin{lemma}[Quadratic overlap]
\label{lem:finite-quadratic-overlap}
For every fixed $s\ge1$,
\begin{equation}
 \|Y_H\|_{L^s(\P_{\post})}
 \le C_{\ref{lem:finite-quadratic-overlap}}(s)\big(nd^{-2m}+(1+V_0)d^{-(m+q_m)}\big).
 \label{eq:mc-17}
\end{equation}
\end{lemma}

\begin{proof}
The addition formula writes $U_q(x,y)=\langle z_q(x),z_q(y)\rangle$
with $\|z_q(x)\|=1$. Hence
\begin{equation}
 0\le n+2O_{q\ell}
 =\Big\|\sum_i z_q(X_i)\otimes z_\ell(X_i')\Big\|^2\le n^2.
 \label{eq:newton-nonnegative-overlap}
\end{equation}
Using $H=\sum_q(b_q^2/N_{q,d})U_q$ in~\eqref{eq:hc-15} gives
\begin{equation}
 Y_H=\sum_{q,\ell\in\mathcal I}
      \frac{b_q^2b_\ell^2}{N_{q,d}N_{\ell,d}}(n+2O_{q\ell}).
 \label{eq:mc-15}
\end{equation}
Pairs with $b_qb_\ell=0$ contribute zero, so fix an active pair.
Choose $D>2s$ and let
$\mathcal G=\mathcal G_D(X)\cap\mathcal G_D(X')$.
For fixed $0<t\le t_0$ and $x\ge0$,
\begin{align*}
 \E[\Gamma_d\1_{\mathcal G}\1_{\{O_{q\ell}>x\}}]
 &\overset{\mathrm{(a)}}{\le}e^{-t|b_qb_\ell|x}
       \E[\1_{\mathcal G}e^{S_\kappa+t|b_qb_\ell|O_{q\ell}}]\overset{\mathrm{(b)}}{\le}
       \exp\big(C(1+V_0)-t|b_qb_\ell|x\big),
\end{align*}
where $\mathrm{(a)}$ uses~\eqref{eq:finite-likelihood-overlap}
and exponential Markov; $\mathrm{(b)}$ uses
\eqref{eq:finite-mgf} in Lemma~\ref{lem:finite-mgf}.
Applying~\eqref{eq:4-6} in Lemma~\ref{lem:planting} gives
\begin{equation}
 \P_{\post}\Big[
 O_{q\ell}>\frac{C(1+V_0)+u}{|b_qb_\ell|},\quad\mathcal G\Big]
 \le Ce^{-cu},\quad u\ge0.
 \label{eq:newton-overlap-tail}
\end{equation}
In particular, with $(x)_+=\max\{x,0\}$,
\[
 \E_{\post}[(O_{q\ell})_+^s\1_{\mathcal G}]
 =s\int_0^\infty x^{s-1}\P_{\post}[O_{q\ell}>x,\mathcal G]\dd x
 \le C_s\Big(\frac{1+V_0}{|b_qb_\ell|}\Big)^s.
\]
Here we use the bound $1$ up to $C(1+V_0)/|b_qb_\ell|$
and~\eqref{eq:newton-overlap-tail} thereafter. Since
$0\le n+2O_{q\ell}\le n+2(O_{q\ell})_+$,
\begin{equation}
 \|(n+2O_{q\ell})\1_{\mathcal G}\|_{L^s(\P_{\post})}
 \le C_s\Big(n+\frac{1+V_0}{|b_qb_\ell|}\Big).
 \label{eq:newton-overlap-good}
\end{equation}
The prior marginal identity in Lemma~\ref{lem:planting} gives
$\P_{\post}[\mathcal G^c]\le2\P[\mathcal G_D^c]
\le2(n+N_{\max})^{-D}$. By~\eqref{eq:newton-nonnegative-overlap},
\begin{equation}
 \|(n+2O_{q\ell})\1_{\mathcal G^c}\|_{L^s(\P_{\post})}
 \le n^2\P_{\post}[\mathcal G^c]^{1/s}\le C_sn.
 \label{eq:newton-overlap-bad}
\end{equation}
Minkowski's inequality in~\eqref{eq:mc-15} now yields
\begin{align*}
 \|Y_H\|_{L^s(\P_{\post})}
 &\overset{\mathrm{(a)}}{\le}
 C_sn\Big(\sum_q\frac{b_q^2}{N_{q,d}}\Big)^2
 +C_s(1+V_0)\Big(\sum_q\frac{|b_q|}{N_{q,d}}\Big)^2\\
 &\overset{\mathrm{(b)}}{\le}
 C_{\ref{lem:finite-quadratic-overlap}}(s)
 \big(nd^{-2m}+(1+V_0)d^{-(m+q_m)}\big).
\end{align*}
Here $\mathrm{(a)}$ uses~\eqref{eq:newton-overlap-good}
and~\eqref{eq:newton-overlap-bad}; $\mathrm{(b)}$ uses
\eqref{eq:finite-bounds} and
$(\sum_q|b_q|/N_{q,d})^2\le
|\mathcal I|\sum_qb_q^2/N_{q,d}^2$.
\end{proof}

\section{Transfer to the original graph law}
\label{app:transfer}

Throughout this appendix, $P$ is the original graph law,
$n\le Cd^{m+q_m/3}$, and $\kappa_d=\Pi+R$, as in~\eqref{eq:taylor-polynomial}.
We first control $R$ in the likelihood bound, then in the
interaction kernels. The final calculation proves the uniform
moment estimate used in Proposition~\ref{prop:cavity-tail}.

\subsection{Likelihood comparison}
\label{app:proof-residual-bounds}
\label{app:central-event}

Recall $\mathcal E_\tau$, $\tau=\sqrt{C_0\log(n+d)/d}$, and
$S_\Pi=\sum_{i<j}\Pi(T_{ij})\Pi(T_{ij}')$ from
Section~\ref{sec:smooth-transfer}. The expectations below use two
independent prior samples. For $f\in\{\Pi,R\}$,
define
$f_\tau(t)=f(t)\1_{\{|t|\le\tau\}}
 -\E[f(T_d)\1_{\{|T_d|\le\tau\}}]$.
Let $Z_i=(X_i,X_i')$ and
\[
 \widetilde\Delta(Z_i,Z_j)
 =\Pi_\tau(T_{ij})R_\tau(T_{ij}')
  +R_\tau(T_{ij})\Pi_\tau(T_{ij}')
  +R_\tau(T_{ij})R_\tau(T_{ij}'),
 \quad U=\sum_{i<j}\widetilde\Delta(Z_i,Z_j).
\]
Each truncated kernel is centered, so $\widetilde\Delta$ is
canonical. We can therefore apply Lemma~\ref{lem:glz} to $U$.

\begin{lemma}[Truncation of the likelihood residual]
\label{lem:central-event}
For every fixed $D>0$, a sufficiently large $C_0$ gives
$\P[\mathcal E_\tau(X)^c]\le(n+d)^{-D}$.
Uniformly on $\mathcal E_\tau(X)\cap\mathcal E_\tau(X')$,
\begin{equation}
 \sum_{i<j}\kappa_d(T_{ij})\kappa_d(T_{ij}')=S_\Pi+U+o(1).
 \label{eq:central-residual}
\end{equation}
\end{lemma}

\begin{proof}
Lemma~\ref{lem:gegen-moment} at degree one gives
$\|T_d\|_p\lesssim\sqrt{p/d}$ for $2\le p\le cd$.
Choose $p=C\log(n+d)$ with $C$ sufficiently large and then $C_0$ sufficiently large.
Markov's inequality gives
$\P[|T_d|>\tau]\le(\|T_d\|_p/\tau)^p\le(n+d)^{-D-6}$.
A union bound then gives
$\P[\mathcal E_\tau(X)^c]\le\binom n2(n+d)^{-D-6}
\le(n+d)^{-D-4}$.
The kernels $\Pi$ and $R$ are bounded and centered, so
\begin{equation}
 \big|\E[\Pi(T_d)\1_{\{|T_d|>\tau\}}]\big|
 +\big|\E[R(T_d)\1_{\{|T_d|>\tau\}}]\big|
 \lesssim(n+d)^{-D-6}.
 \label{eq:central-centering-error}
\end{equation}
For either truncated kernel $f_\tau$,
$\int f_\tau(\langle x,y\rangle)\sigma_d(\dd y)=0$.
Thus each term in $\widetilde\Delta$ integrates to zero in
either pair variable. On $|t|\le\tau$, centering gives
$f_\tau(t)-f(t)=\E[f(T_d)\1_{\{|T_d|>\tau\}}]$.
Expanding $\kappa_d=\Pi+R$ on the two central events therefore gives
\[
 \Big|\sum_{i<j}\kappa_d(T_{ij})\kappa_d(T_{ij}')-S_\Pi-U\Big|
 \overset{\mathrm{(a)}}{\lesssim}n^2(n+d)^{-D-6}=o(1),
\]
where $\mathrm{(a)}$ uses~\eqref{eq:central-centering-error}
and boundedness of $\Pi$ and $R$. This proves~\eqref{eq:central-residual}.
\end{proof}

\phantomsection\label{app:proof-weighted-transfer}
We next establish the exponential moment used in the proof of
Proposition~\ref{prop:weighted-transfer}. The argument has two
steps: estimate the four parameters in Lemma~\ref{lem:glz},
then integrate its tail up to the deterministic bound on $U$.

\noindent\textbf{Step I.} For bounded $f\in\{\Pi,R\}$,
\[
 \|f-f_\tau\|_2+\|T_f-T_{f_\tau}\|_{\op}
 \le 2\|f\|_\infty\P[|T_d|>\tau]^{1/2}.
\]
For each fixed power, increasing $C_0$ makes this error smaller
than that power of $(n+d)^{-1}$. Lemmas~\ref{lem:taylor-finite}
and~\ref{lem:residual-bounds} therefore give
\begin{align*}
 \|\Pi_\tau\|_2&\lesssim d^{-m/2},
 \quad \|T_{\Pi_\tau}\|_{\op}\lesssim d^{-(m+q_m)/2},
 \quad \|\Pi_\tau\|_\infty\lesssim\tau^m,\\
 \|R_\tau\|_2&\lesssim d^{-(m+6)/2},
 \quad \|T_{R_\tau}\|_{\op}\lesssim d^{-(m+6)/2},
 \quad \|R_\tau\|_\infty\lesssim\tau^{m+6}.
\end{align*}
The two replicas are independent under the prior. Therefore
\[
 \|\widetilde\Delta\|_2^2\lesssim d^{-2m-6},
 \quad
 T_{\widetilde\Delta}
 =T_{\Pi_\tau}\otimes T_{R_\tau}
  +T_{R_\tau}\otimes T_{\Pi_\tau}
  +T_{R_\tau}\otimes T_{R_\tau}.
\]
Independent rotations in the two replicas make
$\E[\widetilde\Delta(Z_1,Z_2)^2\mid Z_1]$
constant. Write $A_\Delta,B_\Delta,C_\Delta,D_\Delta$ for the
parameters in Lemma~\ref{lem:glz} with kernel $\widetilde\Delta$
and measure $\sigma_d\otimes\sigma_d$. They satisfy
\begin{align}
 C_\Delta^2&\lesssim n^2d^{-2m-6}
                   \overset{\mathrm{(a)}}{\lesssim}d^{2q_m/3-6},
 \quad B_\Delta^2\lesssim nd^{-2m-6},\notag\\
 D_\Delta&\lesssim nd^{-m-3-q_m/2}
                   \overset{\mathrm{(a)}}{\lesssim}d^{-3-q_m/6},
 \quad A_\Delta\lesssim\tau^{2m+6}.
 \label{eq:tr-3}
\end{align}
Here the first bounds use the displayed $L^2$, operator, and
supremum norms; $\mathrm{(a)}$ uses $n\le Cd^{m+q_m/3}$.
If $A_\Delta=0$, then $U=0$. Otherwise, $|U|\le M$ for
$M=n^2A_\Delta$, and~\eqref{eq:t-8} in
Lemma~\ref{lem:glz} gives
\begin{equation}
 \P[|U|>u]\le C\exp\big(-c\Psi(u)\big),
 \quad
 \Psi(u)=\min\Big\{
 \frac{u^2}{C_\Delta^2},\frac{u}{D_\Delta},
 \frac{u^{2/3}}{B_\Delta^{2/3}},
 \frac{u^{1/2}}{A_\Delta^{1/2}}\Big\}.
 \label{eq:tr-4a}
\end{equation}
At $u=d^{-1}$, the four terms in~\eqref{eq:tr-4a}
have respective lower bounds
$d^{4-2q_m/3}$, $d^{2+q_m/6}$,
$d^{m/3+4/3-q_m/9}$, and $d^{(m+2)/2-o(1)}$
by~\eqref{eq:tr-3}. Hence
$\Psi(d^{-1})\ge d^{\zeta+o(1)}$ for any fixed
\begin{equation}
 \frac{2q_m}{3}<\zeta<
 \min\Big\{4-\frac{2q_m}{3},\quad
            2+\frac{q_m}{6},\quad
            \frac m3+\frac43-\frac{q_m}{9},\quad
            \frac{m+2}{2}\Big\}.
 \label{eq:tr-5}
\end{equation}
Such a $\zeta$ exists because $m\ge q_m$ and $q_m\le2$.

\noindent\textbf{Step II.} Fix $s>0$. If $M\le d^{-1}$,
then $1\le\E[e^{s|U|}]\le e^{s/d}=1+o(1)$.
Otherwise, for $d^{-1}\le u\le M$,
\begin{equation}
 \frac{\Psi(u)}u\ge
 \min\Big\{
 \frac{d^{-1}}{C_\Delta^2},\frac1{D_\Delta},
 (B_\Delta^2M)^{-1/3},(A_\Delta M)^{-1/2}\Big\}
 \overset{\mathrm{(a)}}{\ge}d^{3-q_m/3-o(1)}.
 \label{eq:tr-5a}
\end{equation}
Here $\mathrm{(a)}$ uses~\eqref{eq:tr-3}, $M=n^2A_\Delta$,
and $n\le Cd^{m+q_m/3}$. This gives
$c\Psi(u)-su\ge c\Psi(u)/2$ for all large $d$.
Since $M$ is polynomial in $d$ and $\Psi$ is increasing,
\begin{align*}
 \E[e^{s|U|}\1_{\{|U|>d^{-1}\}}]
 &=e^{s/d}\P[|U|>d^{-1}]
   +s\int_{d^{-1}}^M e^{su}\P[|U|>u]\dd u\\
 &\overset{\mathrm{(a)}}{\le}C(e^{s/d}+sM)
       \exp\big(-c\Psi(d^{-1})/2\big),
\end{align*}
where $\mathrm{(a)}$ uses~\eqref{eq:tr-4a},~\eqref{eq:tr-5a},
and monotonicity of $\Psi$.
After decreasing $\zeta>2q_m/3$ if necessary,
\begin{equation}
 \E[e^{s|U|}\1_{\{|U|>d^{-1}\}}]
 \le\exp\big(-d^{\zeta+o(1)}\big),
 \quad \E[e^{s|U|}]=1+o(1).
 \label{eq:tr-6}
\end{equation}
This supplies the exponential moment used
in~\eqref{eq:tr-7} and proves
Proposition~\ref{prop:weighted-transfer}.

For events $\mathcal G$ and $\mathcal B$ of the independent prior pair,
we also need the tail probability inside the likelihood comparison.
By~\eqref{eq:central-residual},
\begin{align}
\E[\1_{\mathcal G\cap\mathcal E_\tau(X)\cap
           \mathcal E_\tau(X')\cap\mathcal B}\Gamma_d]
 \overset{\mathrm{(a)}}{\le} e^{o(1)}
 \big(\E[\1_{\mathcal G}e^{3S_\Pi}]\big)^{1/3}
 \big(\E[e^{3|U|}]\big)^{1/3}\P[\mathcal B]^{1/3}.
 \label{eq:tr-7a}
\end{align}
Here $\mathrm{(a)}$ uses $\Gamma_d\le\exp(S_\Pi+U+o(1))$
and H\"older's inequality with three equal exponents.
For $\mathcal G\subseteq\mathcal G_D(X)\cap\mathcal G_D(X')$,
the first factor is at most $\exp(C(1+V_0))$ by
Lemma~\ref{lem:finite-mgf}; the second is $1+o(1)$
by~\eqref{eq:tr-6}.

\subsection{Conditional moments and completion}
\label{app:proof-star-transfer}

We prove Proposition~\ref{prop:star-transfer} by first applying
the three fixed order arguments to the polynomial observables,
then expanding their difference from the original observables.

\begin{proof}[Proof of Proposition~\ref{prop:star-transfer}]
Recall that $n=d^{m+q_m/3}/\omega_d$ with $\omega_d\to\infty$.
\noindent\textbf{Step I.} The coefficient bounds for $\Pi$
and the verification of~\eqref{eq:finite-field} are in
Section~\ref{sec:smooth-transfer}. For this expansion,
Lemma~\ref{lem:finite-mgf} gives the cost
\begin{equation}
 V_0\lesssim nd^{-m}\log(n+d)+n^2d^{-2m}
     \le d^{2q_m/3+o(1)}.
 \label{eq:8-3}
\end{equation}
The fixed powers in~\eqref{eq:tr-7} and~\eqref{eq:tr-7a}
preserve~\eqref{eq:finite-field}. We apply them with the
posterior under the original $P$:
\begin{itemize}
 \item $k=2$: take $F=t(b_{q,d}^\Pi)^2O_{qq}$ in~\eqref{eq:tr-7}.
 The overlap tail and diagonal summation in
 Proposition~\ref{prop:finite-order-two} give
 \[
 \binom n3\bar g_2^\Pi
 \lesssim n(1+V_0)d^{-(m+q_m)}
 \lesssim\frac{d^{-2q_m/3}}{\omega_d}
 +\frac{d^{-q_m/3}\log d}{\omega_d^2}
 +\frac1{\omega_d^3}=o(1).
 \]
 \item $k=3$: take the tail event in~\eqref{eq:hc-10a}
 in~\eqref{eq:tr-7a}. By~\eqref{eq:4-6} in
 Lemma~\ref{lem:planting}, its posterior probability on the
 retained events is at most
 $\exp(C(1+V_0)-d^{\zeta+o(1)})
 =\exp(-d^{\zeta+o(1)})$, since $\zeta>2q_m/3$.
 The integration in~\eqref{eq:hc-13} therefore gives
 $\binom n4\bar g_3^\Pi=o(1)$.
 \item $k\ge4$: take $F=t|b_{q,d}^\Pi b_{\ell,d}^\Pi|O_{q\ell}$
 in~\eqref{eq:tr-7}. The tail integration in
 Lemma~\ref{lem:finite-quadratic-overlap} gives~\eqref{eq:mc-17}
 with $H=T_\Pi^2$ and the original posterior.
 Substitution into~\eqref{eq:newton-partition}, as in the two
 cases of Proposition~\ref{prop:finite-higher}, gives
 \[
 \binom n{k+1}\bar g_k^\Pi\lesssim_k
 \begin{cases}
 d^{q_m-m}/\omega_d^{1+k/2},&n\ge d^{m-q_m},\\
 d^{(m-q_m)(1+k/2)-mk},&n<d^{m-q_m}.
 \end{cases}
 \]
 Both bounds tend to zero because $k\ge4$ and $m\ge q_m$.
\end{itemize}
In each case, Lemma~\ref{lem:planting} bounds the posterior
probability of excluded events by $O((n+d)^{-D})$.
The observables are polynomially bounded for fixed $k$,
so choosing $D$ after $k$ makes their contribution $o(1)$.
We have proved
\begin{equation}
 \binom n{k+1}\bar g_k^\Pi=o(1)
 \quad\text{for every fixed }k\ge2.
 \label{eq:8-6}
\end{equation}

\noindent\textbf{Step II.} We now estimate each term
containing $R$. For $J\subseteq[k]$, define
\[
 h_{k,d}^{J}(x_1,\ldots,x_k)
 =\int\prod_{i\in J}R(x_i,z)
       \prod_{i\notin J}\Pi(x_i,z)\sigma_d(\dd z).
\]
For $H_\Pi=T_\Pi^2$, Lemmas~\ref{lem:taylor-finite},
\ref{lem:spherical-hypercontractivity}, and~\ref{lem:residual-bounds}
give, for every fixed $s$,
\[
 \sup_x\|\Pi(x,\cdot)\|_{2s}^2\lesssim d^{-m},
 \quad \|H_\Pi\|_s\lesssim d^{-m-q_m/2},
 \quad \|T_R^2\|_\infty\lesssim d^{-m-6}.
\]
To bound a slice, square $h_{k,d}^J$ and integrate the variables
in $I\subseteq[k]$, with $I^c=[k]\setminus I$ and
$\sigma_d^I=\bigotimes_{i\in I}\sigma_d$. Each integrated factor $\Pi$ gives $H_\Pi$,
and each integrated factor $R$ gives $T_R^2$.
After taking absolute values, H\"older's inequality on $(z,z')$
with $k$ equal exponents gives
\begin{align*}
 \|h_{k,d}^J(x_I,x_{I^c})\|_{L^2(\sigma_d^I)}^2
 \overset{\mathrm{(a)}}{\le}&~\|T_R^2\|_k^{|I\cap J|}\|H_\Pi\|_k^{|I\setminus J|}
 \Big(\sup_x\|R(x,\cdot)\|_{2k}^2\Big)^{|J\setminus I|}
 \Big(\sup_x\|\Pi(x,\cdot)\|_{2k}^2\Big)^{k-|I\cup J|}\\
 \overset{\mathrm{(b)}}{\lesssim}&~
 d^{-mk-6|J|-q_m|I\setminus J|/2}.
\end{align*}
Here $\mathrm{(a)}$ integrates the roots independently.
For $\mathrm{(b)}$, use the displayed bounds and the uniform residual
moment bound in Lemma~\ref{lem:residual-bounds}. Consequently,
\begin{equation}
 \esssup_{x_{I^c}}\|h_{k,d}^{J}(x_I,x_{I^c})\|_{L^2(\sigma_d^I)}
 \lesssim d^{-mk/2-3|J|-q_m|I\setminus J|/4}.
 \label{eq:hc-22}
\end{equation}
For $J\ne\varnothing$, this is at most
$Cd^{-mk/2-3-(|I|-1)_+q_m/4}$, since $q_m/4<3$.

Since $h_{k,d}=\sum_{J\subseteq[k]}h_{k,d}^J$
and $h_{k,d}^\Pi=h_{k,d}^{\varnothing}$, the difference of the
replica products is a sum of terms containing $R$.
For $Z_i=(X_i,X_i')$, define
\begin{align*}
 D_k(Z_1,\ldots,Z_k)
 &=h_{k,d}(X_{[k]})h_{k,d}(X_{[k]}')
   -h_{k,d}^{\Pi}(X_{[k]})h_{k,d}^{\Pi}(X_{[k]}')\\
 &=\sum_{(J,J')\ne(\varnothing,\varnothing)}
       h_{k,d}^{J}(X_{[k]})h_{k,d}^{J'}(X_{[k]}').
\end{align*}
Let $W_k=\sum_{|I|=k}D_k(Z_I)$ and let $a_r$ be the slice norms of $D_k$ in
Lemma~\ref{lem:adamczak-cor}. The two replicas are independent
under the prior, so their slice norms multiply. Since at least
one of $J,J'$ is nonempty,~\eqref{eq:hc-22} gives
\[
 a_0\lesssim d^{-mk-3},\quad
 a_r\lesssim d^{-mk-3-(2r-1)q_m/4},\quad1\le r\le k.
\]
The sum defining $D_k$ is symmetric. Since $\Pi$ and $R$ are
centered, each factor $h_{k,d}^J$ integrates to zero in every
root. Independence of the two prior replicas therefore makes
$D_k$ canonical in each $Z_i$. Applying~\eqref{eq:t-7} in Lemma~\ref{lem:adamczak-cor},
with all coefficients equal
to one, gives
\begin{equation}
 \P[|W_k|>t]
 \le C\exp\Big(-c\min_{0\le r\le k}
       \Big(\frac{t}{n^{r/2}a_r}\Big)^{2/(2k-r)}\Big),
 \label{eq:hc-24}
\end{equation}
provided these ratios are at least one.

At $t=n^{-1}$ and $n\le d^{m+q_m/3}$, the powers of $d$
in~\eqref{eq:hc-24} are bounded below by
\begin{align}
 \alpha_0&=\frac{m(k-1)+3-q_m/3}{k},\notag\\
 \alpha_r&=
 \frac{2\big(m(k-1-r/2)+3+rq_m/3-7q_m/12\big)}{2k-r},
 \quad 1\le r\le k.
 \label{eq:hc-25b}
\end{align}
They all exceed $2q_m/3$. Indeed,
$k(\alpha_0-2q_m/3)\ge3-2q_m/3>0$ because $m\ge q_m$ and $k\ge2$.
For $r\ge1$, the sign of $\alpha_r-2q_m/3$ is the sign
of an affine function of $r$. Its values at $r=0$ and
$r=k$, respectively, satisfy
\begin{align*}
 m(k-1)+3-q_m(2k/3+7/12)&\ge3-11q_m/12>0,\\
 m(k/2-1)+3-7q_m/12&\ge3-7q_m/12>0.
\end{align*}
Choose $\delta>0$ so small that
$2\delta(m+q_m/3)/(2k-r)<\alpha_r-2q_m/3$ for every $r$.
Then the ratios required by~\eqref{eq:hc-24} diverge, and
for some $\zeta>2q_m/3$,
\begin{equation}
 \P[|W_k|>n^{-1-\delta}]
 \le\exp\big(-d^{\zeta+o(1)}\big).
 \label{eq:hc-26}
\end{equation}

\noindent\textbf{Step III.} Define
$\mathcal G_{\mathrm{tr}}=
 \mathcal G_D(X)\cap\mathcal G_D(X')\cap
 \mathcal E_\tau(X)\cap\mathcal E_\tau(X')$.
Apply~\eqref{eq:tr-7a} with
$\mathcal B=\{|W_k|>n^{-1-\delta}\}$.
We obtain
\begin{align*}
 \P_{\post}[\mathcal B\cap\mathcal G_{\mathrm{tr}}]
 &\overset{\mathrm{(a)}}{\le}
 \big(\E[\1_{\mathcal G_{\mathrm{tr}}\cap\mathcal B}\Gamma_d]\big)^{1/2}\overset{\mathrm{(b)}}{\le}
 \exp\big(O(1+V_0)-c d^{\zeta+o(1)}\big)
 =\exp\big(-d^{\zeta+o(1)}\big),
\end{align*}
where $\mathrm{(a)}$ is~\eqref{eq:4-6} in Lemma~\ref{lem:planting}
and $\mathrm{(b)}$ uses~\eqref{eq:tr-7a},~\eqref{eq:tr-6},
and~\eqref{eq:hc-26}. The last equality follows from~\eqref{eq:8-3}
and $\zeta>2q_m/3$.
The marginal identity in Lemma~\ref{lem:planting} gives
$\P_{\post}[\mathcal G_{\mathrm{tr}}^c]
\le2\P[\mathcal G_D^c]+2\P[\mathcal E_\tau^c]$.
Since $|W_k|\lesssim n^k$ deterministically, choosing $D$
after $k$ yields
\begin{equation}
 \E_{\post}[|W_k|]
 \le n^{-1-\delta}
     +Cn^k\P_{\post}[\mathcal G_{\mathrm{tr}}^c]
     +Cn^k\exp\big(-d^{\zeta+o(1)}\big)
 =o(n^{-1}).
 \label{eq:hc-27}
\end{equation}
The replica identity and exchangeability now give
$\binom nk|\bar g_k-\bar g_k^\Pi|
=|\E_{\post}[W_k]|\le\E_{\post}[|W_k|]=o(n^{-1})$.
Together with~\eqref{eq:8-6}, this implies
\begin{equation}
 \binom n{k+1}\bar g_k
 \le\binom n{k+1}\bar g_k^\Pi
   +\frac{n-k}{k+1}\binom nk|\bar g_k-\bar g_k^\Pi|
 =o(1).
 \label{eq:8-7}
\end{equation}
Together with the estimate for the difference, this proves
\eqref{eq:star-transfer-summary}.
\end{proof}

\phantomsection\label{app:proof-cavity-tail}
We finish by proving~\eqref{eq:growing-moment}, the prior moment
bound used in Proposition~\ref{prop:cavity-tail}.

\begin{proof}[Proof of~\eqref{eq:growing-moment}]
By Cauchy--Schwarz and rotation invariance,
\begin{equation}
 \|H_d\|_\infty
 \le\sup_{x,y}\|\kappa_d(x,\cdot)\|_2
                   \|\kappa_d(y,\cdot)\|_2
 =\|\kappa_d\|_2^2\overset{\mathrm{(a)}}{\lesssim}d^{-m}.
 \label{eq:tr-8b2}
\end{equation}
Here $\mathrm{(a)}$ follows from $\kappa_d=\Pi+R$ and
Lemmas~\ref{lem:taylor-finite} and~\ref{lem:residual-bounds}.
Recall $\Pi=\sum_qb_{q,d}^\Pi P_{q,d}$ from
Lemma~\ref{lem:taylor-finite}. Its contraction is
\[
 H_\Pi(t)=\sum_q\frac{(b_{q,d}^\Pi)^2}{N_{q,d}}P_{q,d}(t),
 \quad \sum_q\frac{(b_{q,d}^\Pi)^2}{N_{q,d}}=O(d^{-m}).
\]
For $2\le p\le c_{\ref{lem:gegen-moment}}(m+5)d$, all active
degrees lie in the moment range of Lemma~\ref{lem:gegen-moment}.
Since $p/d\le1$, every degree $q\ge q_m$ satisfies
$\|P_{q,d}(T_d)\|_p\lesssim(p/d)^{q/2}\le(p/d)^{q_m/2}$.
When $m$ is even, degree one can first occur at Taylor order
$m+1$. By~\eqref{eq:smooth-monomial-coefficients},
\[
 \frac{(b_{1,d}^\Pi)^2}{N_{1,d}}\|P_{1,d}(T_d)\|_p
 \lesssim d^{-m-1}\sqrt{p/d}\le d^{-m}(p/d).
\]
Minkowski's inequality therefore gives
$\|H_\Pi(T_d)\|_p\lesssim d^{-m}(p/d)^{q_m/2}$.
Since $T_{\kappa_d}^2-T_\Pi^2=T_\Pi T_R+T_RT_\Pi+T_R^2$,
\[
 \|H_d-H_\Pi\|_\infty
 \overset{\mathrm{(a)}}{\le}2\|\Pi\|_2\|R\|_2+\|R\|_2^2
 \overset{\mathrm{(b)}}{\lesssim}d^{-m-3},
\]
where $\mathrm{(a)}$ applies Cauchy--Schwarz to each kernel;
$\mathrm{(b)}$ uses Lemmas~\ref{lem:taylor-finite}
and~\ref{lem:residual-bounds}. This error is at most
$Cd^{-m}(p/d)^{q_m/2}$ because $p\ge2$ and $q_m\le2$.
Thus
\begin{equation}
 \|H_d(T_d)\|_p\lesssim d^{-m}(p/d)^{q_m/2},
 \quad2\le p\le c_{\ref{lem:gegen-moment}}(m+5)d,
 \label{eq:tr-8}
\end{equation}
which proves~\eqref{eq:growing-moment}.
\end{proof}

\end{document}